\documentclass[11pt]{amsart}

\usepackage{amsmath}
\usepackage{amssymb}
\usepackage{amscd}
\usepackage{color}
\usepackage{hyperref}
\usepackage{mathrsfs}

\usepackage{tikz}

\usepackage{xy}
\xyoption{all}

\newcommand{\nc}{\newcommand}

\nc{\CC}{{\mathbb{C}}}
\nc{\LL}{{\mathbb{L}}}
\nc{\RR}{{\mathbb{R}}}
\renewcommand{\P}{{\mathbb{P}}}
\nc{\OO}{{\mathbb{O}}}
\renewcommand{\SS}{{\mathbb{S}}}
\nc{\QQ}{{\mathbb{Q}}}
\nc{\ZZ}{{\mathbb{Z}}}

\nc{\rA}{{\mathrm{A}}}
\nc{\rB}{{\mathrm{B}}}
\nc{\rC}{{\mathrm{C}}}
\nc{\rD}{{\mathrm{D}}}
\nc{\rE}{{\mathrm{E}}}
\nc{\rF}{{\mathrm{F}}}
\nc{\rG}{{\mathrm{G}}}
\nc{\rH}{{\mathrm{H}}}
\nc{\rI}{{\mathrm{I}}}
\nc{\rJ}{{\mathrm{J}}}
\nc{\rK}{{\mathrm{K}}}
\nc{\rL}{{\mathrm{L}}}
\nc{\rM}{{\mathrm{M}}}
\nc{\rN}{{\mathrm{N}}}
\nc{\rO}{{\mathrm{O}}}
\nc{\rP}{{\mathrm{P}}}
\nc{\rQ}{{\mathrm{Q}}}
\nc{\rR}{{\mathrm{R}}}
\nc{\rS}{{\mathrm{S}}}
\nc{\rT}{{\mathrm{T}}}
\nc{\rU}{{\mathrm{U}}}
\nc{\rV}{{\mathrm{V}}}
\nc{\rW}{{\mathrm{W}}}
\nc{\rX}{{\mathrm{X}}}
\nc{\rY}{{\mathrm{Y}}}
\nc{\rZ}{{\mathrm{Z}}}

\nc{\cA}{{\mathscr{A}}}
\nc{\cB}{{\mathscr{B}}}
\nc{\cC}{{\mathscr{C}}}
\nc{\cD}{{\mathscr{D}}}
\nc{\cE}{{\mathscr{E}}}
\nc{\cF}{{\mathscr{F}}}
\nc{\cG}{{\mathscr{G}}}
\nc{\cH}{{\mathscr{H}}}
\nc{\cI}{{\mathscr{I}}}
\nc{\cJ}{{\mathscr{J}}}
\nc{\cK}{{\mathcal{K}}}
\nc{\cL}{{\mathscr{L}}}
\nc{\cM}{{\mathscr{M}}}
\nc{\cN}{{\mathscr{N}}}
\nc{\cO}{{\mathscr{O}}}
\nc{\cP}{{\mathscr{P}}}
\nc{\cQ}{{\mathcal{Q}}}
\nc{\cR}{{\mathscr{R}}}
\nc{\cS}{{\mathcal{S}}}
\nc{\cT}{{\mathscr{T}}}
\nc{\cU}{{\mathscr{U}}}
\nc{\cV}{{\mathscr{V}}}
\nc{\cW}{{\mathscr{W}}}
\nc{\cX}{{\mathscr{X}}}
\nc{\cY}{{\mathscr{Y}}}
\nc{\cZ}{{\mathscr{Z}}}

\nc{\bA}{{\mathbf{A}}}
\nc{\bB}{{\mathbf{B}}}
\nc{\bC}{{\mathbf{C}}}
\nc{\bD}{{\mathbf{D}}}
\nc{\bE}{{\mathbf{E}}}
\nc{\bF}{{\mathbf{F}}}
\nc{\bG}{{\mathbf{G}}}
\nc{\bH}{{\mathbf{H}}}
\nc{\bI}{{\mathbf{I}}}
\nc{\bJ}{{\mathbf{J}}}
\nc{\bK}{{\mathbf{K}}}
\nc{\bL}{{\mathbf{L}}}
\nc{\bM}{{\mathbf{M}}}
\nc{\bN}{{\mathbf{N}}}
\nc{\bO}{{\mathbf{O}}}
\nc{\bP}{{\mathbf{P}}}
\nc{\bQ}{{\mathbf{Q}}}
\nc{\bR}{{\mathbf{R}}}
\nc{\bS}{{\mathbf{S}}}
\nc{\bT}{{\mathbf{T}}}
\nc{\bU}{{\mathbf{U}}}
\nc{\bV}{{\mathbf{V}}}
\nc{\bW}{{\mathbf{W}}}
\nc{\bX}{{\mathbf{X}}}
\nc{\bY}{{\mathbf{Y}}}
\nc{\bZ}{{\mathbf{Z}}}

\nc{\ba}{{\mathbf{a}}}
\nc{\bb}{{\mathbf{b}}}
\nc{\bc}{{\mathbf{c}}}
\nc{\bd}{{\mathbf{d}}}
\nc{\be}{{\mathbf{e}}}
\nc{\bg}{{\mathbf{g}}}
\nc{\bh}{{\mathbf{h}}}
\nc{\bi}{{\mathbf{i}}}
\nc{\bj}{{\mathbf{j}}}
\nc{\bk}{{\mathbf{k}}}
\nc{\bl}{{\mathbf{l}}}
\nc{\bm}{{\mathbf{m}}}
\nc{\bn}{{\mathbf{n}}}
\nc{\bo}{{\mathbf{o}}}
\nc{\bp}{{\mathbf{p}}}
\nc{\bq}{{\mathbf{q}}}
\nc{\br}{{\mathbf{r}}}
\nc{\bs}{{\mathbf{s}}}
\nc{\bt}{{\mathbf{t}}}
\nc{\bu}{{\mathbf{u}}}
\nc{\bv}{{\mathbf{v}}}
\nc{\bw}{{\mathbf{w}}}
\nc{\bx}{{\mathbf{x}}}
\nc{\by}{{\mathbf{y}}}
\nc{\bz}{{\mathbf{z}}}

\nc{\fA}{{\mathfrak{A}}}
\nc{\fB}{{\mathfrak{B}}}
\nc{\fC}{{\mathfrak{C}}}
\nc{\fD}{{\mathfrak{D}}}
\nc{\fE}{{\mathfrak{E}}}
\nc{\fF}{{\mathfrak{F}}}
\nc{\fG}{{\mathfrak{G}}}
\nc{\fH}{{\mathfrak{H}}}
\nc{\fI}{{\mathfrak{I}}}
\nc{\fJ}{{\mathfrak{J}}}
\nc{\fK}{{\mathfrak{K}}}
\nc{\fL}{{\mathfrak{L}}}
\nc{\fM}{{\mathfrak{M}}}
\nc{\fN}{{\mathfrak{N}}}
\nc{\fO}{{\mathfrak{O}}}
\nc{\fP}{{\mathfrak{P}}}
\nc{\fQ}{{\mathfrak{Q}}}
\nc{\fR}{{\mathfrak{R}}}
\nc{\fS}{{\mathfrak{S}}}
\nc{\fT}{{\mathfrak{T}}}
\nc{\fU}{{\mathfrak{U}}}
\nc{\fV}{{\mathfrak{V}}}
\nc{\fW}{{\mathfrak{W}}}
\nc{\fX}{{\mathfrak{X}}}
\nc{\fY}{{\mathfrak{Y}}}
\nc{\fZ}{{\mathfrak{Z}}}

\nc{\fa}{{\mathfrak{a}}}
\nc{\fb}{{\mathfrak{b}}}
\nc{\fc}{{\mathfrak{c}}}
\nc{\fd}{{\mathfrak{d}}}
\nc{\fe}{{\mathfrak{e}}}
\nc{\ff}{{\mathfrak{f}}}
\nc{\fg}{{\mathfrak{g}}}
\nc{\fh}{{\mathfrak{h}}}
\nc{\fj}{{\mathfrak{j}}}
\nc{\fk}{{\mathfrak{k}}}
\nc{\fl}{{\mathfrak{l}}}
\nc{\fm}{{\mathfrak{m}}}
\nc{\fn}{{\mathfrak{n}}}
\nc{\fo}{{\mathfrak{o}}}
\nc{\fp}{{\mathfrak{p}}}
\nc{\fq}{{\mathfrak{q}}}
\nc{\fr}{{\mathfrak{r}}}
\nc{\fs}{{\mathfrak{s}}}
\nc{\ft}{{\mathfrak{t}}}
\nc{\fu}{{\mathfrak{u}}}
\nc{\fv}{{\mathfrak{v}}}
\nc{\fw}{{\mathfrak{w}}}
\nc{\fx}{{\mathfrak{x}}}
\nc{\fy}{{\mathfrak{y}}}
\nc{\fz}{{\mathfrak{z}}}

\nc{\sA}{{\mathsf{A}}}
\nc{\sB}{{\mathsf{B}}}
\nc{\sC}{{\mathsf{C}}}
\nc{\sD}{{\mathsf{D}}}
\nc{\sE}{{\mathsf{E}}}
\nc{\sF}{{\mathsf{F}}}
\nc{\sG}{{\mathsf{G}}}
\nc{\sH}{{\mathsf{H}}}
\nc{\sI}{{\mathsf{I}}}
\nc{\sJ}{{\mathsf{J}}}
\nc{\sK}{{\mathsf{K}}}
\nc{\sL}{{\mathsf{L}}}
\nc{\sM}{{\mathsf{M}}}
\nc{\sN}{{\mathsf{N}}}
\nc{\sO}{{\mathsf{O}}}
\nc{\sP}{{\mathsf{P}}}
\nc{\sQ}{{\mathsf{Q}}}
\nc{\sR}{{\mathsf{R}}}
\nc{\sS}{{\mathsf{S}}}
\nc{\sT}{{\mathsf{T}}}
\nc{\sU}{{\mathsf{U}}}
\nc{\sV}{{\mathsf{V}}}
\nc{\sW}{{\mathsf{W}}}
\nc{\sX}{{\mathsf{X}}}
\nc{\sY}{{\mathsf{Y}}}
\nc{\sZ}{{\mathsf{Z}}}

\nc{\sa}{{\mathsf{a}}}
\nc{\sd}{{\mathsf{d}}}
\nc{\se}{{\mathsf{e}}}
\nc{\sg}{{\mathsf{g}}}
\nc{\sh}{{\mathsf{h}}}
\nc{\si}{{\mathsf{i}}}
\nc{\sj}{{\mathsf{j}}}
\nc{\sk}{{\mathsf{k}}}
\nc{\sm}{{\mathsf{m}}}
\nc{\sn}{{\mathsf{n}}}
\nc{\so}{{\mathsf{o}}}
\nc{\sq}{{\mathsf{q}}}
\nc{\sr}{{\mathsf{r}}}
\nc{\st}{{\mathsf{t}}}
\nc{\su}{{\mathsf{u}}}
\nc{\sv}{{\mathsf{v}}}
\nc{\sw}{{\mathsf{w}}}
\nc{\sx}{{\mathsf{x}}}
\nc{\sy}{{\mathsf{y}}}
\nc{\sz}{{\mathsf{z}}}

\nc{\oA}{{\overline{A}}}
\nc{\oB}{{\overline{B}}}
\nc{\oC}{{\overline{C}}}
\nc{\oD}{{\overline{D}}}
\nc{\oE}{{\overline{E}}}
\nc{\oF}{{\overline{F}}}
\nc{\oG}{{\overline{G}}}
\nc{\oH}{{\overline{H}}}
\nc{\oI}{{\overline{I}}}
\nc{\oJ}{{\overline{J}}}
\nc{\oK}{{\overline{K}}}
\nc{\oL}{{\overline{L}}}
\nc{\oM}{{\overline{M}}}
\nc{\oN}{{\overline{N}}}
\nc{\oO}{{\overline{O}}}
\nc{\oP}{{\overline{P}}}
\nc{\oQ}{{\overline{Q}}}
\nc{\oR}{{\overline{R}}}
\nc{\oS}{{\overline{S}}}
\nc{\oT}{{\overline{T}}}
\nc{\oU}{{\overline{U}}}
\nc{\oV}{{\overline{V}}}
\nc{\oW}{{\overline{W}}}
\nc{\oX}{{\overline{X}}}
\nc{\oY}{{\overline{Y}}}
\nc{\oZ}{{\overline{Z}}}

\nc{\oa}{{\overline{a}}}
\nc{\ob}{{\overline{b}}}
\nc{\oc}{{\overline{c}}}
\nc{\od}{{\overline{d}}}
\nc{\of}{{\overline{f}}}
\nc{\og}{{\overline{g}}}
\nc{\oh}{{\overline{h}}}
\nc{\oi}{{\overline{i}}}
\nc{\oj}{{\overline{j}}}
\nc{\ok}{{\overline{k}}}
\nc{\ol}{{\overline{l}}}
\nc{\om}{{\overline{m}}}
\nc{\on}{{\overline{n}}}
\nc{\oo}{{\overline{o}}}
\nc{\op}{{\overline{p}}}
\nc{\oq}{{\overline{q}}}
\nc{\os}{{\overline{s}}}
\nc{\ot}{{\overline{t}}}
\nc{\ou}{{\overline{u}}}
\nc{\ov}{{\overline{v}}}
\nc{\ow}{{\overline{w}}}
\nc{\ox}{{\overline{x}}}
\nc{\oy}{{\overline{y}}}
\nc{\oz}{{\overline{z}}}

\nc{\tA}{{\tilde{A}}}
\nc{\tB}{{\tilde{B}}}
\nc{\tC}{{\tilde{C}}}
\nc{\tD}{{\tilde{D}}}
\nc{\tE}{{\tilde{E}}}
\nc{\tF}{{\tilde{F}}}
\nc{\tG}{{\tilde{G}}}
\nc{\tH}{{\tilde{H}}}
\nc{\tI}{{\tilde{I}}}
\nc{\tJ}{{\tilde{J}}}
\nc{\tK}{{\tilde{K}}}
\nc{\tL}{{\tilde{L}}}
\nc{\tM}{{\tilde{M}}}
\nc{\tN}{{\tilde{N}}}
\nc{\tO}{{\tilde{O}}}
\nc{\tP}{{\tilde{P}}}
\nc{\tQ}{{\tilde{Q}}}
\nc{\tR}{{\tilde{R}}}
\nc{\tS}{{\tilde{S}}}
\nc{\tT}{{\tilde{T}}}
\nc{\tU}{{\tilde{U}}}
\nc{\tV}{{\tilde{V}}}
\nc{\tW}{{\tilde{W}}}
\nc{\tX}{{\tilde{X}}}
\nc{\tY}{{\tilde{Y}}}
\nc{\tZ}{{\tilde{Z}}}

\nc{\tcF}{{\widetilde{\cF}}}
\nc{\tcG}{{\widetilde{\cG}}}
\nc{\tcQ}{{\widetilde{\cQ}}}

\nc{\bcF}{{\overline{\cF}}}
\nc{\bcS}{{\overline{\cS}}}

\nc{\ta}{{\tilde{a}}}
\nc{\tb}{{\tilde{b}}}
\nc{\tc}{{\tilde{c}}}
\nc{\td}{{\tilde{d}}}
\nc{\te}{{\tilde{e}}}
\nc{\tf}{{\tilde{f}}}
\nc{\tg}{{\tilde{g}}}
\nc{\ti}{{\tilde{i}}}
\nc{\tj}{{\tilde{j}}}
\nc{\tk}{{\tilde{k}}}
\nc{\tl}{{\tilde{l}}}
\nc{\tm}{{\tilde{m}}}
\nc{\tn}{{\tilde{n}}}
\nc{\tp}{{\tilde{p}}}
\nc{\tq}{{\tilde{q}}}
\nc{\tr}{{\tilde{r}}}
\nc{\ts}{{\tilde{s}}}
\nc{\tu}{{\tilde{u}}}
\nc{\tv}{{\tilde{v}}}
\nc{\tw}{{\tilde{w}}}
\nc{\tx}{{\tilde{x}}}
\nc{\ty}{{\tilde{y}}}
\nc{\tz}{{\tilde{z}}}

\nc{\hA}{{\hat{A}}}
\nc{\hB}{{\hat{B}}}
\nc{\hC}{{\hat{C}}}
\nc{\hD}{{\hat{D}}}
\nc{\hE}{{\hat{E}}}
\nc{\hF}{{\hat{F}}}
\nc{\hG}{{\hat{G}}}
\nc{\hH}{{\hat{H}}}
\nc{\hI}{{\hat{I}}}
\nc{\hJ}{{\hat{J}}}
\nc{\hK}{{\hat{K}}}
\nc{\hL}{{\hat{L}}}
\nc{\hM}{{\hat{M}}}
\nc{\hN}{{\hat{N}}}
\nc{\hO}{{\hat{O}}}
\nc{\hP}{{\hat{P}}}
\nc{\hQ}{{\hat{Q}}}
\nc{\hR}{{\hat{R}}}
\nc{\hS}{{\hat{S}}}
\nc{\hT}{{\hat{T}}}
\nc{\hU}{{\hat{U}}}
\nc{\hV}{{\hat{V}}}
\nc{\hW}{{\hat{W}}}
\nc{\hX}{{\hat{X}}}
\nc{\hY}{{\hat{Y}}}
\nc{\hZ}{{\hat{Z}}}

\nc{\ha}{{\hat{a}}}
\nc{\hb}{{\hat{b}}}
\nc{\hc}{{\hat{c}}}
\nc{\hd}{{\hat{d}}}
\nc{\he}{{\hat{e}}}
\nc{\hf}{{\hat{f}}}
\nc{\hg}{{\hat{g}}}
\nc{\hh}{{\hat{h}}}
\nc{\hi}{{\hat{i}}}
\nc{\hj}{{\hat{j}}}
\nc{\hk}{{\hat{k}}}
\nc{\hl}{{\hat{l}}}
\nc{\hm}{{\hat{m}}}
\nc{\hn}{{\hat{n}}}
\nc{\ho}{{\hat{o}}}
\nc{\hp}{{\hat{p}}}
\nc{\hq}{{\hat{q}}}
\nc{\hr}{{\hat{r}}}
\nc{\hs}{{\hat{s}}}
\nc{\hu}{{\hat{u}}}
\nc{\hv}{{\hat{v}}}
\nc{\hw}{{\hat{w}}}
\nc{\hx}{{\hat{x}}}
\nc{\hy}{{\hat{y}}}
\nc{\hz}{{\hat{z}}}

\nc{\eps}{\varepsilon}
\nc{\lan}{\big\langle}
\nc{\ran}{\big\rangle}
\nc{\kk}{{\mathsf{k}}}

\def\bw#1#2{\textstyle{\bigwedge\hskip-0.9mm^{#1}}\hskip0.2mm{#2}}
\DeclareMathOperator{\Sym}{{\mathrm{Sym}}}

\DeclareMathOperator{\Aut}{\mathrm{Aut}}
\DeclareMathOperator{\Hom}{\mathrm{Hom}}
\DeclareMathOperator{\Ext}{\mathrm{Ext}}

\DeclareMathOperator{\Hilb}{\mathrm{Hilb}}
\DeclareMathOperator{\Spec}{\mathrm{Spec}}
\DeclareMathOperator{\Proj}{\mathrm{Proj}}

\DeclareMathOperator{\Bl}{\mathrm{Bl}}
\DeclareMathOperator{\Sing}{\mathrm{Sing}}
\DeclareMathOperator{\Pic}{\mathrm{Pic}}

\DeclareMathOperator{\Ker}{\mathrm{Ker}}

\DeclareMathOperator{\Ima}{\mathrm{Im}}
\DeclareMathOperator{\Cone}{\mathrm{Cone}}

\DeclareMathOperator{\Gr}{\mathrm{Gr}}
\DeclareMathOperator{\OGr}{\mathrm{OGr}}

\DeclareMathOperator{\LGr}{\mathrm{LGr}}

\DeclareMathOperator{\Fl}{\mathrm{Fl}}
\DeclareMathOperator{\OFl}{\mathrm{OFl}}

\DeclareMathOperator{\GL}{\mathrm{GL}}

\DeclareMathOperator{\SO}{\mathrm{SO}}
\DeclareMathOperator{\PSO}{\mathrm{PSO}}
\DeclareMathOperator{\GO}{\mathrm{O}}
\DeclareMathOperator{\Spin}{\mathrm{Spin}}

\DeclareMathOperator{\codim}{\mathrm{codim}}

\DeclareMathOperator{\rc}{\mathrm{c}}

\DeclareMathOperator{\CH}{\mathrm{CH}}

\DeclareMathOperator{\pr}{\mathrm{pr}}
\DeclareMathOperator{\rk}{\mathrm{rk}}

\makeatletter\@addtoreset{equation}{section} \makeatother

\theoremstyle{plain}

\newtheorem{theorem}[equation]{Theorem}

\newtheorem{lemma}[equation]{Lemma}
\newtheorem{proposition}[equation]{Proposition}
\newtheorem{corollary}[equation]{Corollary}

\theoremstyle{definition}

\newtheorem{definition}[equation]{Definition}

\theoremstyle{remark}

\newtheorem{remark}[equation]{Remark}

\title{On linear sections of the spinor tenfold, I}
\author{Alexander Kuznetsov}
\address{{\sloppy
\parbox{0.9\textwidth}{
Algebraic Geometry Section, Steklov Mathematical Institute of Russian Academy of Sciences,\\
8 Gubkin str., Moscow 119991 Russia
\\[5pt]
The Poncelet Laboratory, Independent University of Moscow
\hfill\\[5pt]
Laboratory of Algebraic Geometry, National Research University Higher School of Economics, Russian Federation
}\bigskip}}
\email{akuznet@mi.ras.ru}
\date{}
\thanks{This work is supported by the Russian Science Foundation under grant 14-50-00005.}

\begin{document}

\begin{abstract}
We discuss the geometry of transverse linear sections of the spinor tenfold $X$, the connected component 
of the orthogonal Grassmannian of 5-dimensional isotropic subspaces in a 10-dimensional vector space 
equipped with a non-degenerate quadratic form.
In particular, we show that as soon as the dimension of a linear section of $X$ is at least 5, 
its integral Chow motive is of Lefschetz type.
We discuss classification of smooth linear sections of $X$ of small codimension;
in particular we check that there is a unique isomorphism class of smooth hyperplane sections
and exactly two isomorphism classes of smooth linear sections of codimension~2.
Using this, we define a natural quadratic line complex associated with a linear section of~$X$.
We also discuss the Hilbert schemes of linear spaces and quadrics on $X$ and its linear sections.
\end{abstract}

\maketitle


\section{Introduction}

\subsection{Overview}

The {\sf spinor tenfold}
\begin{equation*}
X = \Spin(10)/\bP_5 \subset \P^{15}
\end{equation*}
is one of the most interesting rational homogeneous spaces. 
Here $\Spin(10)$ is the simply connected covering of the special orthogonal group $\SO(10)$
and $\bP_5$ is its parabolic subgroup associated with the last vertex of the Dynkin diagram $\rD_5$
(the corresponding vertex on the picture below is black):
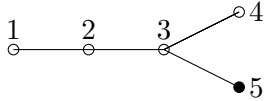
\begin{figure}[h]
\begin{tikzpicture}
\draw (0,.5) node [above] {1} -- (1,.5) node [above] {2} -- (2,.5) node [above] {3} -- (3,1) node [right] {4} -- (2,.5) -- (3,0) node [right] {5};
\draw (0,.5) circle 		[radius=.07];
\draw (1,.5) circle 		[radius=.07];
\draw (2,.5) circle 		[radius=.07];
\draw (3,1) circle 		[radius=.07];
\draw[fill] (3,0) circle 	[radius=.07];
\end{tikzpicture}
\end{figure}

The spinor tenfold is classically represented as a connected component
\begin{equation*}
X \cong \OGr_+(5,\rV)
\end{equation*}
of the isotropic Grassmannian $\OGr(5,\rV)$ for a non-degenerate quadratic form on a 10-dimensional vector space~$\rV$.
Note however, that the Pl\"ucker embedding of $\OGr_+(5,\rV) \subset \Gr(5,\rV) \subset \P(\bw5\rV)$ 
corresponds to the square of the generator of the Picard group $\Pic(X)$.

One of the most interesting features of the spinor tenfold $X$ is its \emph{projective self-duality} ---
the projective dual variety $X^\vee \subset \check\P^{15}$ of $X$ is projectively isomorphic to $X$
(here $\check\P^{15}$ is the dual projective space of~$\P^{15}$, and these two are the projectivizations
of the two half-spinor representations $\SS$ and $\SS^\vee$ of $\Spin(\rV)$).
More canonically, 
\begin{equation*}
X^\vee \cong \Spin(\rV)/\bP_4 \cong \OGr_-(5,\rV)
\end{equation*}
(so it is obtained from $X$ by an outer automorphism of $\Spin(\rV)$ corresponding to the involution of the Dynkin diagram $\rD_5$, 
and can be also described as the other connected component of the isotropic Grassmannian).
The self-duality property, actually, is very special --- among smooth projective varieties besides the spinor tenfold only 
quadrics $Q^n$, Segre varieties $\P^1 \times \P^n$, and the Grassmannian $\Gr(2,5)$ are self-dual.

The projective self-duality of the spinor tenfold lifts to the higher homological level.
In fact, it is also \emph{homologically projectively self-dual} (see~\cite[Section~6.2]{k2006hyperplane}, \cite[Theorem~5.5]{icm2014}).
This means that there is a nice relation (see Theorem~\ref{theorem:db-xk}) 
between derived categories of coherent sheaves of linear sections of $X$ and $X^\vee$.

The goal of this paper, first in a series, is to start a systematic study of the geometry of these linear sections 
\begin{equation*}
X_K = X \cap \P(K^\perp) \subset \P(\SS) = \P^{15}.
\end{equation*}
Here $K \subset \SS^\vee$ is a linear subspace and $K^\perp \subset \SS$ is the orthogonal complement of $K$.
We are mostly interested in linear sections which are smooth and dimensionally transverse of codimension at most~5.
We describe the Chow motive of all these varieties, provide a classification in codimension~1 and codimension~2 cases,
discuss the most special sections of codimension~3, and introduce an important ``quadratic invariant'' $R_K$ of $X_K$
which will play an important role in subsequent papers.

Actually, significant part of the results of this paper are known to experts, 
but the references are scattered (and some of these results are folklore) and use different approaches.
For instance see~\cite{zak1993book,Ranestad2000,pasquier2009,manivel2017double,fu2015special}.
So, we provided proofs for these results trying to keep the paper self-contained.

\subsection{Smooth, complete and non-isotrivial families}

Before explaining in more detail the content of the paper, let us mention one interesting property of varieties $X_K$.
By classical projective duality, in the case~$k = \dim K \le 5$ the linear section $X_K$ is smooth and dimensionally transverse
if and only if the corresponding linear subspace $\P(K) \subset \P(\SS^\vee) = \check\P^{15}$ satisfies the property
\begin{equation*}
X^\vee \cap \P(K) = \varnothing.
\end{equation*}
It follows that all intermediate linear sections $X_K \subset X_{K'} \subset X$ 
(sometimes we will call them {\sf over-sections} of $X_K$) are also smooth and dimensionally transverse.
Moreover, the simple smoothness criterion above has the following striking consequence.

Assume that $B$ is a smooth projective variety and $\phi \colon B \to \Gr(k,\SS^\vee)$ is a map such that 
\begin{equation*}
X^\vee \cap \P(K_b) = \varnothing
\end{equation*}
for each point $b \in B$, where $K_b = \phi(b)$ is the $k$-dimensional subspace of $\SS^\vee$ 
associated with the point~$b \in B$ by the map~$\phi$.
Since $\codim_{\P(\SS^\vee)} X^\vee = 5$, in the case $\dim B + k - 1 < 5$ the above assumption is easy to satisfy
(in this case this is a generality assumption).
Then each linear section $X_{K_b} \subset X$ is smooth of codimension $k$ in $X$.
Consider the total family of these sections
\begin{equation*}
\cX_B := X \times_{\P(\SS)} \P_B(\phi^*(\cK^\perp)),
\end{equation*}
where $\cK$ is the tautological subbundle of $\Gr(k,\SS^\vee)$ and $\cK^\perp$ is the subbundle of its orthogonal complements.
Then it follows that the morphism $\cX_B \to B$ is smooth, hence $\cX_B \to B$ is a \emph{complete family} of smooth projective varieties.
It is also not hard to choose $k$, $B$, and $\phi$ in such a way, that this family is not \emph{isotrivial}
(for this one should assume $k \ge 2$).
This gives one of not so many known examples of a complete non-isotrivial family of smooth varieties.

\subsection{Results}

The homological projective duality implies that every smooth linear section $X_K$ of $X$ of codimension~$k \le 5$ 
comes with a full exceptional collection of vector bundles of length $2\dim(X) - 4$.
The existence of a full exceptional collection implies easily that the rational Chow motive of $X_K$ is of Lefschetz type.
It is, however, not known whether the existence of a full exceptional collection implies that 
the Chow motive with \emph{integral} coefficients is of Lefschetz type 
(see, however, \cite{gorchinskiy2017} for some results in the 3-dimensional case).

The first main result of the paper (Theorem~\ref{theorem:motives-z-xk-06}) is a proof of this fact by a geometrical construction. 
Actually, we show (Proposition~\ref{proposition:blowup-xk}) that the blowup of the projective space $\P(K^\perp)$ with center in $X_K$ is 
a Zariski piecewise trivial fibration in projective spaces over the 8-dimensional quadric~$\rQ = \Spin(\rV)/\bP_1$.
Modifying this fibration (Proposition~\ref{proposition:double-blowup-xk}) to a projective bundle over a blowup of $\rQ$, 
we deduce that the motive of~$X_K$ is a direct summand of a sum of Lefschetz motives, 
which immediately implies that it is a direct sum of Lefschetz motives as well.
Another geometric argument (Corollary~\ref{corollary:xk-rational}) proves that every $X_K$ is rational
(it is expected, but not proven yet, that any smooth projective variety with a full exceptional collection is rational).

It may seem from the above that all smooth linear sections of $X$ are uniform and boring.
In the rest of the paper we show that this is far from being true, by exhibiting rich and interesting geometry 
associated with them.
Even more of this will come in subsequent papers.

First, we discuss hyperplane sections of $X$.
As it is well known, there are only two projective isomorphism classes of hyperplane sections --- smooth and singular.
We reprove this and provide a convenient geometric description of the hyperplane section in both cases.

For the singular hyperplane section $X'_1$ the description is the following.
We check that the singular locus of $X'_1$ is a 4-space $\P^4 \subset X$
and prove that the blowup of $X'_1$ along this 4-space is an explicit $\P^3$-bundle over~$\Gr(2,5)$, 
see Corollary~\ref{corollary:singular-hyperplane-sections} for details.
In fact, we deduce this isomorphism from a more general result (Proposition~\ref{proposition:blowup-x-pi4}) --- 
an identification of the blowup of $X$ with center in a 4-space with the blowup of $\P^{10}$ 
with center in the Grassmannian~$\Gr(2,5)$ (contained in a hyperplane $\P^9 \subset \P^{10}$).

Similarly, for the smooth hyperplane section $X''_1$ of $X$ we show that 
there is a unique 6-dimensional quadric $Q^6$ contained in $X''_1$ (Lemma~\ref{lemma:g6-xkappa-smooth}), 
and that the blowup of $X''_1$ along this quadric is isomorphic to a $\P^4$-bundle 
over a 5-dimensional quadric (Corollary~\ref{corollary:blowup-xkappa-smooth}, cf.~\cite[Lemma~1.17]{pasquier2009}).
Again, we deduce this from a more general result (Proposition~\ref{proposition:blowup-x-ogr48}) --- 
an identification of the blowup of $X$ with center in~$Q^6$ with a~$\P^4$-bundle over another 6-dimensional quadric.

Next we consider smooth linear sections $X_K \subset X$ of codimension~2. 
We show that there are exactly two isomorphism classes of those, that can be distinguished 
by looking at the Hilbert schemes $F_p(X_K)$ of linear subspaces of dimension $p$ on them.
First of all, for $X_K$ of the first type one has $F_4(X_K) = \varnothing$, 
while for the second type $F_4(X_K) \cong \P^1$ (Proposition~\ref{proposition:x2-with-p4}).
Second, for $X_K$ of the first type $F_1(X_K)$ is smooth, while for the second type $F_1(X_K)$ has a unique singular point (Corollary~\ref{corollary:f1-x1}).

We say that~$X_K$ of the second type is {\sf special} 
and the line on $X_K$ corresponding to the singular point of $F_1(X_K)$ is the {\sf special line} of $X_K$.
Geometrically, a special $X_K$ can be obtained by blowing up a quintic del Pezzo fourfold inside $\P^8$ 
(it is contained in a hyperplane $\P^7 \subset \P^8$) and then contracting the strict transform of the hyperplane.

We study the subvariety $R_0$ of the Grassmannian $\Gr(2,\SS^\vee)$ of lines in $\P(\SS^\vee)$ 
parameterizing special linear sections of $X$ and its closure $R = \overline{R_0} \subset \Gr(2,\SS^\vee)$.
We show (Lemma~\ref{lemma:secondary-quadric}) that $R$ is a \emph{quadratic line complex}, i.e., 
it is a hypersurface cut out on $\Gr(2,\SS^\vee)$ by a quadric in the corresponding Pl\"ucker space,
and we call it {\sf the spinor quadratic line complex}.

We show (Corollary~\ref{corollary:sing-r}) that the singular locus of $R$ is the variety of secant lines to $X^\vee$ 
(in particular, it follows that its codimension in $R$ is equal to~7)
and construct (Lemma~\ref{lemma:resolution-r}) a nice resolution of singularities $\tR \to R$ 
with $\tR$ isomorphic to a $\Gr(2,8)$-bundle over~$\OGr(3,\rV)$.

We use the spinor quadratic line complex $R$ to define an interesting invariant for all linear sections of~$X$ of codimension at least~2.
Given such $X_K \subset X$ we define a {\sf quadratic invariant} of $X_K$ as
\begin{equation*}
R_K := \Gr(2,K) \cap R \subset \Gr(2,\SS^\vee).
\end{equation*}
It is easy to show (Lemma~\ref{lemma:rk-invariant}) that if $X_{K_1} \cong X_{K_2}$, 
the associated quadratic invariants are isomorphic as well: $R_{K_1} \cong R_{K_2}$
In terms of the quadratic invariant $R_K$ it is easy to characterize special codimension~2 linear sections --- 
a section~$X_K$ is special if and only if $R_K$ is nonempty (this is, of course, a tautological characterization).
Associating with a linear section $X_K$ its quadratic line complex $R_K$ defines a (rational) map from the moduli stack
of linear sections $X_K \subset X$ of codimension $k$ to the moduli stack of quadratic line complexes in $\Gr(2,k)$.
It is an interesting question to understand the relation between these moduli stacks.

We finish the paper by discussing some properties of $R_K$.
We show that $R_K \subset \Gr(2,K)$ is almost always a divisor (Lemma~\ref{lemma:rk-codim4}).
The only exception (besides the special linear sections of codimension 2) is the case of a linear section $X_K \subset X$ of codimension~3
such that~$F_4(X_K) \ne \varnothing$.
We show (Proposition~\ref{proposition:x3-with-p4}) that there is a unique isomorphism class of such $X_K$ and call them {\sf very special}.
Geometrically, a very special $X_K$ can be obtained by blowing up a quintic del Pezzo threefold inside $\P^7$ 
(it is contained in a hyperplane $\P^6 \subset \P^7$) and then contracting the strict transform of the hyperplane.
This transformation, as well as the birational isomorphisms of $X$ with $\P^{10}$ 
and of a special codimension~2 linear section of $X$ with $\P^8$ discussed above,
are particular cases of \emph{special birational transformations of type~$(2,1)$} 
studied by Fu and Hwang in~\cite[Proposition~2.12]{fu2015special}.

\subsection{Minifolds}

Finally, let us say some words about some further results.
Probably, one of the most interesting cases that will be considered in subsequent papers
is the case of linear sections $X_K \subset X$ of codimension~5.
These varieties are particularly interesting, because they are \emph{minifolds} (see~\cite{galkin}) --- 
they have the same Hodge diamond as $\P^5$ and their derived category of coherent sheaves is generated 
by the minimal possible number~$\dim X_K + 1 = 6$ of exceptional bundles.
Besides $X_K$ only four other minifolds of dimension 5 are known --- 
these are $\P^5$, $Q^5$, the adjoint $\rG_2$-Grassmannian, 
and a hyperplane section of the Lagrangian Grassmannian $\LGr(3,6)$.
Among those~$X_K$ is the only minifold with moduli.
Moreover, besides two other examples in dimension 3 (the quintic del Pezzo threefold and prime Fano threefolds of genus~12), 
the only minifolds known by now are projective spaces and odd-dimensional quadrics.

One of the motivations to start this project was the following funny observation.
Consider the three 5-dimensional minifolds of index 3: 
the adjoint $\rG_2$-Grassmannian, the hyperplane section of $\LGr(3,6)$, and a fivefold $X_K$.
For each of them the Hilbert scheme of lines is again a Fano variety of dimension~5.
It is immediate to see that for the first of them, the Hilbert scheme of lines is isomorphic to $Q^5$.
It is much less evident (see, however,~\cite[Corollary~6.7]{k2016kuchle-c5}) that 
for the second of them the Hilbert scheme of lines is isomorphic to the adjoint $\rG_2$-Grassmannian.
In particular, in both cases the Hilbert scheme is itself a 5-dimensional minifold!
So, one could hope that the Hilbert schemes $F_1(X_K)$ of lines on fivefolds~$X_K$ would give a new example of a minifold.

It turns out, however, that this is not true, but still the geometry of $F_1(X_K)$ is quite interesting.
In fact, in a forthcoming paper we will show that there is a natural \emph{Sarkisov link} relating the Hilbert scheme $F_1(X_K)$
to the quadratic line complex $R_K \subset \Gr(2,K)$, which in this case is a \emph{Gushel--Mukai fivefold} (see~\cite{debarre2015gushel}).
Explicitly, there are certain natural $\P^1$-bundles over $R_K$ and $F_1(X_K)$ that are related by a flop
\begin{equation*}
\xymatrix@1{\P^1_{R_K} \ar@{<-->}[rr] && \P^1_{F_1(X_K)}}.
\end{equation*}
The flopping locus on both sides is a $\P^2$-bundle over the curve $F_2(X_K)$ (the Hilbert scheme of planes on~$X_K$),
which can be equivalently described as the locus in $\Gr(3,K)$ of very special codimension-3 over-sections of $X_K$.
In particular, it follows that $F_1(X_K)$ is smooth if and only if $R_K$ is smooth, and that
the Hodge numbers of $F_1(X_K)$ and $R_K$ are the same
(see~\cite[Proposition~3.1]{debarre2016gushel} for the Hodge numbers of~$R_K$), and so $F_1(X_K)$ is not a minifold.

In fact, the Hilbert scheme $F_1(X_K)$ was already considered in~\cite{Ranestad2000}.
In particular, it was proved in~\cite[Theorem~8.6]{Ranestad2000} that $F_1(X_K)$ can be realized 
as a variety of sums of powers for a general cubic threefold (and some invariants of $F_1(X_K)$ were computed).
It would be very interesting to understand the relation between the cubic threefold and the Gushel--Mukai fivefold
associated with the minifold $X_K$.

Note also that the derived category of a smooth Gushel--Mukai fivefold 
has an interesting semiorthogonal decomposition~\cite[Proposition~2.3]{kuznetsov2016derived}
consisting of six exceptional vector bundles and an Enriques-type category~\cite[Proposition~2.6]{kuznetsov2016derived}.
The relation of $F_1(X_K)$ with $R_K$ suggests that the derived category of $F_1(X_K)$ has a semiorthogonal 
decomposition of the same type (thus, the Enriques category serves as the obstruction to the minifold property).
It will be interesting to find it.

\subsection{Structure of the paper}

The paper is organized as follows.
In Section~\ref{section:preliminaries} we give a short reminder of isotropic Grassmannians, spinor representations and bundles,
and prove a useful blowup lemma (Lemma~\ref{lemma:blowup}).
In Section~\ref{section:x-and-xk} we introduce the spinor tenfold $X$ and describe its Hilbert schemes of lines, planes, and other linear spaces.
We also discuss a smoothness criterion for linear sections of $X$, their semiorthogonal decompositions, 
and consequences of those for Chow motives with rational coefficients.
In Section~\ref{section:blowup-x} we prove that the blowup of $\P^{15}$ along $X$ is isomorphic to a $\P^7$-bundle over~$Q^8$
and deduce many consequences of this result.
Among these there is a description of Hilbert schemes of quadrics on~$X$ and of the integral Chow motives of linear sections of~$X$.
In Section~\ref{section:hyperplanes} we prove that the blowup of $X$ along a 4-space is isomorphic to the blowup of $\P^{10}$ along~$\Gr(2,5)$
and extract from this a description of a singular hyperplane section of~$X$.
We also prove that the blowup of $X$ along a 6-dimensional quadric is isomorphic to a $\P^4$-bundle over~$Q^6$,
and deduce from this a description of a smooth hyperplane section of $X$ as a $\P^4$-bundle over $Q^5$.
In Section~\ref{section:codimension-2} we classify all smooth linear sections of $X$ of codimension~2 and introduce the spinor quadratic line complex~$R$.
In Section~\ref{section:other} we define the quadratic invariant $R_K$ of a linear section $X_K \subset X$
and use it to answer some questions about geometry of linear sections of $X$ of codimension greater than~2.

\subsection{Conventions}

We work over a field $\Bbbk$, which we assume to be an algebraically closed field of zero characteristic.
By $\Gr(s,V)$ we denote the Grassmannian of $s$-dimensional vector subspaces in $V$.
In particular, $\P(V) = \Gr(1,V)$ is the projectivization of a vector space $V$.
Similarly, for a vector bundle $\cV$ on a scheme~$S$ we denote by
\begin{equation*}
\P_S(\cV) = \Proj \left( \bigoplus_{p=0}^\infty \Sym^p \cV^\vee \right)
\end{equation*}
the projectivization of $\cV$.
We denote by $\cO_{\P_S(\cV)}(1)$ the Grothendieck line bundle on $\cV$, normalized by the property $\pi_*\cO_{\P_S(\cV)}(1) \cong \cV^\vee$.
Its first Chern class is called the {\sf relative hyperplane class} of $\P_S(\cV)$.

\subsection{Acknowledgement}

I am grateful to Dima Orlov and Yura Prokhorov for useful discussions and to Laurent Manivel for the reference~\cite{Landsberg2003}.
When the first version of this paper was published I was informed by Baohua Fu about similar results obtained 
in~\cite{pasquier2009} and~\cite{fu2015special}; I am very grateful to him for these and other references.

\section{Preliminaries}
\label{section:preliminaries}

\subsection{Isotropic orthogonal Grassmannians}

Let $V$ be a vector space over an algebraically closed field~$\Bbbk$ of characteristic zero  with a nondegenerate quadratic form $\bq_V$.
We denote by $\OGr(s,V) \subset \Gr(s,V)$ the subvariety of the Grassmannian that parameterizes $\bq_V$-isotropic $s$-dimensional subspaces in $V$.
In particular,
\begin{equation*}
\OGr(1,V) = Q_V \subset  \P(V)
\end{equation*}
is a smooth quadric defined by $\bq_V$.

Of course, when $2s > \dim V$ the isotropic Grassmannian is empty, so we will always assume~$2s \le \dim V$.
Each of these isotropic Grassmannians is a homogeneous space for the group $\Spin(V)$.
Restricting to the case when $\dim V = 2m$ is even, so that $\Spin(V)$ is a group of Dynkin type $\rD_m$, we can write
\begin{equation*}
\OGr(s,V) = 
\begin{cases}
\Spin(V)/\bP_s, & \text{if $s \le m - 2$},\\
\Spin(V)/\bP_{m-1,m}, & \text{if $s = m - 1$},\\
(\Spin(V)/\bP_{m}) \sqcup (\Spin(V)/\bP_{m-1}) , & \text{if $s = m$},
\end{cases}
\end{equation*}
where $\bP_I$ denote a parabolic subgroup in $\Spin(V)$ corresponding to the set of vertices $I$ of the Dynkin diagram.
In particular, the isotropic Grassmannian $\OGr(m,V)$ has two connected components which we denote by $\OGr_+(k,V)$ and $\OGr_-(k,V)$ respectively.
We will use the convention 
\begin{equation*}
\OGr_+(m,V) = \Spin(V)/\bP_m
\qquad\text{and}\qquad
\OGr_-(m,V) = \Spin(V)/\bP_{m-1}.
\end{equation*}
Note that these varieties are abstractly isomorphic (and an isomorphism is carried out by an outer automorphism of $\Spin(V)$).
To figure out the component to which a given isotropic subspace belongs the following property is useful:
if $U'$ and $U''$ are maximal isotropic subspaces then $U'$ and $U''$ belong to the same component if and only if 
$\dim(U' \cap U'') \equiv \dim U' \pmod 2$.
Note also that for small $m$ one has the following identifications of the isotropic Grassmannians:
\begin{equation}
\label{eq:small-ogr}
\OGr_+(1,2) \cong \Spec(\Bbbk),
\qquad 
\OGr_+(2,4) \cong \P^1,
\qquad 
\OGr_+(3,6) \cong \P^3,
\qquad 
\OGr_+(4,8) \cong Q^6,
\end{equation}
and similarly for $\OGr_-(m,2m)$.
The last isomorphism is a manifestation of \emph{triality}:
\begin{equation*}
\Spin(8)/\bP_1 \cong \Spin(8)/\bP_3 \cong \Spin(8)/\bP_4.
\end{equation*}
Besides the Grassmannians, we also need isotropic flag varieties.
Assuming (for simplicity) that~$V$ is a vector space of even dimension $2m$ with a non-degenerate quadratic form $\bq_V$,
and $0 < s_1 < \dots < s_r \le m$ is a sequence of integers, we denote by $\OFl(s_1,\dots,s_r;V) \subset \Fl(s_1,\dots,s_r;V)$
the subvariety of the flag variety that parameterizes $\bq_V$-isotropic flags.
If $s_r = m$, it has two connected components which we denote 
\begin{equation*}
\OFl_\pm(s_1,\dots,s_{r-1},m;V) = \OFl(s_1,\dots,s_{r-1},m;V) \times_{\OGr(m,V)} \OGr_\pm(m,V).
\end{equation*}

\subsection{Spinor spaces and bundles}
\label{subsection:spinors}

For the material in this section we refer to~\cite{ottaviani} for the case of quadrics and to~\cite[Section~6]{k2008exceptional} in general.
Note that the conventions for the definition of spinor bundles in these two references are opposite.
Here we stick to the convention used in~\cite{ottaviani}.

Assume again that $\dim V = 2m$ and $\bq_V$ is a nondegenerate quadratic form. 
Denote by $\omega_i$ the fundamental weight of $\Spin(V)$ corresponding to the vertex $i$ of the Dynkin diagram $\rD_m$,
and by $\rV_{\Spin(V)}^{\lambda}$ the irreducible $\Spin(V)$-representation with the highest weight $\lambda$.
Then
\begin{equation*}
V \cong \rV_{\Spin(V)}^{\omega_1},
\end{equation*}
and the irreducible representations corresponding to the last two vertices 
\begin{equation*}
\SS = \SS_+ := \left(\rV_{\Spin(V)}^{\omega_m}\right)^\vee
\qquad\text{and}\qquad
\SS_- := \left(\rV_{\Spin(V)}^{\omega_{m-1}}\right)^\vee
\end{equation*}
are called the {\sf half-spinor representations}.
Their dimensions are equal to
\begin{equation*}
\dim(\SS) = \dim(\SS_-) = 2^{m-1}
\end{equation*}
and they are swapped by outer automorphisms of $\Spin(V)$.
The half-spinor representations are self-dual or mutually dual depending on the parity of $m$; explicitly
\begin{equation}
\label{eq:dual-s}
\SS^\vee \cong \SS_{(-1)^m}
\qquad\text{and}\qquad
\SS_-^\vee \cong \SS_{(-1)^{m-1}}.
\end{equation}

A similar construction allows to define spinor vector bundles on isotropic Grassmannians of $V$.
Namely, for the maximal isotropic Grassmannians $\OGr_\pm(m,V)$ 
the spinor bundles are just the anti-ample generators of the Picard groups:
\begin{equation*}
\cS_1 = \cS_{1,+} := \cL_{-\omega_m} \in \Pic(\OGr_+(m,V))
\qquad\text{and}\qquad
\cS_{1,-} := \cL_{-\omega_{m-1}} \in \Pic(\OGr_-(m,V))
\end{equation*}
(we use the subscript to specify the rank of the bundles).
In what follows we denote $\cS_{1,\pm}^\vee$ simply by~$\cO_{\OGr_\pm(m,V)}(1)$ and consider these line bundles 
as polarizations of~$\OGr_\pm(m,V)$.
Note that, if $\cU_\pm$ denote the tautological rank-$m$ vector bundles on $\OGr_\pm(m,V)$, then
\begin{equation}
\label{eq:detu-s}
\det \cU_m^\vee \cong \cO_{\OGr_\pm(m,V)}(2).
\end{equation}

Similarly, for $s \le m - 2$ we consider the isotropic flag varieties $\OFl_\pm(s,m;V)$ 
with the projections $\pr_s$ and $\pr_{m,\pm}$ to $\OGr(s,V)$ and $\OGr_\pm(m,V)$ respectively, and define
\begin{equation*}
\cS_{2^{m-s-1}} = \cS_{2^{m-s-1},+} := \left(\pr_{s*}(\pr_m^*(\cS_1^\vee))\right)^\vee
\qquad\text{and}\qquad
\cS_{2^{m-s-1},-} := \left(\pr_{s*}(\pr_{m,-}^*(\cS_{1,-}^\vee))\right)^\vee,
\end{equation*}
where again the subscript specifies the rank. 
The analogue of the duality isomorphisms~\eqref{eq:dual-s} for spinor bundles is similar, but also involves a twist.

\begin{lemma}[cf.~{\cite[Theorem~2.8]{ottaviani}}, {\cite[Corollary~6.5, Proposition~6.6]{k2008exceptional}}]
\label{lemma:cs-dual}
Let $\dim V = 2m$. If~$s \le m - 2$, the spinor bundle $\cS$ on $\OGr(s,V)$ has the following property
\begin{equation*}
\cS^\vee \cong \cS_{(-1)^{m-s}}(1).
\qquad\text{and}\qquad 
\cS_-^\vee \cong \cS_{(-1)^{m-s-1}}(1).
\end{equation*}
In particular, $\det(\cS_\pm) \cong \cO(-2^{m-s-2})$.
\end{lemma}

Furthermore, we have an identification $\SS_\pm^\vee = H^0(\OGr(s,V),\cS_\pm^\vee)$.
It gives canonical evaluation morphisms $\SS_\pm^\vee \otimes \cO_{\OGr(s,V)} \to \cS_\pm^\vee$, 
which are surjective by homogeneity of $\OGr(s,V)$, 
and by duality it gives fiberwise monomorphisms $\cS_\pm \hookrightarrow \SS_\pm \otimes \cO_{\OGr(s,V)}$.
For $s = m$ this defines embeddings 
\begin{equation}
\label{eq:spinor-embedding}
\OGr_\pm(m,V) \to \P(\SS_\pm).
\end{equation}
In cases $m \in \{1,2,3\}$ it induces the first three of isomorphisms~\eqref{eq:small-ogr},
and for $m = 4$ it induces the embedding $\OGr_\pm(4,8) \hookrightarrow \P^7$ as a quadric, 
thus giving the last of isomorphisms~\eqref{eq:small-ogr}.
In general, \eqref{eq:spinor-embedding} is called the {\sf spinor embedding}.
By~\eqref{eq:detu-s} the Pl\"ucker embedding $\OGr_\pm(m,V) \hookrightarrow \Gr(m,V) \hookrightarrow \P(\bw{m}V)$
is a composition of the spinor embedding with the double Veronese embedding.

In the case $s = 1$ (so that $\OGr(1,V) = Q$) the embeddings $\cS_\pm \to \SS_\pm \otimes \cO_Q$ extend to exact sequences.

\begin{lemma}[{\cite[Theorem~2.8]{ottaviani}}]
\label{lemma:spinor-sequence}
If\/ $Q \subset \P(V)$ is an even-dimensional quadric then there are canonical exact sequences
\begin{equation*}
0 \to \cS \to \SS \otimes \cO_Q \to \cS_-(1) \to 0
\qquad\text{and}\qquad 
0 \to \cS_- \to \SS_- \otimes \cO_Q \to \cS(1) \to 0
\end{equation*}
\end{lemma}

In the case $s > 1$ the situation is more complicated.
Instead of a short exact sequence, one extends the spinor subbundle to a filtration, 
whose factors involve both the original spinor bundle and the tautological vector bundle $\cU_s$ on the isotropic Grassmannian.

\begin{lemma}[cf.~{\cite[Proposition~6.3]{k2008exceptional}}]
\label{lemma:filtration-ss}
Let $V$ be even-dimensional, $\dim V = 2m$. 
If $s \le m - 2$, the trivial vector bundle $\SS \otimes \cO_{\OGr(s,V)}$ has a natural filtration whose factors are isomorphic to 
\begin{align*}
&\cS_{(-1)^i} \otimes \bw{i}\cU_s^\vee, &\qquad& 0 \le i \le s.\\
\intertext{Similarly, the trivial vector bundles $\SS \otimes \cO_{\OGr_+(m,V)}$ and $\SS_- \otimes \cO_{\OGr_+(m,V)}$ 
have natural filtrations whose factors are isomorphic to}
&\cS \otimes \bw{2i}\cU_m^\vee, 	&\qquad& 0 \le 2i \le m,	&& \text{for $\SS \otimes \cO_{\OGr_+(m,V)}$}, \\
&\cS \otimes \bw{2i + 1}\cU_m^\vee, 	&\qquad& 0 \le 2i + 1 \le m,	&& \text{for $\SS_- \otimes \cO_{\OGr_+(m,V)}$}.
\end{align*}
\end{lemma}

\begin{remark}
\label{remark:filtrations}
Choosing a point $[U_m] \in \OGr_+(m,V)$ and trivializing the spinor line bundle $\cS$ at this point, we obtain from the lemma 
filtrations on the vector spaces $\SS$ and $\SS_-$ with the factors being $\bw{2i}U_m^\vee$ for the first and $\bw{2i+1}U_m^\vee$ for the second.
Similarly, for a point $[U_{m,-}] \in \OGr_-(m,V)$ we have filtrations on $\SS$ and $\SS_-$ with the factors being~$\bw{2i+1}U_{m,-}^\vee$ 
for the first and $\bw{2i}U_{m,-}^\vee$ for the second.

We note that these filtrations are compatible with the duality between $\SS$ and $\SS^\vee$.
In particular, assuming that $m$ is odd, the first step of the filtration $\Bbbk = \bw{0}U_{m,-}^\vee$ gives a point of $\SS^\vee = \SS_-$ 
(this is the point $[U_{m,-}] \in \OGr_-(m,V)$ in the half-spinor embedding $\OGr_-(m,V) \hookrightarrow \P(\SS_-) = \P(\SS^\vee)$),
and the corresponding hyperplane in $\P(\SS)$ corresponds to the projection $\SS \to \bw{m}U_{m,-}^\vee$ to the last factor
of the second filtration.
\end{remark}

\subsection{Blowup lemma}

The following result of Ein and Shepherd-Barron \cite{ein1989} will be used several times in the paper
to prove that a particular birational morphism is a smooth blowup.
For a projective morphism $X \to Y$ we denote by $\rho(X/Y)$ the relative Picard rank.

\begin{lemma}
\label{lemma:blowup}
Assume that there is a commutative diagram,
\begin{equation*}
\xymatrix@C=5em{
E
\ar@{^{(}->}[r]^i \ar[d]_p & 
X \ar[d]^f 
\\
Z \ar@{^{(}->}[r]^j & 
Y
}
\end{equation*}
where $X$, $Y$, $Z$ are smooth varieties, $\codim_Y(Z) \ge 2$, $E$ is an irreducible divisor,
$f$ is a projective birational morphism, $p$ is surjective, and $i$ and $j$ are closed embeddings.
If $\rho(X/Y) = 1$ 
then $f$ is the blowup of\/~$Y$ with center in~$Z$, $X \cong \Bl_Z(Y)$, and $E$ is the exceptional divisor of $f$.
\end{lemma}

\begin{proof}
%
Since $X$ and $Y$ are smooth, the exceptional locus of $f$ is a divisor, 
and since $\rho(X/Y) = 1$, the exceptional divisor is irreducible.
Since it contains $E$, we conclude that $E$ is the exceptional locus of~$f$.
By Zariski connectedness Theorem $f^{-1}(Z)$ is in the exceptional locus of $f$, 
hence $f^{-1}(Z) = E$ set-theoretically, and $f \colon X \setminus E \to Y \setminus Z$ an isomorphism.
%
Thus $Z$ is the base locus of $f^{-1}$.
Since $E$ is the set-theoretic preimage of $Z$, \cite[Theorem~1.1]{ein1989} proves 
that $f$ is the blowup of $Z$ and $E$ is its exceptional divisor.
\end{proof}

The following statement is also well known, but is quite useful.

\begin{lemma}
\label{lemma:smoothness-criterion}
If $f \colon X \to Y$ is a projectivization of a vector bundle then $X$ is smooth if and only if $Y$ is smooth.
If $f \colon X \to Y$ is a blowup with center in $Z \subset Y$, where $Z$ is a locally complete intersection in~$Y$,
then $X$ is smooth if and only if both $Y$ and $Z$ is smooth.
\end{lemma}
\begin{proof}
The first part and one direction of the second part are evident. 
So, assume that $X = \Bl_Z(Y)$ is smooth. 
Clearly $Y$ is smooth away of $Z$.
Set $c = \codim_Y(Z)$ and let $f_1,\dots,f_c$ be local equations of~$Z$ in~$Y$.
Then $X \subset Y \times \P^{c-1}$ is given by the equations $u_if_j - u_jf_i = 0$, 
where $(u_1:\dots:u_c)$ are the homogeneous coordinates on~$\P^{c-1}$.
In the chart $u_c \ne 0$ (in this chart we can assume $u_c = 1$ and take $u_1,\dots,u_{c-1}$ to be the coordinates) 
the equations can be rewritten as $f_i - u_if_c = 0$, $i = 1, \dots, c - 1$.
It follows that~$X$ is a locally complete intersection in~$Y \times \P^{c-1}$.
Since $X$ is smooth, we conclude that~$Y \times \P^{c-1}$ is smooth along $X$.
In particular, it is smooth along the exceptional divisor of the blowup, i.e., along~$Z \times \P^{c-1}$.
Thus $Y$ is smooth along $Z$, hence is smooth everywhere.
Finally, by comparing the Jacobian matrices corresponding to the equations of $X$ in $Y \times \P^{c-1}$ and $Z$ in $Y$,
we easily deduce smoothness of $Z$ from smoothness of $X$.
\end{proof}

\section{The spinor tenfold and its linear sections}
\label{section:x-and-xk}

The spinor tenfold $X$ and its projective dual variety $X^\vee$ (which is abstractly isomorphic to $X$) were described in the Introduction.
We start by recalling some notation introduced earlier.

\subsection{Notation}
\label{subsection:notation}

We fix a vector space $\rV$ of dimension~10 (in the notation of Section~\ref{section:preliminaries} 
this means that~\mbox{$m = 5$}) and a nondegenerate quadratic form $\bq_\rV$ on it.
We will always identify the spaces $\rV$ and~$\rV^\vee$ with the help of the quadratic form $\bq_\rV$.
We denote by $\SS$ and $\SS^\vee \cong \SS_-$ (see~\eqref{eq:dual-s}) the corresponding half-spinor representations.
Recall that
\begin{alignat*}{3}
X\hphantom{{}^\vee} 	&:= \OGr_+(5;\rV) &&\cong \Spin(\rV)/\bP_5 &&\subset \P(\SS),\\
X^\vee 			&:= \OGr_-(5;\rV) &&\cong \Spin(\rV)/\bP_4 &&\subset \P(\SS^\vee).
\intertext{We usually denote a point of $X$ by $[U_5]$ and a point of $X^\vee$ by $[U_{5,-}]$,
meaning that $U_5,U_{5,-} \subset \rV$ are the corresponding 5-dimensional isotropic subspaces.
Accordingly, we denote by $\cU_5$ and $\cU_{5,-}$ the tautological vector bundles on $X$ and $X^\vee$,
and in many cases we abbreviate these to just $\cU$ and $\cU_-$.
Furthermore, we denote}
\rQ\hphantom{{}_-} 	&:= \OGr(1,\rV)  &&\cong \Spin(\rV)/\bP_1 &&\subset \P(\rV),
\intertext{this is a smooth quadric of dimension 8, and}
\cQ\hphantom{{}_-} 	&:= \OFl_+(1,5;\rV) && \cong \Spin(\rV)/\bP_{1,5},\\
\cQ_- 			&:= \OFl_-(1,5;\rV) && \cong \Spin(\rV)/\bP_{1,4}.
\end{alignat*}
Then we have a diagram
\begin{equation}
\label{diagram:x-q-x}
\vcenter{\xymatrix{
&
\cQ \ar[dl] \ar[dr] &&
\cQ_- \ar[dl] \ar[dr] \\
X && \rQ && X^\vee
}}
\end{equation}
in which the outer arrows are $\P^4$-fibrations, while the inner arrows are fibrations in smooth 6-dimensional quadrics.
To be more precise, on one hand, we have isomorphisms
\begin{equation}
\label{eq:cq-pu}
\cQ \cong \P_X(\cU),
\qquad 
\cQ_- \cong \P_{X^\vee}(\cU_-),
\end{equation}
and on the other hand, we have canonical embeddings into the projectivizations of the spinor bundles
\begin{equation}
\label{eq:cq-in-ps}
\cQ \hookrightarrow \P_\rQ(\cS_8),
\qquad 
\cQ_- \hookrightarrow \P_\rQ(\cS_{8,-}),
\end{equation}
which are the relative versions of the last embedding of~\eqref{eq:small-ogr}.
In most cases, we consider $\cQ$ and $\cQ_-$ as families of 6-dimensional quadrics $\cQ_v \subset X$ and $\cQ_{v,-} \subset X^\vee$
parameterized by $v \in \rQ$ (see also~\eqref{eq:qv-qvminus}).

We remind that the two components of $\OGr(m,\rV)$ can be distinguished by the parity of the dimension of intersection of subspaces:
\begin{equation}
\label{eq:parity-intersection}
\dim (U'_5 \cap U''_5) \equiv 
\begin{cases}
{0 \pmod 2}, & \text{if $U'_5$ and $U''_5$ are in different components of $\OGr(5,\rV)$},\\
{1 \pmod 2}, & \text{if $U'_5$ and $U''_5$ are in the same component of $\OGr(5,\rV)$}.
\end{cases}
\end{equation} 

We denote by $\cO_X(-1) = \cS_1$, $\cO_{X^\vee}(-1) = \cS_{1,-}$ the spinor line bundles on $X$ and $X^\vee$.
Then
\begin{equation}
\det \cU \cong \cO_X(-2),
\qquad 
\det \cU_- \cong \cO_{X^\vee}(-2).
\end{equation} 
Moreover, the canonical bundle of $X$ can be written as
\begin{equation}
\omega_X \cong (\det\cU)^{\otimes 4} \cong \cO_X(-8).
\end{equation} 

\begin{remark}
\label{remark:qv-ogr-zero}
For further use we note that the families of quadrics $\cQ$ and $\cQ_-$ in~\eqref{diagram:x-q-x} have the following interpretation:
for $v \in \rQ$ the outer arrows in the diagram induce identifications
\begin{equation}
\label{eq:qv-qvminus}
\cQ_v = \OGr_+(4,v^\perp/v) \subset X
\qquad\text{and}\qquad 
\cQ_{-,v} = \OGr_-(4,v^\perp/v) \subset X^\vee.
\end{equation}
Note also that these quadrics are the zero loci of the global section $v$ of the vector bundles $\cU^\vee$ and $\cU^\vee_-$ on $X$ and $X^\vee$ respectively.
Moreover, it follows from~\eqref{diagram:x-q-x} that if $v \in \P(\rV) \setminus \rQ$ the corresponding global sections 
of $\cU^\vee$ and $\cU^\vee_-$ are everywhere non-zero.
\end{remark}

\subsection{Linear spaces on the spinor tenfold}

In this section we describe all linear spaces on the spinor tenfold.
We denote by 
\begin{equation*}
F_d(X) = \Hilb^{(t+1)\cdots(t+d)/d!}(X)
\end{equation*}
the Hilbert scheme of linearly embedded $\P^d \subset X \subset \P(\SS)$.

Let $U_s \subset \rV$ be a $\bq_\rV$-isotropic subspace of dimension $s$.
Denote by $U_s^\perp \subset \rV$ its orthogonal with respect to the quadratic form $\bq_\rV$; 
then $U_s \subset U_s^\perp$, the quotient space $U_s^\perp/U_s$ is $(10 -2s)$-dimensional and 
has a canonical quadratic form induced by $\bq_\rV$.
Moreover, for any isotropic subspace of dimension~$d$ in $U_s^\perp/U_s$ its preimage in $U_s^\perp \subset \rV$ is an isotropic subspace of dimension $d + s$.
In particular, we have a natural embedding
\begin{equation}
\OGr_+(5-s,U_s^\perp/U_s) \subset \OGr_+(5,\rV) = X
\end{equation}
and under this embedding the line bundle $\cO_X(1)$ restricts to the ample generator of the Picard group.
For instance, by~\eqref{eq:small-ogr} for any isotropic 2-dimensional subspace $U_2 \subset \rV$ the subvariety
\begin{equation}
\label{def:pi3-u2}
\Pi^3_{U_2} := \OGr_+(3,U_2^\perp/U_2) \cong \P^3 \hookrightarrow X
\end{equation}
is a linearly embedded 3-space,
and for any isotropic 3-dimensional subspace $U_2 \subset \rV$ the subvariety
\begin{equation}
\label{def:l-u3}
L_{U_3} := \OGr_+(2,U_3^\perp/U_3) \hookrightarrow X.
\end{equation}
is a line in $X$.
In the same vein we define a line on $X^\vee$ by
\begin{equation}
\label{def:l-u3-minus}
L^-_{U_3} := \OGr_-(2,U_3^\perp/U_3) \hookrightarrow X^\vee.
\end{equation}
We note that the family of 3-spaces~\eqref{def:pi3-u2} on $X$ is given by the diagram
\begin{equation}
\label{eq:universal-3space}
\vcenter{\xymatrix{
& \OFl_+(2,5;\rV) \ar@{=}[r]^-{\sim} \ar[dl] & 
\P_{\OGr(2,\rV)}(\cS_4) \ar[dr] 
\\
X &&& 
\OGr(2,\rV)
}}
\end{equation}
and the families of lines on $X$ and $X^\vee$ are given by the diagrams
\begin{equation}
\label{eq:universal-lines}
\vcenter{\xymatrix@C=1em{
& \P_{\OGr(3,\rV)}(\cS_2) \ar[dl] \ar@{=}[r]^-{\sim} &
\OFl_+(3,5;\rV) \ar[dr] && 
\OFl_-(3,5;\rV) \ar@{=}[r]^-{\sim} \ar[dl] &
\P_{\OGr(3,\rV)}(\cS_{2,-}) \ar[dr] \\
X &&& 
\OGr(3,\rV) &&& 
X^\vee
}}
\end{equation}
where $\cS_4$, $\cS_2$, and $\cS_{2,-}$ are the corresponding spinor bundles
and the isomorphism are given by relative versions of~\eqref{eq:small-ogr}.

On the other hand, consider an isotropic 5-dimensional subspace $U_{5,-} \subset \rV$ 
(corresponding to a point of~$X^\vee$.
Then we have a natural embedding $\Gr(4,U_{5,-}) \subset \OGr(4,\rV)$.
Furthermore, by~\eqref{eq:small-ogr} every isotropic subspace in $\rV$ of dimension 4 extends in a unique way to a five-dimensional subspace
corresponding to a point of~$X$.
This defines a regular map $\OGr(4,\rV) \to X$.
Combining these two we obtain an embedding 
\begin{equation}
\label{def:pi4-u5}
\Pi^4_{U_{5,-}} := \Gr(4,U_{5,-}) \cong \P^4 \hookrightarrow X.
\end{equation}
This is a linearly embedded 4-space. The corresponding family is given by the diagram
\begin{equation*}
\xymatrix@C=3em{
& 
\P_X(\cU^\vee(-1)) \ar[dl] \ar@{=}[r]^-\sim &
\OGr(4,V) \ar[dl] \ar[dr] &
\P_{X^\vee}(\cU_-^\vee(-1)) \ar[dr] \ar@{=}[l]_-\sim 
\\
X \ar@{=}[r]^-\sim &
\OGr_+(5,\rV) &&
\OGr_-(5,\rV) &
X^\vee \ar@{=}[l]_-\sim &
}
\end{equation*}
Note also that for any subspace $U_s \subset U_{5,-}$ we have a natural embedding
\begin{equation}
\label{def:pi-us-u5}
\Pi^{4 -s}_{U_s,U_{5,-}} := \Gr(4-s,U_{5,-}/U_s) \hookrightarrow \Gr(4,U_{5,-}) \hookrightarrow X,
\end{equation}
and this is a linear subspace of dimension $4-s$ on $X$.

\begin{theorem}
\label{theorem:linear-spaces}
Any linear space on $X$ is one of the following:

\noindent$(1)$ 
If $L \subset X$ is a line then there is a unique isotropic $3$-dimensional subspace $U_3 \subset \rV$ such that $L = L_{U_3}$.
Moreover, for any isotropic $U_{5,-} \subset \rV$ containing $U_3$ the line $L_{U_3}$ can also be written as $L_{U_3} = \Pi^1_{U_3,U_{5,-}}$.

\noindent$(2)$ 
If $\Pi \subset X$ is a plane, there is a unique isotropic flag $U_2 \subset U_{5,-} \subset \rV$ such that $\Pi = \Pi^2_{U_2,U_{5,-}}$.

\noindent$(3)$ 
If $\Pi \subset X$ is a $3$-space, then exactly one of the following two possibilities holds:
\begin{enumerate}
\item[(a)] 
either there is a unique isotropic $2$-dimensional subspace $U_2 \subset \rV$ such that $\Pi = \Pi^3_{U_2}$;
\item[(b)]
or there is a unique isotropic flag $U_1 \subset U_{5,-} \subset \rV$ such that $\Pi := \Pi^3_{U_1,U_{5,-}}$.
\end{enumerate}

\noindent$(4)$ 
If $\Pi \subset X$ is a $4$-space, then there is a unique isotropic subspace $U_{5,-} \subset \rV$ such that $\Pi := \Pi^4_{U_{5,-}}$.

\noindent
In particular, there are no linear subspaces on $X$ of dimension $d \ge 5$.

Furthermore, the Hilbert schemes of linear spaces on $X$ are the homogeneous $\Spin(V)$-varieties:
\begin{equation}
\label{eq:fk-x}
\begin{aligned}
F_1(X) &\cong \Spin(\rV)/\bP_3 					&&\cong \OGr(3,V),\\
F_2(X) &\cong \Spin(\rV)/\bP_{2,4} 					&&\cong \OFl_-(2,5;\rV),\\
F_3(X) &\cong \Spin(\rV)/\bP_2 \sqcup \Spin(\rV)/\bP_{1,4} 	&&\cong \OGr(2,\rV) \sqcup \OFl_-(1,5;\rV),\\
F_4(X) &\cong \Spin(\rV)/\bP_4 					&&\cong \OGr_-(5,\rV) = X^\vee.
\end{aligned}
\end{equation} 
\end{theorem}

\begin{proof}
This follows from a general result~\cite[Theorem~4.9]{Landsberg2003} of Landsberg and Manivel.
Indeed, to describe $F_d(X)$ we should consider all minimal sets of vertices of the Dynkin diagram $\rD_5$ such that 
the connected component of its complement containing vertex $5$ is the Dynkin diagram of type $\rA_d$ with vertex~5 being its end-point.
The next picture shows all the possibilities:
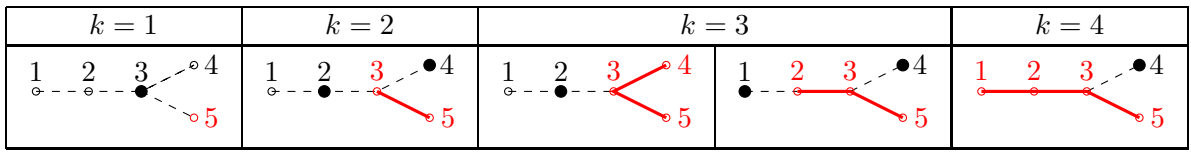
\begin{equation*}
\begin{array}{|c|c|c|c|c|}
\hline
k = 1 & k = 2 & \multicolumn{2}{c|}{k = 3} & k = 4 
\\
\hline
\begin{tikzpicture}[xscale=.7,yscale=.7]
\draw[dashed] (0,.5) node [above] {1} -- (1,.5) node [above] {2} -- (2,.5) node [above] {3} -- (3,1) node [right] {4} -- (2,.5) -- (3,0) node [right,red] {5};
\draw (0,.5) circle 		[radius=.07];
\draw (1,.5) circle 		[radius=.07];
\draw[fill] (2,.5) circle 		[radius=.11];
\draw (3,1) circle 		[radius=.07];
\draw[red] (3,0) circle 	[radius=.07];
\end{tikzpicture}
&
\begin{tikzpicture}[xscale=.7,yscale=.7]
\draw[dashed] (0,.5) node [above] {1} -- (1,.5) node [above] {2} -- (2,.5) -- (3,1) node [right] {4};
\draw[red, very thick] (2,.5) node [above] {3} -- (3,0) node [right] {5};
\draw (0,.5) circle 		[radius=.07];
\draw[fill] (1,.5) circle 		[radius=.11];
\draw[red] (2,.5) circle 		[radius=.07];
\draw[fill] (3,1) circle 		[radius=.11];
\draw[red] (3,0) circle 	[radius=.07];
\end{tikzpicture}
&
\begin{tikzpicture}[xscale=.7,yscale=.7]
\draw[dashed] (0,.5) node [above] {1} -- (1,.5) node [above] {2} -- (2,.5);
\draw[red, very thick] (3,1) node [right] {4} -- (2,.5) node [above] {3} -- (3,0) node [right] {5};
\draw (0,.5) circle 		[radius=.07];
\draw[fill] (1,.5) circle 		[radius=.11];
\draw[red] (2,.5) circle 		[radius=.07];
\draw[red] (3,1) circle 		[radius=.07];
\draw[red] (3,0) circle 	[radius=.07];
\end{tikzpicture}
&
\begin{tikzpicture}[xscale=.7,yscale=.7]
\draw[dashed] (0,.5) node [above] {1} -- (1,.5);
\draw[dashed] (2,.5) -- (3,1) node [right] {4};
\draw[red, very thick] (1,.5) node [above] {2} -- (2,.5) node [above] {3} -- (3,0) node [right] {5};
\draw[fill] (0,.5) circle 		[radius=.11];
\draw[red] (1,.5) circle 		[radius=.07];
\draw[red] (2,.5) circle 		[radius=.07];
\draw[fill] (3,1) circle 		[radius=.11];
\draw[red] (3,0) circle 	[radius=.07];
\end{tikzpicture}
&
\begin{tikzpicture}[xscale=.7,yscale=.7]
\draw[dashed] (2,.5) -- (3,1) node [right] {4};
\draw[red, very thick] (0,.5) node [above] {1} -- (1,.5) node [above] {2} -- (2,.5) node [above] {3} -- (3,0) node [right] {5};
\draw[red] (0,.5) circle 		[radius=.07];
\draw[red] (1,.5) circle 		[radius=.07];
\draw[red] (2,.5) circle 		[radius=.07];
\draw[fill] (3,1) circle 		[radius=.11];
\draw[red] (3,0) circle 	[radius=.07];
\end{tikzpicture}
\\
\hline
\end{array}
\end{equation*}
Here solid (red) segments form the subdiagram of type $\rA_d$ containing vertex~$5$ as an end-point, 
and solid (black) vertices form the minimal set such that $\rA_d$ is a connected component of its complement.
This gives~\eqref{eq:fk-x}.
Furthermore, unwinding the Tits construction explained in~\cite[Section~4]{Landsberg2003} 
we obtain the descriptions of linear spaces on $X$ in the theorem.
\end{proof}

\subsection{Linear sections and their derived categories}

The main characters of this paper are linear sections of~$X$.
Let 
\begin{equation*}
K \subset \SS^\vee
\end{equation*}
be a vector subspace of dimension $k$ and let
\begin{equation*}
K^\perp := \Ker(\SS \to K^\vee) \subset \SS
\end{equation*}
be its orthogonal complement (of codimension $k$ and dimension $16 - k$).
We define 
\begin{equation}
X_K := X \times_{\P(\SS)} \P(K^\perp)
\qquad\text{and}\qquad
X^\vee_K := X^\vee \times_{\P(\SS^\vee)} \P(K)
\end{equation} 
to be the corresponding linear sections of the spinor tenfold and its projective dual.
If the intersections are dimensionally transverse, we have
\begin{equation*}
\dim X_K = 10 - k,
\qquad
\dim X^\vee_K = 10 - (16 - k) = k - 6
\end{equation*}
with the convention that the dimension of an empty set is an arbitrary negative number.
The following simple observation is extremely useful.

\begin{lemma}
\label{lemma:smoothness}
A linear section $X_K$ is smooth and dimensionally transverse if and only if  $X^\vee_K$ is smooth and dimensionally transverse.
In particular, for $k = \dim K \le 5$ a linear section $X_K$ is smooth and dimensionally transverse 
if and only if $X^\vee_K = \varnothing$.
\end{lemma}
\begin{proof}
The proof is the same as in~\cite[Proposition~2.24]{debarre2015gushel}.
\end{proof}

In what follows we frequently abbreviate ``smooth and dimensionally transverse'' to just ``smooth''.

The tautological bundle $\cU$ and the structure sheaf $\cO_X$ give a very nice exceptional collection on $X$.
We write $\bD(X)$ for the bounded derived category of coherent sheaves on $X$.

\begin{proposition}[{\cite[Sections~6.2]{k2006hyperplane}}]
\label{proposition:db-x}
There is a full exceptional collection in $\bD(X)$ of the form
\begin{equation*}
\bD(X) = \langle \cO_X, \cU^\vee, \cO_X(1), \cU^\vee(1), \dots, \cO_X(7), \cU^\vee(7) \rangle.
\end{equation*}
\end{proposition}

This exceptional collection is Lefschetz and rectangular in the terminology of~\cite{k2007hpd,icm2014}, 
which just means that it consists of several twists of the starting block $\langle \cO_X, \cU^\vee \rangle \subset \bD(X)$.
Moreover, the main result of~\cite[Section~6.2]{k2006hyperplane} (see also~\cite[Theorem~5.5]{icm2014}) ensures that 
the classical projective duality between the spinor varieties~$X \subset \P(\SS)$ and $X^\vee \subset \P(\SS^\vee)$ extends
to a homological projective duality.
The main theorem of homological projective duality~\cite[Theorem~6.3]{k2007hpd} implies the following set 
of semiorthogonal decompositions relating derived categories of $X_K$ and $X^\vee_K$.

\begin{theorem}
\label{theorem:db-xk}
Assume $X_K$ and $X^\vee_K$ are both dimensionally transverse and $k = \dim K \le 8$.
Denote by~$\cU_{X_K}$ the restriction $\cU\vert_{X_K}$ of the tautological bundle.
Then there is a semiorthogonal decomposition
\begin{equation*}
\bD(X_K) = \langle  \bD(X^\vee_K), \cO_{X_K}, \cU^\vee_{X_K}, \dots, \cO_{X_K}(7- k), \cU_{X_K}^\vee(7-k) \rangle.
\end{equation*}
Moreover, for $k \le 7$ we have 
\begin{equation}
\label{eq:ext-o-u}
\Ext^\bullet(\cO_{X_K},\cU^\vee_{X_K}) \cong \Ext^\bullet(\cO_X,\cU^\vee) \cong \rV,
\qquad 
\Ext^\bullet(\cO_{X_K},\cU_{X_K}(1)) \cong \Ext^\bullet(\cO_X,\cU(1)) \cong \SS^\vee.
\end{equation}
In particular, the subcategory in $\bD(X_K)$ generated by the exceptional pair $\cO_{X_K},\cU^\vee_{X_K}$ is equivalent 
to the subcategory in $\bD(X)$ generated by $\cO_X$ and $\cU^\vee$ via the restriction functor.
\end{theorem}

Smoothness of $X_K$ and $X^\vee_K$ is unnecessary for the theorem, but we will usually assume it below.

Let us spell out what the above semiorthogonal decompositions tell:
\begin{itemize}
\item 
For $0 \le k \le 5$ the dimensional transversality assumption ensures that $X_K$ is a Fano variety of dimension $10 - k$ and $X^\vee_K = \varnothing$. 
Therefore the semiorthogonal decomposition reduces just to an exceptional collection of length $16 - 2k$
\begin{equation}
\label{eq:db-xk-05}
\bD(X_K) = \langle \cO_{X_K}, \cU_{X_K}^\vee, \dots, \cO_{X_K}(7- k), \cU_{X_K}^\vee(7-k) \rangle, 
\end{equation}
that can be considered as a reduced replica of the original collection.
\item 
For $k = 6$ the dimensional transversality assumption ensures that $X_K$ is a Fano fourfold and~$X^\vee_K$ is a finite scheme of length~12. 
Assuming also that $X^\vee_K$ is reduced (by Lemma~\ref{lemma:smoothness} this is equivalent to smoothness of~$X_K$), 
we obtain a semiorthogonal decomposition
\begin{equation}
\label{eq:db-xk-6}
\bD(X_K) = \langle \cE_1, \cE_2, \dots, \cE_{12}, \cO_{X_K}, \cU_{X_K}^\vee, \cO_{X_K}(1), \cU_{X_K}^\vee(1) \rangle, 
\end{equation}
where $\cE_1$, \dots, $\cE_{12}$ is a completely orthogonal exceptional collection.
In fact, one can check that each of $\cE_i$ is a vector bundle of rank 2.
\item 
For $k = 7$ the dimensional transversality assumption ensures that $X_K$ is a Fano threefold and~$X^\vee_K$ is a curve of arithmetic genus~7.
The semiorthogonal decomposition takes the form
\begin{equation}
\label{eq:db-xk-7}
\bD(X_K) = \langle \bD(X^\vee_K), \cO_{X_K}, \cU_{X_K}^\vee \rangle.
\end{equation}
In the smooth case one can check that $X^\vee_K$ is the moduli space of rank 2 vector bundles on $X_K$ and 
the embedding of the derived category is given by the Fourier--Mukai functor with kernel the universal bundle, 
see~\cite[Corollary~2.5, Theorem~4.4]{k2005v12}.
\item 
For $k = 8$ the dimensional transversality assumption ensures that both $X_K$ and~$X^\vee_K$ are polarized K3 surfaces of degree 12.
The semiorthogonal decomposition then reduces to an equivalence of categories
\begin{equation}
\label{eq:db-xk-8}
\bD(X_K) \cong \bD(X^\vee_K).
\end{equation}
In the smooth case this is the classical equivalence discovered by Mukai (\cite[Example~1.3]{mukai1999})); 
the surface $X^\vee_K$ again can be identified with the moduli space of rank 2 
vector bundles on $X_K$ and the equivalence of the derived category is given by the Fourier--Mukai functor with kernel the universal bundle.
\end{itemize}

The semiorthogonal decompositions of Theorem~\ref{theorem:db-xk} have many consequences for geometry of the varieties involved in them.
The simplest of these is the computation of the Grothendieck group:
\begin{equation}
\label{eq:k0-xk}
\rk K_0(X_K) = 
\begin{cases}
16 - 2k, & \text{for $0 \le k \le 5$},\\
16, & \text{for $k = 6$},
\end{cases}
\end{equation}
(the first line follows from~\eqref{eq:db-xk-05} and the second from~\eqref{eq:db-xk-6}) as soon as $X_K$ is smooth.

The next result is also quite useful.

\begin{corollary}
\label{corollary:isomorphism-conjugation}
Let $X_{K_1}$, $X_{K_2}$ be dimensionally transverse linear sections of $X$ of codimension $k \le 7$.
If $X_{K_1} \cong X_{K_2}$ there  is an element $g \in \Spin(\rV)$ such that $g(X_{K_1}) = X_{K_2}$ and $g(K_1) = K_2$.
\end{corollary}
\begin{proof}
Let $\varphi \colon X_{K_1} \to X_{K_2}$ be an isomorphism.
By Lefschetz theorem both $\Pic(X_{K_1})$ and $\Pic(X_{K_2})$ are generated by the restrictions $H_1$ and $H_2$ of the hyperplane class of $X \subset \P(\SS)$.
Therefore 
\begin{equation*}
\varphi^*(\cO_{X_{K_2}}(H_2)) \cong \cO_{X_{K_1}}(H_1).
\end{equation*}
Choosing such an isomorphism, we obtain an isomorphism $\bar\varphi$ between the vector spaces
\begin{equation*}
K_1^\perp \cong H^0(X_{K_1},\cO_{X_{K_1}}(H_1))^\vee \cong H^0(X_{K_2},\cO_{X_{K_2}}(H_2))^\vee \cong K_2^\perp
\end{equation*}
such that the diagram
\begin{equation*}
\xymatrix{
X_{K_1} \ar[r]^\varphi \ar[d] &
X_{K_2} \ar[d]
\\
\P(K_1^\perp) \ar[r]^{\bar{\varphi}} &
\P(K_2^\perp)
}
\end{equation*}
is commutative, where 
the vertical arrows are the natural embeddings.
By using the standard identification of the normal bundle of $X$ (see also Corollary~\ref{corollary:normal-x} below)
and the dimension transversality of $X_{K_i}$, we deduce an isomorphism
\begin{equation*}
\cU_{X_{K_1}}(2H_1) \cong 
\cN_{X_{K_1}/\P(K_1^\perp)} \cong
\varphi^*\cN_{X_{K_2}/\P(K_2^\perp)} \cong
\varphi^*\cU_{X_{K_2}}(2H_2).
\end{equation*}
Finally, the isomorphism $\cU_{X_{K_1}} \cong \varphi^*\cU_{X_{K_2}}$
with~\eqref{eq:ext-o-u} taken into account, produces an isomorphism
\begin{equation*}
\rV \cong 
H^0(X_{K_1},\cU_{X_{K_1}}^\vee) \cong
H^0(X_{K_2},\cU_{X_{K_2}}^\vee) \cong
\rV
\end{equation*}
that is an element $g_\rV \in \GL(\rV)$.
Since the zero locus of a global section $v \in \rV$ of $\cU^\vee_{X_{K_i}}$ is empty 
if and only if $v$ lies on the quadric $\rQ \subset \P(\rV)$ (see Remark~\ref{remark:qv-ogr-zero}),
it follows that $g_\rV$ preserves the quadric $\rQ$, i.e.,~$g_\rV \in \GO(\rV)$.
Finally, it is easy to see that the induced action of $g_\rV$ on $\P(\SS)$ takes $X_{K_1}$ to $X_{K_2}$ and~$K_1$ to~$K_2$.
The element $g$ can be defined as any lift to $\Spin(\rV)$ of the image of $g_\rV$ in $\PSO(\rV)$.
\end{proof}

In what follows, to unburden notation we will denote the restriction $\cU_{X_K}$ simply by $\cU$.

\subsection{Rational Chow motives}

For a smooth projective variety $Y$ we denote by $\bM(Y)$ its Chow motive and by $\bM_\QQ(Y)$ its Chow motive with rational coefficients.

\begin{corollary}
\label{corollary:motives-q-xk-06}
Let $X_K$ be a smooth linear section of the spinor tenfold of codimension $k \le 6$.
Then the rational Chow motive of $X_K$ is of Lefschetz type:
\begin{equation*}
\bM_\QQ(X_K) = 
\begin{cases}
1 \oplus \bL_\QQ \oplus \bL_\QQ^2 \oplus 2\bL_\QQ^3 \oplus 2\bL_\QQ^4 \oplus 2\bL_\QQ^5 \oplus 2\bL_\QQ^6 \oplus 2\bL_\QQ^7 \oplus \bL_\QQ^8 \oplus \bL_\QQ^9 \oplus \bL_\QQ^{10},
& \text{for $k = 0$},\\
1 \oplus \bL_\QQ \oplus \bL_\QQ^2 \oplus 2\bL_\QQ^3 \oplus 2\bL_\QQ^4 \oplus 2\bL_\QQ^5 \oplus 2\bL_\QQ^6 \oplus \bL_\QQ^7 \oplus \bL_\QQ^8 \oplus \bL_\QQ^9,
& \text{for $k = 1$},\\
1 \oplus \bL_\QQ \oplus \bL_\QQ^2 \oplus 2\bL_\QQ^3 \oplus 2\bL_\QQ^4 \oplus 2\bL_\QQ^5 \oplus \bL_\QQ^6 \oplus \bL_\QQ^7 \oplus \bL_\QQ^8,
& \text{for $k = 2$},\\
1 \oplus \bL_\QQ \oplus \bL_\QQ^2 \oplus 2\bL_\QQ^3 \oplus 2\bL_\QQ^4 \oplus \bL_\QQ^5 \oplus \bL_\QQ^6 \oplus \bL_\QQ^7,
& \text{for $k = 3$},\\
1 \oplus \bL_\QQ \oplus \bL_\QQ^2 \oplus 2\bL_\QQ^3 \oplus \bL_\QQ^4 \oplus \bL_\QQ^5 \oplus \bL_\QQ^6,
& \text{for $k = 4$},\\
1 \oplus \bL_\QQ \oplus \bL_\QQ^2 \oplus \bL_\QQ^3 \oplus \bL_\QQ^4 \oplus \bL_\QQ^5,
& \text{for $k = 5$},\\
1 \oplus \bL_\QQ \oplus 12\bL_\QQ^2 \oplus \bL_\QQ^3 \oplus \bL_\QQ^4,
& \text{for $k = 6$}.
\end{cases}
\end{equation*}
Moreover, $\CH^i(X_K) \otimes \QQ \cong \QQ^{n_i}$, where the dimensions $n_i$ is equal to the multiplicity of the corresponding Lefschetz motive $\bL_\QQ^i$ in $\bM_\QQ(X_K)$.
\end{corollary}

\begin{proof}
The motive of $X_K$ is of Lefschetz type by ~\cite[Theorem~1.1]{marcolli2015}, see also a simplified proof in~\cite[Proposition~2.1]{galkin}.
The multiplicities in case $k = 0$ can be read off the Hodge diamond of $X$, which is well known.
Alternatively, one can argue as follows.
Clearly, we have
\begin{equation*}
\bM_\QQ(X_K) = \bigoplus_{i = 0}^{10} n_i \bL_\QQ^i,
\end{equation*}
where $1 = n_0 \le n_1 \le n_2 \le n_3 \le n_4 \le n_5 \ge n_6 \ge n_7 \ge n_8 \ge n_9 \ge n_{10} = 1$.
Moreover, $n_{11 - i} = n_i$ by Poincar\'e duality,
and~$\sum n_i = 16$, since this is the rank of the Grothendieck group $K_0(X)$, see~\eqref{eq:k0-xk}.
So, to determine $n_i$ it is enough to check that $n_5 \le 2$.
Assume on a contrary that $n_5 \ge 3$.
Then for a smooth hyperplane section $X_1 \subset X$, since $\bM_\QQ(X_1)$ is of Lefschetz type,
by Lefschetz hyperplane theorem it follows that 
\begin{equation*}
\bM_\QQ(X_1) = \left( \bigoplus_{i = 0}^4 n_i \bL_\QQ^i \right) \oplus \left( \bigoplus_{i = 6}^{10} n_i \bL_\QQ^{i-1} \right).
\end{equation*}
From the assumption we have $\sum_{i=0}^4 n_i + \sum_{i=6}^{10} n_i = \sum_{i=0}^{10} n_i - n_5 \le 16 - 3 = 13 < 14 = \rk K_0(X_1)$, 
see~\eqref{eq:k0-xk}, a contradiction.
It proves that $n_5 \le 2$, and thus gives the required expression for $\bM_\QQ(X)$.

Next, the description of $\bM_\QQ(X_K)$ for $1 \le k \le 6$ follows from a combination of Lefschetz hyperplane theorem 
(that allows to determine the multiplicities of all Lefschetz motives, except possibly for the middle one)
with~\eqref{eq:k0-xk},
which allows to determine the multiplicity of the middle Lefschetz motive when~$k$ is even.
The result for the Chow groups follows immediately from the expression for the motive.
\end{proof}

\section{The blowup of the spinor tenfold}
\label{section:blowup-x}

In this section we discuss a description of the blowup of the projective space $\P(\SS)$ along the spinor tenfold $X$ 
and its consequences for linear sections of $X$.

\subsection{The blowup of the space of spinors along $X$}

Recall the notation of Section~\ref{subsection:notation}.
The next result can be extracted from~\cite[Theorem~III.3.8(4)]{zak1993book}.
For completeness, we provide a proof using the blowup lemma.

\begin{proposition}
\label{proposition:blowup-x}
Let $X \subset \P(\SS)$ be the spinor tenfold, let $\rQ \subset \P(\rV)$ be the corresponding $8$-dimensional quadric,
and let $\cS = \cS_8$ be the spinor bundle on $\rQ$.
The left part of~\eqref{diagram:x-q-x} extends to a diagram
\begin{equation}
\label{diagram:blowup-x-s}
\vcenter{\xymatrix@C=5em{
\Bl_X(\P(\SS)) \ar@{=}[r]^\sim \ar[d]_{f_\SS} &
\P_\rQ(\cS_8) \ar[dl] &
\cQ \ar@{_{(}->}[l] \ar[dl]  \ar[dr] 
\\
\P(\SS) &
X \ar@{_{(}->}[l] &&
\rQ
}}
\end{equation}
Under the isomorphism $\Bl_X(\P(\SS)) \cong \P_\rQ(\cS_8)$ the exceptional divisor of the blowup morphism $f_\SS$ 
coincides with the family of quadrics $\cQ \subset \P_\rQ(\cS_8)$.
Moreover, if $H_\SS$ is the hyperplane class of $\P(\SS)$ and $H_\rQ$ is the hyperplane class of~$\rQ$ 
then the class of the divisor $\cQ$ in $\Pic(\P_\rQ(\cS_8))$ can be expressed as
\begin{equation}
\label{eq:blowup-x-picard-relation}
\cQ \sim 2H_\SS - H_\rQ.
\end{equation}
\end{proposition}

\begin{proof}
The canonical embedding $\cS_8 \hookrightarrow \SS \otimes \cO_\rQ$ 
(we denote the corresponding quotient bundle by~$\SS/\cS_8$)
induces a $\Spin(\rV)$-equivariant morphism
\begin{equation}
\label{eq:blowup-spinor}
f_\SS \colon \P_\rQ(\cS_8) \to \P(\SS).
\end{equation}
We claim that this morphism is a blowup along the spinor variety $X$.

First, let us check that the morphism $f_\SS$ is birational.
Indeed, the image of $\P_\rQ(\cS_8)$ in $\rQ \times \P(\SS)$ is the zero locus of a global section 
of the vector bundle $(\SS/\cS_8) \boxtimes \cO(1) \cong \cS_{8,-}^\vee \boxtimes \cO(1)$
(the isomorphism follows from a combination of Lemma~\ref{lemma:spinor-sequence} and~\ref{lemma:cs-dual}).
Therefore, the fibers of~\eqref{eq:blowup-spinor} are the zero loci of global sections of $\cS_{8,-}^\vee$.
Since $\cS_{8,-}^\vee$ is globally generated of rank~8 with top Chern class equal to 1 (see~\cite[Remark~2.9]{ottaviani}) 
it follows that the general fiber is a single point, hence $f_\SS$ is birational.

Next, we apply the blowup lemma to the morphism $f_\SS \colon \P_\rQ(\cS_8) \to \P(\SS)$.
We have~$\Pic(\P_\rQ(\cS_8)) \cong \ZZ^2$, while~$\Pic(\P(\SS)) \cong \ZZ$, so the relative Picard number for $f_\SS$ equals~1.
On the other hand, we have a natural embedding $\cQ \hookrightarrow \P_\rQ(\cS_8)$ (see~\eqref{eq:cq-in-ps}),
and its composition with the map $f_\SS$ is defined by the pullback 
of the spinor line bundle from $X$, hence the middle parallelogram in~\eqref{diagram:blowup-x-s} is commutative.
Since $\cQ \subset \P_\rQ(\cS_8)$ is a divisor and its image $f_\SS(\cQ) = X \subset \P(\SS)$ is smooth,
we conclude by Lemma~\ref{lemma:blowup} that $f_\SS$ is the blowup of $X$ and $\cQ$ is its exceptional divisor.


Finally, the equation of the relative quadric $\cQ \subset \P_\rQ(\cS_8)$ is induced 
by the self-duality isomorphism~$\cS_8^\vee \cong \cS_8(H_\rQ)$ (see Lemma~\ref{lemma:cs-dual}), 
which means that $\cQ = 2H_\SS - H_\rQ$, thus proving~\eqref{eq:blowup-x-picard-relation}.
\end{proof}

This result has several useful consequences for the geometry of $X$.
First, it gives a simple proof of transitivity of $\Spin(\rV)$-action on $\P(\SS) \setminus X$
(which is well-known, see, e.g., \cite[Proposition~2.1]{manivel2017double},
\cite[Remark~2.13(1)]{fu2015special}, and also~\cite[Proposition~31]{Sato1977}).

\begin{corollary}
\label{corollary:transitivity}
The action of $\Spin(\rV)$ on $\P(\SS) \setminus X$ is transitive.
\end{corollary}
\begin{proof}
The blowup morphism induces a $\Spin(\rV)$-equivariant isomorphism
\begin{equation*}
\P(\SS) \setminus X \cong \P_\rQ(\cS_8) \setminus \cQ.
\end{equation*}
Since the action of $\Spin(\rV)$ on $\rQ$ is transitive, it is enough to check that the stabilizer of a point $v \in \rQ$ in $\Spin(\rV)$
acts transitively on $\P(\cS_{8,v}) \setminus \cQ_v$.
But the stabilizer contains $\Spin(v^\perp/v) \cong \Spin(\cS_{8,v})$ as a subgroup, hence the claim.
\end{proof}

As another consequence, we deduce a resolution for the structure sheaf of $\cO_X$ on $\P(\SS)$, 
which was also deduced by other tools earlier (see~\cite[Section~5.1]{weyman2012geometry}).

\begin{corollary}
\label{corollary:resolution-x}
There is an exact sequence 
\begin{equation*}
0 \to \cO(-8) \to \rV(-6) \to \SS^\vee(-5) \to \SS(-3) \to \rV(-2) \to \cO \to \cO_X \to 0.
\end{equation*}
In particular, $X$ is an intersection of quadrics and $H^0(\P(\SS),I_X(2)) \cong \rV$,
an isomorphism of $\Spin(\rV)$-representations.
\end{corollary}
\begin{proof}
It was explained in the proof of Proposition~\ref{proposition:blowup-x} that the projective bundle $\P_\rQ(\cS_8)$ can be written inside the product $\rQ \times \P(\SS)$ as the zero locus 
of a global section of the vector bundle $\cS_{8,-}^\vee \boxtimes \cO(1)$.
Consequently, its structure sheaf has a Koszul resolution
\begin{equation*}
0 \to \bw8\cS_{8,-} \otimes \cO(-8) \to \bw7\cS_{8,-} \otimes \cO(-7) \to \dots \to 
\cS_{8,-} \otimes \cO(-1) \to \cO \to \cO_{\P_\rQ(\cS_8)} \to 0.
\end{equation*}
Since $\cO_{\P_\rQ(\cS_8)}(-\cQ) \cong \cO_{\P_\rQ(\cS_8)}(H_\rQ - 2H_\SS)$ by~\eqref{eq:blowup-x-picard-relation}, 
the pushforward of $\cO_{\P_\rQ(\cS_8)}(H_\rQ)$ provides a resolution 
for the twisted ideal sheaf $I_X(2)$.
The twists of wedge powers of $\cS_{8,-}$ are direct sums of irreducible homogeneous vector bundles on $\rQ$, 
and the corresponding weights of the group $\Spin(\rV)$ are listed in the second lines of the two tables below:
\begin{equation*}
\begin{array}{||c||c||c||c||c|c||}
\hline
\cO(1) & \cS_{8,-}(1) & \bw2\cS_{8,-}(1) & \bw3\cS_{8,-}(1) & \multicolumn{2}{||c||}{\bw4\cS_{8,-}(1)} 
\\
\hline
0 & \omega_4 - \omega_1 & \omega_3 - 2\omega_1 & \omega_2 + \omega_5 - 3\omega_1 & 
2\omega_2 - 4\omega_1 & 2\omega_5 - 3\omega_1 
\\
\hline
\rV & \SS & 0 & \SS^\vee[-1] & \rV[-1] & 0 
\\
\hline
\end{array}
\end{equation*}
\begin{equation*}
\begin{array}{||c||c||c||c||c||}
\hline
\bw5\cS_{8,-}(1) & \bw6\cS_{8,-}(1) & \bw7\cS_{8,-}(1) & \bw8\cS_{8,-}(1)
\\
\hline
\omega_2 + \omega_5 - 4\omega_1 & \omega_3 - 4\omega_1 & \omega_4 - 4\omega_1 & -4\omega_1 \\
\hline
0 & \Bbbk[-2] & 0 & 0 \\
\hline
\end{array}
\end{equation*}
The third lines of the tables list the cohomology on $\rQ$ of the corresponding bundles 
(computed via Borel--Bott--Weil theorem) with the cohomology degree in brackets.
As a result, we obtain the required resolution.
\end{proof}

Yet another useful consequence of the proposition is the following well-known isomorphism.

\begin{corollary}
\label{corollary:normal-x}
The normal bundle of the spinor tenfold has the following description
\begin{equation*}
\cN_{X/\P(\SS)} \cong \bw4\cU^\vee \cong \cU(2),
\end{equation*}
where $\cU$ is the resriction of the tautological bundle.
\end{corollary}
\begin{proof}
Since the exceptional divisor of a blowup is isomorphic to the projectivization of the normal bundle 
and $\cQ \cong \P_X(\cU)$ by~\eqref{eq:cq-pu},
it follows that the normal bundle is isomorphic to a twist of~$\cU$. 
On the other hand, by the adjunction formula $\det \cN_{X/\P(\SS)} \cong \cO_X(8)$, while $\det \cU \cong \cO_X(-2)$ by~\eqref{eq:detu-s}, 
so the required twist is given by $\cO_X(2)$.
\end{proof}

Of course, the argument of Proposition~\ref{proposition:blowup-x} can be applied to the blowup of $\P(\SS^\vee)$ along $X^\vee$,
with a completely analogous result (or one can formally apply an outer automorphism of $\Spin(\rV)$ to the diagram~\eqref{diagram:blowup-x-s}).
On the next diagram we merge the resulting digram with~\eqref{diagram:blowup-x-s}:
\begin{equation}
\label{diagram:blowups-x-xvee}
\vcenter{\xymatrix{
&&
\cQ \ar[ddll] \ar[d] &&
\cQ_- \ar[ddrr] \ar[d] 
\\
&&
\P_\rQ(\cS_8) \ar[dl]^p \ar[dr]_q &&
\P_\rQ(\cS_{8,-}) \ar[dl]^{q_-} \ar[dr]_{p_-} \\
X \ar[r] &
\P(\SS) &&
\rQ &&
\P(\SS^\vee) &
X^\vee \ar[l]
}}
\end{equation} 
The rational map $\gamma := q_- \circ p_- \colon \P(\SS^\vee) \dashrightarrow \rQ$ is essential for the paper.

\subsection{Quadrics on the spinor tenfold}

One can also use the Proposition~\ref{proposition:blowup-x} to describe quadrics on~$X$.
Denote by 
\begin{equation*}
G_d(X) = \Hilb^{(t+1)\cdots(t+d-1)(2t+d)/d!}(X)
\end{equation*}
the Hilbert scheme of quadrics of dimension $d$ in $X$.

\begin{corollary}
\label{corollary:quadrics-in-xk}
Assume $Z \subset X$ is a quadric of dimension $d$.
Then $d \le 6$ and either there exists a unique point $v \in \rQ$ and 
a unique linear subspace $\P^{d+1} \subset \P(\cS_{8,v}) = q^{-1}(v)$ such that 
\begin{equation*}
Z = p(\P^{d+1} \cap \cQ_v);
\end{equation*}
or $d \le 3$ and there is a unique linear space $\Pi^{d+1} \subset X$ such that $Z \subset \Pi^{d+1}$.
\end{corollary}
\begin{proof}
Let $\Pi := \langle Z \rangle \subset \P(\SS)$ be the linear span of $Z$ in $\P(\SS)$.
If $\Pi$ is contained in $X$ there is nothing to prove (since by Theorem~\ref{theorem:linear-spaces} 
the maximal linear space in $X$ has dimension at most~4, the dimension of such $Z$ is bounded by~3), 
so assume $\Pi \not\subset X$.

Since $X$ is an intersection of quadrics (Corollary~\ref{corollary:resolution-x}), 
we have a scheme theoretic equality $\Pi \cap X = Z$.
Therefore, the map $q$ contracts the strict transform~$\widetilde\Pi$ of~$\Pi$ in $\Bl_X(\P(\SS))$ to a point.
Denoting this point by $v$ we see that $\widetilde\Pi \subset q^{-1}(v) = \P(\cS_{8,v})$ and 
$Z = p(\widetilde\Pi) \cap X = p(\widetilde\Pi \cap \cQ_v)$.
\end{proof}

\begin{remark}
In fact, one can push forward the results of the Corollary~\ref{corollary:quadrics-in-xk} to get a description
of the Hilbert scheme of quadrics on $X$ as follows:
\begin{align*}
\Bl_{\P_{\OGr(3,\rV)}(\Sym^2(\cS_2))}(G_0(X)) 					&\cong \Bl_{\OGr_{\rQ}(2,\cS_8)}(\Gr_{\rQ}(2,\cS_8));\\
\Bl_{\P_{\OFl_-(2,5;\rV)}(\Sym^2(\cU_{5,-}/\cU_2))}(G_1(X)) 				&\cong \Bl_{\OGr_{\rQ}(3,\cS_8)}(\Gr_{\rQ}(3,\cS_8));\\
\Bl_{\P_{\OGr(2,\rV)}(\Sym^2(\cS_4))}(G_2(X)) 					&\cong \Bl_{\OGr_{\rQ}(4,\cS_8)}(\Gr_{\rQ}(4,\cS_8));\\
G_3(X)									&\cong \Gr_{\rQ}(5,\cS_8) \sqcup \P_{\OGr_-(5,\rV)}(\Sym^2(\cU_{5,-}));\\
G_4(X)									&\cong \Gr_{\rQ}(6,\cS_8);\\
G_5(X)									&\cong \Gr_{\rQ}(7,\cS_8);\\
G_6(X)									&\cong \rQ.
\end{align*}
However, we do not need these results, so we leave this as an exercise. 
\end{remark}

In particular, we see that maximal (6-dimensional) quadrics in $X$ are all of the form $\cQ_v$ for $v \in \rQ$.

Later we will need a description of intersections of maximal quadrics with maximal linear spaces on the spinor tenfold $X$.

\begin{lemma}
\label{lemma:linear-quadric-intersection}
Let $\Pi^4_{U_{5,-}} \subset X$ be a linear $4$-space and let $\cQ_v \subset X$ be a $6$-dimensional quadric on the spinor tenfold $X$.
Then
\begin{equation*}
\Pi^4_{U_{5,-}} \cap \cQ_v = 
\begin{cases}
\Spec(\Bbbk), &  \text{if $v \not \in U_{5,-}$},\\
\Pi^3_{v,U_{5,-}}, & \text{otherwise.}
\end{cases}
\end{equation*}
\end{lemma}

\begin{proof}
Recall that $\cQ_v = \OGr_+(4,v^\perp/v)$.
First, assume that $v \not\in U_{5,-}$.
Then $v$ is not orthogonal to~$U_{5,-}$ (since $U_{5,-}^\perp = U_{5,-}$), hence the intersection $U_{5,-} \cap v^\perp$ is 4-dimensional.
On the other hand, if $U_{5}$ contains~$v$ and has a 4-dimensional intersection with $U_{5,-}$, this intersection should be contained in $v^\perp$, 
hence $U_{5}$ is equal to the linear span $\langle v, U_{5,-} \cap v^\perp \rangle$, and this is the only intersection point of $\Pi^4_{U_{5,-}}$ and $\cQ_v$.

Now assume that $v \in U_{5,-}$.
Then isotropic subspaces $U_5 \subset \rV$ that have a 4-dimensional intersection with~$U_{5,-}$ 
are parameterized by the 4-space~$\Pi^4_{U_{5,-}} = \Gr(4,U_{5-})$, and 
those of them that contain $v$ are parameterized by the 3-space $\Pi^3_{v,U_{5,-}} = \Gr(3,U_{5,-}/v)$.
\end{proof}

\subsection{Blowups of linear sections}

The diagram~\eqref{diagram:blowups-x-xvee} induces a similar diagram for linear sections of the spinor tenfold.
To state the result we introduce the following notation.

Let $K \subset \SS^\vee$ be a subspace of dimension $k$.
Consider the composition of morphisms of sheaves on $\rQ$:
\begin{equation}
\label{eq:map-k-s}
\sigma_K \colon K \otimes \cO_{\rQ} \hookrightarrow \SS^\vee \otimes \cO_{\rQ} \twoheadrightarrow \cS_8^\vee,
\end{equation} 
where the first morphism is induced by the embedding $K \hookrightarrow \SS^\vee$ and the second is 
the evaluation morphism for the natural identification $H^0(\rQ,\cS_8^\vee) \cong \SS^\vee$ (see Section~\ref{subsection:spinors}).
We denote by 
\begin{equation*}
\rQ = \fD_{\ge 0}(\sigma_K) \supset \fD_{\ge 1}(\sigma_K) \supset \fD_{\ge 2}(\sigma_K) \supset \dots 
\end{equation*}
the discriminant stratification of the quadric $\rQ$ by the corank strata of the morphism $\sigma_K$.
In other words, $\fD_{\ge c}(\sigma_K)$ is
the subscheme of $\rQ$ where the corank of $\sigma_K$ is at least $c$
(the ideal of this subscheme is generated by minors of the map~\eqref{eq:map-k-s} of size $k - s + 1$).
We also put $\fD_{c}(\sigma_K) := \fD_{\ge c}(\sigma_K) \setminus \fD_{\ge c+1}(\sigma_K)$.

\begin{proposition}
\label{proposition:blowup-xk}
Assume that $X_K$ and $X^\vee_K$ are dimensionally transverse linear sections of $X$ and $X^\vee$ respectively.
Then there is a diagram
\begin{equation}
\label{diagram:blowup-xk}
\vcenter{\xymatrix@C=2em{
&
& 
\Bl_{X_K}(\P(K^\perp)) \ar[dl]_p \ar[dr]^q
&&
\Bl_{X^\vee_K}(\P(K)) \ar[dl]_{q_-} \ar[dr]^{p_-}
\\
X_K \ar@{^{(}->}[r]
& 
\P(K^\perp) 
&&
\rQ
&&
\P(K)
&
X^\vee_K \ar@{_{(}->}[l]
}}
\end{equation}
where the maps $p$, $q$, $p_-$, and $q_-$ are the restrictions of the same named maps in~\eqref{diagram:blowups-x-xvee}.
The maps $p$ and~$p_-$ are the blowup maps, and the maps $q$ and $q_-$ are piecewise Zariski locally trivial fibrations 
with fibers over the stratum $\fD_c(\sigma_K) \subset \rQ$ isomorphic to $\P^{7 + c - k}$ and $\P^{c-1}$ respectively.
In particular,
\begin{equation*}
\fD_{\ge 1}(\sigma_K) = q_-(\Bl_{X^\vee_K}(\P(K))).
\end{equation*}
\end{proposition}

\begin{proof}
Consider the diagram~\eqref{diagram:blowups-x-xvee}.
The $p$-preimage in $\Bl_X(\P(\SS)) \cong \P_{\rQ}(\cS_8)$ of a hyperplane in $\P(\SS)$ is a relative hyperplane section of $q \colon \P_\rQ(\cS_8) \to \rQ$..
Further, under the transversality assumption we have
\begin{equation*}
p^{-1}(\P(K^\perp)) \cong \Bl_{X_K}(\P(K^\perp)),
\end{equation*}
and this is the zero locus in $\P_\rQ(\cS_8)$ of the natural section of the vector bundle $K^\vee \otimes \cO_{\P_\rQ(\cS_8)}(H_\SS)$ 
that corresponds to the morphisms $\sigma_K$, or more precisely, to its dual
\begin{equation}
\label{eq:map-s-k}
\sigma_K^\vee \colon \cS_8 \hookrightarrow \SS \otimes \cO_{\rQ} \twoheadrightarrow K^\vee \otimes \cO_{\rQ}.
\end{equation}
So, the fiber of the map $q$ in the diagram~\eqref{diagram:blowup-xk} over a point $v \in \rQ$ is the projectivization 
of the kernel of~$\sigma_K^\vee$ at $v$.
Therefore, $q$ is a piecewise Zariski locally trivial fibration over $\rQ$ with fiber isomorphic to~$\P^{7 + c - k}$ 
over the stratum $\fD_c(\sigma_K^\vee) = \fD_c(\sigma_K) \subset \rQ$.

On the other hand, the first map in~\eqref{eq:map-s-k} is a fiberwise monomorphism 
and by Lemma~\ref{lemma:spinor-sequence} its cokernel is the natural epimorphism $\SS \otimes \cO_{\rQ} \twoheadrightarrow \cS_{8,-}^\vee$,
while the second map in~\eqref{eq:map-s-k} is an epimorphism 
whose kernel is the natural fiberwise monomorphism $K^\perp \otimes \cO_{\rQ} \hookrightarrow \SS \otimes \cO_{\rQ}$.
Therefore the rank stratification for~\eqref{eq:map-s-k} coincides with the rank stratification for the composition
\begin{equation}
\label{eq:map-kperp-s}
\sigma_K^\perp \colon K^\perp \otimes \cO_{\rQ} \hookrightarrow \SS \otimes \cO_{\rQ} \twoheadrightarrow \cS_{8,-}^\vee
\end{equation}
of the above kernel and cokernel maps.
Taking the dual of~$\sigma_K^\perp$ and repeating the above arguments we conclude 
that $q_-$ is also a piecewise Zariski locally trivial fibration over $\rQ$ with fiber $\P^{c-1}$ 
over the stratum $\fD_c(\sigma^\perp_K) \subset \rQ$.
But we already checked that $\fD_c(\sigma^\perp_K) = \fD_c(\sigma_K)$.
Since $\fD_{\ge 1}(\sigma^\perp_K)$ is the locus of non-empty fibers of $q_-$, we conclude that $\fD_{\ge 1}(\sigma_K) = q_-(\Bl_{X^\vee_K}(\P(K)))$.
This completes the proof.
\end{proof}

The following particular case of the proposition will be used very extensively.

\begin{corollary}
\label{corollary:stratification-small-k}
Assume that $X_K$ is smooth of codimension $k = \dim K \le 5$ in $X$, so that $X_K^\vee = \varnothing$.
Then $\Bl_{X_K^\vee}(\P(K)) = \P(K)$ and the map $q_- \circ p_-^{-1}$ is a closed embedding
\begin{equation*}
\gamma := q_- \circ p_-^{-1} \colon \P(K) \hookrightarrow \rQ
\end{equation*}
such that $\gamma^*(\cO_\rQ(1)) \cong \cO_{\P(K)}(2)$.
The map $q$ has fibers $\P^{8-k}$ over $\gamma(\P(K))$ and $\P^{7-k}$ over its complement.
\end{corollary}
\begin{proof}
First, let us check that $\fD(\sigma_K)_{\ge 2} = \varnothing$.
By Proposition~\ref{proposition:blowup-xk} it is enough to check that no fiber of~$q_-$ contains a $\P^1$.
But if $\P^1 \subset q_-^{-1}(v) \subset \P(\cS_{8,v,-})$, we have $q_-^{-1}(v) \cap \cQ_{-,v} \ne \varnothing$, and hence
\begin{equation*}
\varnothing \ne p_-(q_-^{-1}(v) \cap \cQ_{-,v}) \subset X^\vee \cap \P(K) = X^\vee_K
\end{equation*}
which contradicts to the assumptions.
Thus, $q_- \colon \Bl_{X^\vee_K}(\P(K)) \to \rQ$ is a closed embedding, 
and since the map $p_- \colon \Bl_{X^\vee_K}(\P(K)) \to \P(K)$ is an isomorphism, 
the composition $\gamma = q_- \circ p_-^{-1} \colon \P(K) \to \rQ$ is a closed embedding.

Further, note that $\P(K) = \Bl_{X^\vee_K}(\P(K)) \subset \Bl_{X^\vee}(\P(\SS^\vee))$ 
does not intersect the exceptional divisor~\mbox{$\cQ_- \subset \P_\rQ(\cS_{8,-})$},
hence the class of $\cQ_-$ restricts trivially to $\P(K)$.
Taking into account the analogue of~\eqref{eq:blowup-x-picard-relation} for the right half of the diagram~\eqref{diagram:blowups-x-xvee},
we obtain an isomorphism $\gamma^*(\cO_\rQ(1)) \cong \cO_{\P(K)}(2)$.
The last part of the corollary follows immediately from Proposition~\ref{proposition:blowup-xk}.
\end{proof}

\begin{remark}
\label{remark:stratification-explicit}
Let us also spell out the conclusion of the proposition in the case when $X_K$ (and hence also~$X^\vee_K$) is 
smooth and dimensionally transverse of codimension $k \ge 6$.
\begin{itemize}
\item 
Assume $k = 6$.
Then $X^\vee_K$ is the set of 12 reduced points, and the map $q_-$ contracts the strict transforms of the 66 lines connecting these points.
Consequently, we have three strata: 
$\fD_2(\sigma_K)$ consists of 66 points (the images of the strict transforms of the lines),
$\fD_{\ge 1}(\sigma_K) = q_-(\Bl_{X^\vee_K}(\P(K))$, and $\fD_0(\sigma_K)$ is its open complement.
\item 
Assume $k = 7$.
Then $X^\vee_K$ is a smooth (canonical) curve of genus 7, and the map $q_-$ contracts the strict transforms of its secant lines.
Consequently, we have three strata: 
\mbox{$\fD_2(\sigma_K) \cong \Sym^2(X^\vee_K)$} is the image of the secant variety,
$\fD_{\ge 1}(\sigma_K) = q_-(\Bl_{X^\vee_K}(\P(K))$, and 
$\fD_0(\sigma_K)$ is its open complement.
\item 
Assume $k = 8$.
Then $X^\vee_K$ is a smooth K3 surface of degree 12, and the map $q_-$ contracts the strict transforms of its secant lines
and strict transforms of planes intersecting $X^\vee_K$ along a conic.
Consequently, we have at most four strata: 
$\fD_3(\sigma_K)$ is a finite number (possibly zero) of points (the images of the planes spanned by conics on $X^\vee_K$),
$\fD_{\ge 2}(\sigma_K)$ is the image of the secant variety (a contraction of~$\Sym^2(X^\vee_K)$),
$\fD_{\ge 1}(\sigma_K) = q_-(\Bl_{X^\vee_K}(\P(K))$, and $\fD_{0}(\sigma_K)$ is its open complement.
\end{itemize}
\end{remark}

\begin{remark}
In the case $k = 7$, considering the fibers of the map $q \colon \Bl_{X_K}(\P(K^\perp)) \to \rQ$ one can also show 
that $\fD_{\ge 2}(\sigma_K) \cong G_1(X_K)$, the Hilbert scheme of conics on the Fano threefold $X_K$.
In a combination with the observation of Remark~\ref{remark:stratification-explicit} this gives a geometric construction of an isomorphism 
\begin{equation*}
G_1(X_K) \cong \Sym^2(X^\vee_K)
\end{equation*}
that was originally proved in~\cite[Theorem~5.3]{k2005v12} by means of derived categories.
\end{remark}

The diagram~\eqref{diagram:blowup-xk} in some aspects is not too convenient, because the map $q$ is not flat.
In the case~$k \le 5$ it can be, however, transformed into a flat $\P^{7- k}$-bundle by an extra blowup.

\begin{proposition}
\label{proposition:double-blowup-xk}
Let $X_K$ be a smooth and dimensionally transverse linear section of $X$ of codimension~$k \le 5$,
so that~$X^\vee_K = \varnothing$.
Then there is a commutative diagram
\begin{equation}
\label{diagram:double-blowup-xk}
\vcenter{\xymatrix@C=1.2em{
&
& 
\Bl_{X_K}(\P(K^\perp)) \ar[dl]_p \ar[dr]^q
&&
\Bl_{q^{-1}(\gamma(\P(K)))}(\Bl_{X_K}(\P(K^\perp))) \ar[ll]_-{\tilde{r}} \ar[dr]^{\tilde{q}}
\\
X_K \ar@{^{(}->}[r]
& 
\P(K^\perp) 
&&
\rQ
&&
\Bl_{\gamma(\P(K))}(\rQ) \ar[ll]_-r
}}
\end{equation}
and the map $\tilde{q}$ is the projectivization of a vector bundle of rank $8 - k$.
\end{proposition}
\begin{proof}
The scheme theoretic preimage of the subscheme $\gamma(\P(K)) \subset \rQ$ in $\Bl_{X_K}(\P(K^\perp))$ is $q^{-1}(\gamma(\P(K)))$ (by definition), 
hence its scheme theoretic preimage in $\Bl_{q^{-1}(\gamma(\P(K)))}(\Bl_{X_K}(\P(K^\perp)))$ is the exceptional divisor of the blowup $\tilde{r}$.
Therefore, by the universal property of the blowup the composition $q \circ \tilde{r}$
factors through the blowup $\Bl_{\gamma(\P(K))}(\rQ)$, thus defining the map $\tq$, that makes the diagram commutative.
It remains to show that $\tq$ is the projectivization of a vector bundle.

Consider the morphisms $\sigma^\vee_K$ and $\sigma^\perp_K$ defined by~\eqref{eq:map-s-k} and~\eqref{eq:map-kperp-s}.
By the argument of Proposition~\ref{proposition:blowup-xk} and Corollary~\ref{corollary:stratification-small-k} 
the corank of $\sigma^\vee_K$ is less or equal than~1, and its degeneration scheme $\fD_{\ge 1}(\sigma^\vee_K)$ 
coincides with the subscheme $\gamma(\P(K)) \subset \rQ$.
Therefore, its cokernel is a line bundle on $\gamma(\P(K))$.
Moreover, the proof of the equality $\fD_{\ge 1}(\sigma^\vee_K) = \fD_{\ge 1}(\sigma^\perp_K)$ also shows that 
the cokernel sheaf is isomorphic to the line bundle~$\cO_{\P(K)}(1)$.
In other words, we have an exact sequence
\begin{equation}
\label{eq:maps-s-k-extended}
\cS_8 \xrightarrow{\ \sigma^\vee_K\ } K^\vee \otimes \cO_{\rQ} \to \gamma_*(\cO_{\P(K)}(1)) \to 0.
\end{equation}
Pulling it back to the blowup $\Bl_{\gamma(\P(K))}(\rQ)$, we obtain an exact sequence
\begin{equation}
\label{eq:maps-s-k-blowup}
r^*(\cS_8) \xrightarrow{\ r^*(\sigma^\vee_K)\ } K^\vee \otimes \cO_{\Bl_{\gamma(\P(K))}(\rQ)} \to i_*\pr^*(\cO_{\P(K)}(1)) \to 0.
\end{equation}
where $i \colon E \hookrightarrow \Bl_{\gamma(\P(K))}(\rQ)$ is the embedding of the exceptional divisor of the blowup $r$, 
while the map~$\pr \colon E \to \P(K)$ is the natural projection.
Since $E$ is a Cartier divisor, the image and the kernel of the left map are vector bundles.
We denote the kernel by $\cF$, so that we have an exact sequence 
\begin{equation}
\label{eq:tcf}
0 \to \cF \to r^*(\cS_8) \xrightarrow{\ r^*(\sigma^\vee_K)\ } K^\vee \otimes \cO_{\Bl_{\gamma(\P(K))}(\rQ)} \to i_*\pr^*(\cO_{\P(K)}(1)) \to 0,
\end{equation}
with the first map being a fiberwise monomorphism.
Below we prove that the map $\tq$ is the projectivization of the vector bundle $\cF$.

First, we consider the composition
\begin{equation*}
f \colon 
\P_{\Bl_{\gamma(\P(K))}(\rQ)}(\cF) \hookrightarrow
\P_{\Bl_{\gamma(\P(K))}(\rQ)}(r^*(\cS_8)) \to
\P_\rQ(\cS_8) \cong
\Bl_X(\P(\SS)),
\end{equation*}
where the first map is induced by the first map in~\eqref{eq:tcf}, the second is induced by the blowup $r$, 
and the third map is the isomorphism of Proposition~\ref{proposition:blowup-x}.
Clearly, the image of the morphism $f$ is the strict transform of $\P(K^\perp)$, i.e.\ $\Bl_{X_K}(\P(K^\perp)) \subset \Bl_X(\P(\SS))$.
Moreover, the composition 
\begin{equation*}
\P_{\Bl_{\gamma(\P(K))}(\rQ)}(\cF) \xrightarrow{\ f\ } \Bl_{X_K}(\P(K^\perp)) \xrightarrow{\ q\ } \rQ
\end{equation*}
by construction coincides with the composition of the maps $\P_{\Bl_{\gamma(\P(K))}(\rQ)}(\cF) \to \Bl_{\gamma(\P(K))}(\rQ) \xrightarrow{\ r\ } \rQ$.
Therefore, the scheme-theoretic preimage of the subscheme $\gamma(\P(K)) \subset \rQ$ under this composition
is a Cartier divisor.
It follows that, the map~$f$ factors through a map 
\begin{equation*}
\tilde{f} \colon \P_{\Bl_{\gamma(\P(K))}(\rQ)}(\cF) \to \Bl_{q^{-1}(\gamma(\P(K)))}(\Bl_{X_K}(\P(K^\perp))).
\end{equation*}
Now note that $\tilde{f}$ is a proper map between smooth varieties of the same dimension $15 - k$,
which is an isomorphism over the open subset $\rQ \setminus \gamma(\P(K)) \subset \rQ$, hence is birational.
Therefore, it is a blowup of an ideal.
But these two varieties have the same Picard number $3$, hence the map $\tilde{f}$ is an isomorphism.
\end{proof}

\begin{remark}
\label{remark:flattening-codim-6}
One can probably construct a similar birational flattening of the morphism $q$ in case $k = 6$.
A natural guess is that one has to blow up first the 66 points set $\fD_2(\sigma_K)$, and then to blow up 
the strict transform of~$\fD_{\ge 1}(\sigma_K) = q_-(\Bl_{12}(\P(K)))$ (see Remark~\ref{remark:stratification-explicit}).
Then there should be a rank-2 vector bundle over this blowup, whose projectivization is also an iterated blowup of $\Bl_{X_K}(\P(K^\perp))$.
This description should be useful for an identification of the Chow motive of $X_K$.
\end{remark}

\subsection{Integral Chow motives}

The first application of the blowup relation is to the integral Chow motives.
We prove an analogue of Corollary~\ref{corollary:motives-q-xk-06} on the integral level.

\begin{theorem}
\label{theorem:motives-z-xk-06}
Let $X_K$ be a smooth linear section of the spinor tenfold $X$ of codimension $k \le 5$.
Then the integral Chow motive of $X_K$ is of Lefschetz type:
\begin{equation*}
\bM(X_K) = 
\begin{cases}
1 \oplus \bL \oplus \bL^2 \oplus 2\bL^3 \oplus 2\bL^4 \oplus 2\bL^5 \oplus 2\bL^6 \oplus 2\bL^7 \oplus \bL^8 \oplus \bL^9 \oplus \bL^{10},
& \text{for $k = 0$},\\
1 \oplus \bL \oplus \bL^2 \oplus 2\bL^3 \oplus 2\bL^4 \oplus 2\bL^5 \oplus 2\bL^6 \oplus \bL^7 \oplus \bL^8 \oplus \bL^9,
& \text{for $k = 1$},\\
1 \oplus \bL \oplus \bL^2 \oplus 2\bL^3 \oplus 2\bL^4 \oplus 2\bL^5 \oplus \bL^6 \oplus \bL^7 \oplus \bL^8,
& \text{for $k = 2$},\\
1 \oplus \bL \oplus \bL^2 \oplus 2\bL^3 \oplus 2\bL^4 \oplus \bL^5 \oplus \bL^6 \oplus \bL^7,
& \text{for $k = 3$},\\
1 \oplus \bL \oplus \bL^2 \oplus 2\bL^3 \oplus \bL^4 \oplus \bL^5 \oplus \bL^6,
& \text{for $k = 4$},\\
1 \oplus \bL \oplus \bL^2 \oplus \bL^3 \oplus \bL^4 \oplus \bL^5,
& \text{for $k = 5$},\\
\end{cases}
\end{equation*}
Moreover, $\CH^i(X_K) \cong \ZZ^{n_i}$, where the ranks $n_i$ are equal to the multiplicities of the corresponding Lefschetz motives $\bL^i$ in $\bM(X_K)$.
\end{theorem}

\begin{proof}
By Proposition~\ref{proposition:blowup-xk} we have an isomorphism
\begin{equation*}
\Bl_{q^{-1}(\gamma(\P(K)))}(\Bl_{X_K}(\P(K^\perp))) \cong \P_{\Bl_{\gamma(\P(K))}(\rQ)}(\cF),
\end{equation*}
where $\cF$ is a vector bundle of rank~$8 - k$.
Using the blowup formula for the motives, we deduce
\begin{multline*}
\bM(\Bl_{q^{-1}(\gamma(\P(K)))}(\Bl_{X_K}(\P(K^\perp)))) = \\
\bM(\P(K^\perp)) \oplus 
\bM(X_K) \otimes (\bL \oplus \bL^2 \oplus \bL^3 \oplus \bL^4) \oplus 
\bM(\P(K)) \otimes \bM(\P^{7+c-k}) \otimes (\bL \oplus \dots \oplus \bL^{7-k}).
\end{multline*}
Similarly, using the blowup and the projective bundle formulas, we deduce 
\begin{multline*}
\bM(\P_{\Bl_{\gamma(\P(K))}(\rQ)}(\cF)) = 
\bM(\Bl_{\gamma(\P(K))}(\rQ)) \otimes \bM(\P^{7-k}) = 
\\
\Big(\bM(\rQ) \oplus \bM(\P(K)) \otimes (\bL \oplus \dots \oplus \bL^{8-k})\Big) \otimes (1 \oplus \bL \oplus \dots \oplus \bL^{7-k}).
\end{multline*}
The left hand sides of the equalities are isomorphic.
On the other hand, the right hand side in the second equality is a sum of Lefschetz motives.
Therefore, $\bM(X_K) \otimes \bL$, being a summand of the first equality, is also a sum of Lefschetz motives,
and hence also $\bM(X_K)$ is a sum of Lefschetz motives.

Of course, the multiplicities of $\bL^i$ in $\bM(X_K)$ are determined by multiplicities of $\bL_\QQ^i$ 
in the decomposition of the motive~$\bM_\QQ(X_K)$, which was computed in Corollary~\ref{corollary:motives-q-xk-06}.
This proves the desired formulas for~$\bM(X_K)$.
The isomorphisms for the Chow groups of $X_K$ follow immediately from the obtained expression for the motive $\bM(X_K)$.
\end{proof}

Using the approach sketched in Remark~\ref{remark:flattening-codim-6} one can also show that 
$\bM(X_K) = 1 \oplus \bL \oplus 12\bL^2 \oplus \bL^3 \oplus \bL^4$ in the case $k = 6$.

\section{Linear sections of codimension 1}
\label{section:hyperplanes}

As we already mentioned (Corollary~\ref{corollary:transitivity}), the $\Spin(\rV)$-action 
on the projective space $\P(\SS^\vee)$ of hyperplanes in $\P(\SS)$ has just two orbits,
the dual spinor variety $X^\vee \subset \P(\SS^\vee)$ and its complement $\P(\SS^\vee) \setminus X^\vee$.
Consequently, there are two isomorphism classes of hyperplane sections: singular and smooth.
In this section we will give a geometric description for both. 

\subsection{Blowup of a 4-space on $X$}

Let $U_{5,-} \subset \rV$ be an isotropic subspace corresponding to a point of~$X^\vee$.
Recall the 4-space $\Pi^4_{U_{5,-}} = \Gr(4,U_{5,-}) \cong \P(\bw4U_{5,-}) \subset X$ associated with it, see~\eqref{def:pi4-u5}.
Consider the corresponding embedding $\bw4U_{5,-} \hookrightarrow \SS$ and set 
\begin{equation}
\label{eq:w-quotient}
W := \SS/\bw4U_{5,-}.
\end{equation}
This is a vector space of dimension 11, and $\P(W) \cong  \P^{10}$.

The next result can be found in~\cite[Theorem~III.3.8(5)]{zak1993book}, 
and the rational map $f_X \circ f_W^{-1}$ constructed below is an example of 
a special birational transformations of type $(2,1)$ from~\cite{fu2015special}.
We give an independent proof, again based on the blowup lemma.

\begin{proposition}
\label{proposition:blowup-x-pi4}
There is an isomorphism $\Bl_{\Pi^4_{U_{5,-}}}(X) \cong \Bl_{\Gr(2,U_{5,-})}(\P(W))$ and a diagram
\begin{equation}
\vcenter{\xymatrix{
&
E_\Pi \ar[r] \ar[dl] &
\Bl_{\Pi^4_{U_{5,-}}}(X) \ar@{=}[r]^-\sim \ar[dl]^{f_X} &
\Bl_{\Gr(2,U_{5,-})}(\P(W)) \ar[dr]_{f_W} &
E_{\Gr} \ar[l] \ar[dr]
\\
\Pi^4_{U_{5,-}} \ar[r] &
X &&&
\P(W) &
\Gr(2,U_{5,-}) \ar[l]
}}
\end{equation}
If $H_X$ and $H_W$ denote the hyperplane classes of $X$ and $\P(W)$, while $E_\Pi$ and $E_{\Gr}$ denote the exceptional divisors of the blowups, then
\begin{equation}
\label{eq:picard-relations-blowup-gr-p10}
\begin{aligned}
H_X &= 2H_W - E_{\Gr}, \qquad	&	H_W &= H_X - E_\Pi,\\
E_\Pi &= H_W - E_{\Gr},		&	E_{\Gr} &= H_X - 2E_\Pi.
\end{aligned}
\end{equation} 
Finally, the birational map $f_W \circ f_X^{-1} \colon X \dashrightarrow \P(W)$ is induced by the projection $\SS \to W$,
so that this map is the linear projection with center in~$\Pi^4_{U_{5,-}}$.
\end{proposition}

\begin{proof}
Consider an abstract 5-dimensional vector space $V_5$ and define 
\begin{equation}
\label{eq:w-decomposition}
W = \bw2V_5 \oplus \Bbbk
\end{equation}
(later we will identify $V_5$ with $U_{5,-}$ and the above direct sum with the quotient $\SS/\bw4U_{5,-}$).
Then we have a natural embedding $\Gr(2,V_5) \hookrightarrow \P(\bw2V_5) \hookrightarrow \P(W)$.
Consider the blowup 
\begin{equation*}
f_W \colon \Bl_{\Gr(2,V_5)}(\P(W)) \to \P(W).
\end{equation*}
Below we construct a map $f_X$ from the blowup $\Bl_{\Gr(2,V_5)}(\P(W))$ to the spinor tenfold $X = \OGr_+(5,\rV)$ 
by producing an isotropic rank-5 vector subbundle in the trivial vector bundle with fiber $\rV$.
After that we will check that $f_X$ is birational, and apply the blowup lemma~\ref{lemma:blowup} 
to show that $f_X$ is the blowup of a 4-space on $X$.

Recall that on $\P(\bw2 V_5)$ there is a natural resolution 
\begin{equation*}
0 \to \cO(-5) \xrightarrow{\ \xi \wedge \xi\ } V_5^\vee(-3) \xrightarrow{\ \xi\ } V_5(-2) \xrightarrow{\ \xi \wedge \xi\ } \cO \to \cO_{\Gr(2,V_5)} \to 0,
\end{equation*}
where $\xi \in H^0(\P(\bw2V_5),\bw2V_5(1))$ is the tautological section and $\xi \wedge \xi \in H^0(\P(\bw2V_5),\bw4V_5(2))$ is its exterior square.
Combining it with the Koszul resolution 
\begin{equation*}
0 \to \cO_{\P(W)}(-1) \xrightarrow{\ \eta\ } \cO_{\P(W)} \to \cO_{\P(\wedge^2V_5)} \to 0,
\end{equation*}
where $\eta \in H^0(\P(W), \cO_{\P(W)}(1))$ is the equation of the hyperplane $\P(\bw2V_5) \subset \P(W)$, 
we obtain on $\P(W)$ the following resolution
\begin{multline}
\label{eq:resolution-gr-p10}
0 \to 
\cO_{\P(W)}(-6) 
    \xrightarrow{\ \left(\begin{smallmatrix} -\eta \\ \xi \wedge \xi \end{smallmatrix}\right)\ }
\cO_{\P(W)}(-5) \oplus V_5^\vee \otimes \cO_{\P(W)}(-4) 
\\
    \xrightarrow{\ \left(\begin{smallmatrix} \xi \wedge \xi & \eta \\ 0 & \xi \end{smallmatrix}\right)\ } 
(V_5^\vee \oplus V_5) \otimes \cO_{\P(W)}(-3) 
    \xrightarrow{\ \left(\begin{smallmatrix} \xi & -\eta \\ 0 & \xi \wedge \xi \end{smallmatrix}\right)\ }
\\
V_5 \otimes \cO_{\P(W)}(-2) \oplus \cO_{\P(W)}(-1) 
    \xrightarrow{\ (\xi \wedge \xi,\eta)\ } 
\cO_{\P(W)} 
    \to 
\cO_{\Gr(2,V_5)} \to 0.
\end{multline}
We pullback this complex to the blowup $\Bl_{\Gr(2,V_5)}(\P(W))$ 
(we denote by $\pr \colon E_{\Gr} \to \Gr(2,V_5)$ the projection of its exceptional divisor).
Of course, it is no longer exact, and in fact, its cohomology sheaves are isomorphic to the exterior powers of the excess conormal bundle
\begin{equation*}
\bar\cN^\vee := \Ker \left( \pr^*\left(\cN^\vee_{\Gr(2,V_5)/\P(W)}\right) \xrightarrow{\qquad} \cO_{E_{\Gr}}(-E_{\Gr}) \right).
\end{equation*}
In other words, we have the following exact sequences on $\Bl_{\Gr(2,V_5)}(\P(W))$: 
\begin{align*}
0 \to \cF' \to V_5 \otimes \cO(-2H_W) \oplus \cO(-H_W) \to \cO \to \cO_{E_{\Gr}} \to 0,\\
0 \to \cF'' \to (V_5^\vee \oplus V_5) \otimes \cO(-3H_W) \to \cF' \to \bar\cN^\vee \to 0,\\
0 \to \cF''' \to \cO(-5H_W) \oplus V_5^\vee \otimes \cO(-4H_W) \to \cF'' \to \bw{2}{\bar\cN^\vee} \to 0,\\
0 \to \cO(-6H_W) \to \cF''' \to \bw{3}{\bar\cN^\vee} \to 0.
\end{align*}
Consider the vector space $V_5^\vee \oplus V_5$ with its natural non-degenerate quadratic form (induced by the pairing between the summands).
We claim that the sheaf 
\begin{equation}
\label{eq:cf-pp}
\cF'' := \Ker\left(  (V_5^\vee \oplus V_5) \otimes \cO(-3H_W) 
\xrightarrow{\ \left(\begin{smallmatrix} \xi & -\eta \\ 0 & \xi \wedge \xi \end{smallmatrix}\right)\ }
V_5 \otimes \cO_{\P(W)}(-2) \oplus \cO_{\P(W)}(-1) \right)
\end{equation}
defined by the second of the above sequences is an isotropic subbundle in $(V_5^\vee \oplus V_5) \otimes \cO(-3H_W)$ 
and defines a regular map $\Bl_{\Gr(2,V_5)}(\P(W)) \to \OGr_+(5,V_5^\vee \oplus V_5) = X$.

Indeed, the sheaf $\cO_{E_{\Gr}}$ in the first sequence is locally free on a divisor, hence the sheaf $\cF'$ is locally free of rank 5.
Similarly, the sheaf $\bar\cN^\vee$ in the second sequence is locally free on a divisor, 
hence the kernel of the map $\cF' \to \bar\cN^\vee$ is locally free of rank 5.
Consequently, the sheaf $\cF''$ is locally free of rank~5 and its embedding 
into $(V_5^\vee \oplus V_5) \otimes \cO(-3H_W)$ is a fiberwise monomorphism.

Let us show that $\cF''$ is isotropic as a subbundle in $(V_5^\vee \oplus V_5) \otimes \cO(-3H_W)$.
Clearly, it is enough to check this on the open subset 
\begin{equation}
\label{eq:open-blowup-gr-p10}
\P(W) \setminus \P(\bw2 V_5) \subset \Bl_{\Gr(2,V_5)}(\P(W)),
\end{equation} 
the complement of the linear span of the Grassmannian.
Another way to describe this open subset is by inequality $\eta \ne 0$.
So, by rescaling we may assume $\eta = 1$ on~\eqref{eq:open-blowup-gr-p10} and use $\xi$ as a coordinate.

On the open set~\eqref{eq:open-blowup-gr-p10} the complex~\eqref{eq:resolution-gr-p10} is acyclic, 
hence the bundle $\cF''$ is just the image of the second map in the complex.
Since the image of the first map $\left(\begin{smallmatrix} -1 \\ \xi \wedge \xi \end{smallmatrix}\right)$
surjects over the first summand $\cO(-5)$ in~$\cO(-5) \oplus V_5^\vee \otimes \cO(-4)$, therefore $\cF''$ is the image of the second summand.
Thus 
\begin{equation*}
\cF''\vert_{\P(W) \setminus \P(\wedge^2 V_5) } = \Ima \left( 
V_5^\vee \otimes \cO(-4) 
    \xrightarrow{\ \left(\begin{smallmatrix}  1 \\ \xi \end{smallmatrix}\right)\ } 
(V_5^\vee \oplus V_5) \otimes \cO(-3) 
\right).
\end{equation*}
In other words, $\cF''\vert_{\P(W) \setminus \P(\wedge^2 V_5) }$ is the graph of the map $V_5^\vee \otimes \cO(-4) \xrightarrow{\ \xi\ } V_5 \otimes \cO(-3)$.
The map~$\xi$ is skew-symmetric by definition, hence its graph is isotropic for the natural quadratic form.

Next, we choose an isomorphism 
\begin{equation*}
\rV \cong V_5^\vee \oplus V_5,
\end{equation*}
which is compatible with the quadratic forms on these vector spaces, and such that 
the isotropic subspace~\mbox{$V_5 \subset V_5^\vee \oplus V_5 \cong \rV$} corresponds to a point of $X^\vee$.
Then we obtain a map 
\begin{equation*}
f_X \colon \Bl_{\Gr(2,V_5)}(\P(W)) \to X
\end{equation*}
such that $f_X^*(\cU) \cong \cF''(3H_W)$.
Indeed, the map to $\OGr(5,\rV)$ is given by the universal property of the Grassmannian,
and it lands into the connected component $X$ because
the graph of any map $V_5^\vee \to V_5$ does not intersect the subspace $V_5 \subset V_5^\vee \oplus V_5$, 
and hence by~\eqref{eq:parity-intersection} when isotropic it corresponds to a point of $X$.
This proves that the image of the open subset~\eqref{eq:open-blowup-gr-p10} is in $X$, 
hence the same is true for the entire $X$ by continuity.

To show that $f_X$ is birational, just note that its restriction to the open subset~\eqref{eq:open-blowup-gr-p10} 
is an isomorphism onto the open subset of $X \cong \OGr_+(5,V_5^\vee \oplus V_5)$ parameterizing isotropic subspaces 
that do not intersect the subspace $V_5$.
Indeed, any such subspace is the graph of a map $V_5^\vee \to V_5$, and a graph is isotropic 
if and only if the corresponding map is skew-symmetric.

Using exact sequences defining the sheaves $\cF'$, $\cF''$ and $\cF'''$ it is easy to compute that 
\begin{equation*}
\rc_1(\cF'(3H_W)) = 4 H_W + E_{\Gr},\qquad
\rc_1(\cF''(3H_W)) = - (4H_W + E_{\Gr}) + 3E_{\Gr} = - 2(2H_W - E_{\Gr}).
\end{equation*}
Since $\rc_1(\cU_5) = -2H_X$ by~\eqref{eq:detu-s}, it follows that $H_X = 2H_W - E_{\Gr}$, 
thus proving the first of the relations of~\eqref{eq:picard-relations-blowup-gr-p10}.
In other words, the map $f_X$ is given by quadrics passing through $\Gr(2,V_5)$.

Next, consider the restriction of $f_X$ to the linear span $\P(\bw2V_5) \subset \P(W)$ of the Grassmannian.
It is classical that the map given by the Pl\"ucker quadrics defines an isomorphism
\begin{equation*}
\Bl_{\Gr(2,V_5)}(\P(\bw2V_5)) \cong \P_{\Gr(4,V_5)}(\bw2\cU_4)
\end{equation*}
(actually, this is an analogue for $\Gr(2,5)$ of the isomorphism of Proposition~\ref{proposition:blowup-x}).
Therefore, the map~$f_X$ contracts the strict transform of~$\P(\bw2V_5)$ onto $\Gr(4,V_5) \subset X$. 
In other words, we have a commutative diagram
\begin{equation*}
\xymatrix@C=5em{
\P_{\Gr(4,V_5)}(\bw2\cU_4) \ar[d] \ar@{^{(}->}[r] & 
\Bl_{\Gr(2,V_5)}(\P(W)) \ar[d]^{f_X}
\\
\Gr(4,V_5)
\ar@{^{(}->}[r] & 
X
}
\end{equation*}
It is clear that the relative Picard number for the map $f_X$ equals 1.
Therefore, by Lemma~\ref{lemma:blowup} we conclude that $f_X$ is the blowup of $\Gr(4,V_5) \subset X$ and
\begin{equation*}
E_\Pi := \P_{\Gr(4,V_5)}(\bw2\cU_4)
\end{equation*}
is the exceptional divisor of~$f_X$.
Since $E_\Pi$ is the strict transform of a hyperplane in $\P(W)$ passing through~$\Gr(2,V_5)$, 
we have a linear equivalence $E_\Pi = H_W - E_{\Gr}$.
Combining it with the linear equivalence~$H_X = 2H_W - E_{\Gr}$ proved above, we deduce~\eqref{eq:picard-relations-blowup-gr-p10}.
%
%
This completes the proof of the first two parts of the proposition.
It remains to identify the image $\Gr(4,V_5) = f_X(E_\Pi) \subset X$ with $\Pi^4_{U_{5,-}}$ and 
the rational map~$f_X \circ f_W^{-1}$ with the linear projection from~$\Pi^4_{U_{5,-}}$.

First, by~\eqref{eq:picard-relations-blowup-gr-p10} we have an isomorphism
\begin{equation*}
\SS^\vee \cong H^0(X,\cO_X(H_X)) \cong H^0(\P(W),I_{\Gr(2,V_5)}(2H_W)).
\end{equation*}
The right hand side is the space of quadrics in $\P(W)$ through $\Gr(2,V_5)$;
consequently we get an exact sequence
\begin{equation*}
0 \to W^\vee \xrightarrow{\ \eta\ } \SS^\vee \to \bw4V_5^\vee \to 0,
\end{equation*}
where the first term is the space of quadrics containing the hyperplane $\P(\bw2V_5) \subset \P(W)$,
and the second term is the space of Pl\"ucker quadrics in $\P(\bw2V_5)$.
After dualization we obtain
\begin{equation}
\label{eq:w-sequence-3}
0 \to \bw4V_5 \to \SS \to W \to 0,
\end{equation}
and the image $f_X(E_\Pi)$ of the exceptional divisor $E_\Pi$ is identified with $\Gr(4,V_5) = \P(\bw4V_5) \subset \P(\SS)$.
Moreover, the composition of the map $f_X$ with the linear projection $\P(\SS) \dashrightarrow \P(W)$ coincides
with the map given by the linear system of quadrics in $\P(W)$ containing the hyperplane $\P(\bw2V_5)$,
which coincides with the linear system of all linear functions. 
This proves that the map $f_W \circ f_X^{-1}$ is the linear projection from $\P(\bw4V_5)$.

Finally, it remains to check, that setting $U_{5,-} = V_5$, we obtain an identification of $\P(\bw4V_5)$ with $\Pi^4_{U_{5,-}}$.
For this we consider the restriction of the map $f_X \circ f_W^{-1}$ to the subscheme $\P(\bw2V_5) \setminus \Gr(2,V_5) \subset E_\Pi$.
On this subscheme of $\P(W)$ we have $\eta = 0$ and $\xi$ is a skew-symmetric matrix of rank~4.
By~\eqref{eq:cf-pp} the intersection of each fiber of the subbundle $\cF'' \subset (V_5^\vee \oplus V_5) \otimes \cO(-3H_W) $ 
with the subspace $V_5 \subset V_5^\vee \oplus V_5$ is the kernel of 
$\xi \wedge \xi \colon V_5 \otimes \cO(-3H_W) \to \cO(-H_W)$, i.e., a 4-dimensional subspace in $V_5$.
Therefore, the image $f_X(E_\Pi)$ is contained in the locus of $U_5 \subset \rV$ that have a 4-dimensional intersection with $V_5$.
Thus we have $\P(\bw4V_5) \subset \Pi^4_{V_5} = \Pi^4_{U_{5,-}}$. 
Since both sides are 4-spaces, this is an equality.
\end{proof}

\begin{remark}
\label{remark:w-splitting}
Note that the direct sum decomposition~\eqref{eq:w-decomposition} that we used to prove the proposition is not canonical 
(actually, it corresponds to the isotropic direct sum decomposition $\rV = V_5^\vee \oplus V_5$ 
that we obtained from~\eqref{eq:w-decomposition} during the proof).
On the other hand, there is a canonical exact sequence
\begin{equation}
\label{eq:w-sequence}
0 \to \bw2U_{5,-} \to W \to \Bbbk \to 0,
\end{equation}
where we identified $V_5 = U_{5,-}$ as in the proof.
Indeed, the subspace $\bw2U_{5,-}$ corresponds to the linear span of the Grassmannian $\Gr(2,U_{5,-}) \subset \P(W)$
(the center of the blowup $f_W$).
\end{remark}

\subsection{Blowups of 4-spaces on hyperplane sections of $X$}

As before we let $U_{5,-} \subset \rV$ be an isotropic subspace corresponding to a point of~$X^\vee$ and 
let $W$ be the space defined by~\eqref{eq:w-quotient}.
Consider the preimage in $\SS$ of the hyperplane~$\bw2U_{5,-} \subset W$ (from~\eqref{eq:w-sequence}) 
with respect to the projection $\SS \to W$ of~\eqref{eq:w-quotient}.
This is a hyperplane in~$\SS$.
We denote the corresponding hyperplane section of $X$ by $X_{U_{5,-}} \subset X$.
Denote by~$\cU_2$ the tautological rank-$2$ bundle on $\Gr(2,U_{5,-})$ and by $\cU_2^\perp$ 
the tautological rank-3 bundle on the same Grassmannian.

\begin{corollary}
\label{corollary:singular-hyperplane-sections}
The singular locus of the hyperplane section $X_{U_{5,-}} \subset X$ is the $4$-space $\Pi^4_{U_{5,-}} \subset X$.
Moreover, there is an isomorphism $\Bl_{\Pi^4_{U_{5,-}}}(X_{U_{5,-}}) \cong \P_{\Gr(2,U_{5,-})}(\cU_2^\perp \oplus \cO(-1))$,
such that the exceptional divisor $E'_\Pi$ of the blowup is identified with $\P_{\Gr(2,U_{5,-})}(\cU_2^\perp)$, 
and there is a diagram
\begin{equation}
\vcenter{\xymatrix{
&&
E'_\Pi \ar[ddll] \ar[d] \ar@{=}[r]^-\sim &
\P_{\Gr(2,U_{5,-})}(\cU_2^\perp) \ar[d] \ar@{=}[r]^-\sim &
\Fl(2,4;U_{5,-}) \ar@/^5ex/[ddl]
\\
&&
\Bl_{\Pi^4_{U_{5,-}}}(X_{U_{5,-}}) \ar@{=}[r]^-\sim \ar[dl] &
\P_{\Gr(2,U_{5,-})}(\cU_2^\perp \oplus \cO(-1)) \ar[d] 
\\
\Pi^4_{U_{5,-}} \ar[r] &
X_{U_{5,-}} &&
\Gr(2,U_{5,-}) 
}}
\end{equation}
\end{corollary}

\begin{proof}
We use the notation introduced in the proof of Proposition~\ref{proposition:blowup-x-pi4} with the identification~\mbox{$V_5 = U_{5,-}$}.
Consider the preimage $f_W^{-1}(\P(\bw2U_{5,-})) \subset \Bl_{\Gr(2,U_{5,-})}(\P(W))$ of the hyperplane $\P(\bw2U_{5,-}) \subset \P(W)$.
Since $\Gr(2,U_{5,-}) \subset \P(\bw2U_{5,-})$, this preimage contains the exceptional divisor $E_{\Gr}$.
On the other hand, the strict transform of this hyperplane is the exceptional divisor $E_\Pi$.
Since $E_\Pi + E_{\Gr} = H_W$ by~\eqref{eq:picard-relations-blowup-gr-p10}, 
it follows that there are no other components in the preimage, i.e.,
\begin{equation*}
f_W^{-1}(\P(\bw2U_{5,-})) = E_\Pi \cup E_{\Gr}.
\end{equation*}
The morphism $f_X$ contracts $E_\Pi$ to the 4-space $\Pi^4_{U_{5,-}}$, 
and maps $E_{\Gr}$ to a hyperplane section of $X$ singular along $\Pi^4_{U_{5,-}}$ (this follows from the relation $E_{\Gr} = H_X - 2E_\Pi$).
On the other hand, since the rational map~$X \dashrightarrow \P(W)$ of Proposition~\ref{proposition:blowup-x-pi4} is the linear projection
induced by the map~$\SS \to W$, the strict transform of the hyperplane~$\P(\bw2U_{5,-})$ in $X$ 
is the hyperplane section~\mbox{$X_{U_{5,-}} \subset X$}.
Therefore, the morphism~$f_X$ induces a birational map $E_{\Gr} \to X_{U_{5,-}}$.
Moreover, $E_{\Gr}$ is the strict transform of~$X_{U_{5,-}}$ in~$\Bl_{\Pi^4_{U_{5,-}}}(X)$, and hence
\begin{equation*}
\Bl_{\Pi^4_{U_{5,-}}}(X_{U_{5,-}}) \cong E_{\Gr} \cong \P_{\Gr(2,U_{5,-})}(\cN_{\Gr(2,U_{5,-})/\P(W)}).
\end{equation*}
It remains to show that $X_{U_{5,-}} \setminus \Pi^4_{U_{5,-}}$ is smooth,
and that the normal bundle of $\Gr(2,U_{5,-})$ in $\P(W)$ is isomorphic to a twist of $\cU_2^\perp \oplus \cO(-1)$.
The first follows evidently from smoothness of $E_{\Gr}$.
For the second we use the following exact sequence
\begin{equation*}
0 \to \cN_{\Gr(2,U_{5,-})/\P(\wedge^2 U_{5,-})} \to \cN_{\Gr(2,U_{5,-})/\P(W)} \to \cO_{\P(W)}(1)\vert_{\Gr(2,U_{5,-})} \to 0
\end{equation*}
(the last term comes from the normal bundle of the hyperplane $\P(\bw2 U_{5,-}) \subset \P(W)$).
It is well known that the first term is isomorphic to $\cU_2^\perp(2)$ (see, for instance~\cite[Proposition~A.7]{debarre2015gushel}), 
hence the middle term is an extension of $\cO(1)$ by $\cU_2^\perp(2)$.
On the other hand, by Borel--Bott--Weil we have
\begin{equation*}
\Ext^1(\cO(1),\cU_2^\perp(2)) \cong H^1(\Gr(2,U_{5,-}), \cU_2^\perp(1)) = 0,
\end{equation*}
hence the extension splits and we deduce an isomorphism
\begin{equation*}
\cN_{\Gr(2,U_{5,-})/\P(W)} \cong \cU_2^\perp(2) \oplus \cO(1).
\end{equation*}
Since the projectivization is not affected by a twist, we deduce the required isomorphism.

Finally, we have to identify the exceptional divisor $E'_\Pi$.
The above argument shows that $E'_\Pi = E_\Pi \cap E_{\Gr}$, 
so $E'_\Pi$ is nothing but the exceptional divisor of the blowup of $\P(\bw2U_{5,-})$ along $\Gr(2,U_{5,-})$.
Therefore, 
\begin{equation*}
E'_\Pi \cong \P_{\Gr(2,U_{5,-})}(\cN_{\Gr(2,U_{5,-})/\P(\wedge^2 U_{5,-})}) \cong \P_{\Gr(2,U_{5,-}}(\cU_2^\perp)
\end{equation*}
which embeds into $\P_{\Gr(2,U_{5,-})}(\cU_2^\perp \oplus \cO(-1))$ as the projectivization of the first summand.
\end{proof}

\begin{remark}
It is easy to see that the hyperplane section $X_{U_{5,-}}$ considered above is nothing but the singular hyperplane section of $X$
associated with the point $[U_{5,-}] \in X^\vee$ in view of their projective duality.
One of the ways to do this is the following.
By projective duality every singular hyperplane section of $X$ corresponds to a point of $X^\vee$, in particular $X_{U_{5,-}}$ does.
This defines an automorphism of~$X^\vee$ which is canonical and hence $\Spin(\rV)$-equivariant.
Therefore, it belongs to the center of the group~$\Aut(X^\vee) \cong \PSO(\rV)$ which is trivial.
\end{remark}

As an application of the above results we describe the Hilbert schemes $F_4(X_{U_{5,-}})$ and $G_6(X_{U_{5,-}})$.

\begin{corollary}
If $X_{U_{5,-}} \subset X$ is the singular hyperplane section of $X$ corresponding to an isotropic subspace~$U_{5,-} \subset \rV$ 
then $F_4(X_{U_{5,-}}) \cong \Cone(\Gr(3,U_{5,-}))$
\textup(the cone in the Pl\"ucker embedding\textup).
\end{corollary}
\begin{proof}
By Theorem~\ref{theorem:linear-spaces} every 4-space on $X$ is equal to $\Pi^4_{U'_{5,-}}$ for some isotropic subspace $U'_{5,-} \subset \rV$.
So, we should figure out, when $\Pi^4_{U'_{5,-}} \subset X_{U_{5,-}}$. 
In other words, which condition on $U'_{5,-}$ ensures that 
for any isotropic subspace $U_5 \subset \rV$ such that $\dim(U_5 \cap U'_{5,-}) = 4$
we have $\dim(U_5 \cap U_{5,-}) \ge 2$.
Clearly, $\dim(U'_{5,-} \cap U_{5,-}) \ge 3$ is a sufficient condition.
Let us show it is also necessary.


So, assume $\dim(U'_{5,-} \cap U_{5,-}) = 1$ (recall that the dimension of such intersection is always odd).
Consider a subspace $U_4 \subset U'_{5,-}$ such that $U_4 \cap U_{5,-} = 0$, and let $U_5$ be its unique extension 
to an isotropic subspace corresponding to a point of $X$. 
Then, of course, $\dim(U_5 \cap U_{5,-}) \le 1$.

Thus, we have checked that $F_4(X_{U_{5,-}}) = \{ U'_{5,-} \mid \dim(U'_{5,-} \cap U_{5,-}) \ge 3 \}$.
Let us show this is a cone over the Grassmannian.
For this consider the scheme $\tF_4(X_{U_{5,-}})$ parameterizing pairs of subspaces $(U_3,U'_{5,-})$ such that $U_3 \subset U'_{5,-} \cap U_{5,-}$.
Forgetting $U'_{5,-}$ defines a map $\varphi \colon \tF_4(X_{U_{5,-}}) \to \Gr(3,U_{5,-}) \subset \OGr(3,\rV)$ 
whose fiber over a point $[U_3] \in \Gr(3,U_{5,-})$ is the line $L^-_{U_3} \subset X^\vee$ associated with the isotropic subspace~$U_3$.
The identification of the universal line $L^-$ from the right half of~\eqref{eq:universal-lines} shows that
\begin{equation*}
\tF_4(X_{U_{5,-}}) = \P_{\Gr(3,U_{5,-})}(\varphi^*(\cS_{2,-})),
\end{equation*}
where $\cS_{2,-}$ is the spinor bundle on $\OGr(3,\rV)$.
According to Remark~\ref{remark:filtrations} the restriction of this spinor bundle to $\Gr(3,U_{5,-})$ admits a filtration
which takes the form of a short exact sequence
\begin{equation*}
0 \to \cO \to \cS_{2,-}\vert_{\Gr(3,U_{5,-})} \to \bw2(U_{5,-}/\cU_3)^\vee \to 0.
\end{equation*}
Clearly, the right term is isomorphic to $\cO(1)$, and since on $\Gr(3,U_{5,-})$ there are no non-trivial extensions between $\cO(1)$ and $\cO$,
the restriction of the spinor bundle is isomorphic to $\cO \oplus \cO(1)$. Therefore
\begin{equation*}
\tF_4(X_{U_{5,-}}) \cong \P_{\Gr(3,U_{5,-})}(\cO \oplus \cO(-1)).
\end{equation*}
The projection $\tF_4(X_{U_{5,-}}) \to F_4(X_{U_{5,-}})$, of course, contracts the exceptional section of the above projective bundle
(that parameterizes pairs $(U_3,U'_{5,-})$ with $U'_{5,-} = U_{5,-}$) to the point of $F_4(X_{U_{5,-}})$ corresponding to the subspace $U_{5,-}$.
The result of such a contraction is of course the cone $\Cone(\Gr(3,U_{5,-}))$.
\end{proof}

\begin{lemma}
If $X_{U_{5,-}} \subset X$ is the singular hyperplane section of $X$ corresponding to an isotropic subspace~$U_{5,-} \subset \rV$ 
then $G_6(X_{U_{5,-}}) \cong \P(U_{5,-})$.
\end{lemma}
\begin{proof}
Recall that every 6-dimensional quadric on $X$ is equal to $\cQ_v = \OGr_+(4,v^\perp/v)$ and parameterizes isotropic subspaces $U_5 \subset \rV$ that contain $v$.
On the other hand, the singular hyperplane section $X_{U_{5,-}}$ parameterizes those $U_5$ that intersect a given subspace $U_{5,-}$.
So, let us show that any $U_5$ containing $v$ intersects $U_{5,-}$ if and only if $v \in U_{5,-}$.

One direction is evident.
For the other direction, assume $v \not\in U_{5,-}$, so that $U_{5,-} \not\subset v^\perp$.
Let $U_4 \subset U_{5,-}$ be a subspace which is not contained in $v^\perp$.
Then the unique isotropic extension of $U_4$ to $U_5 \subset \rV$ is not contained in $v^\perp$, hence does not contain $v$.
\end{proof}

One can also use Proposition~\ref{proposition:blowup-x-pi4} for a description of smooth hyperplane sections of $X$.
The following birational transformation is again an example of
a special birational transformations of type $(2,1)$ from~\cite{fu2015special}.

\begin{corollary}
\label{corollary:blowup-xkappa-smooth-pi4}
Let $\kappa \in \P(\SS^\vee) \setminus X^\vee$ be a point and
let $X_\kappa \subset X$ be the corresponding smooth hyperplane section of $X$.
If\/ $\Pi^4_{U_{5,-}} \subset X_\kappa$ then there is an isomorphism $\Bl_{\Pi^4_{U_{5,-}}}(X_\kappa) \cong \Bl_{Z_\kappa}(\P(W_\kappa))$,
and a diagram
\begin{equation}
\vcenter{\xymatrix{
&
E_{\Pi,\kappa} \ar[r] \ar[dl] &
\Bl_{\Pi^4_{U_{5,-}}}(X_\kappa) \ar@{=}[r]^-\sim \ar[dl] &
\Bl_{Z_\kappa}(\P(W_\kappa)) \ar[dr] &
E_{Z_\kappa} \ar[l] \ar[dr]
\\
\Pi^4_{U_{5,-}} \ar[r] &
X_\kappa &&&
\P(W_\kappa) &
Z_\kappa \ar[l]
}}
\end{equation}
where $W_\kappa \subset W$ is the hyperplane corresponding to $\kappa$, 
and $Z_\kappa = \Gr(2,U_{5,-}) \cap \P(W_\kappa)$ is a smooth hyperplane section of the Grassmannian.
\end{corollary}

\begin{proof}
Since the map $X \dashrightarrow \P(W)$ is a linear projection with center at $\Pi^4_{U_{5,-}}$, 
hyperplanes in $\P(\SS)$ containing $\Pi^4_{U_{5,-}}$ correspond to hyperplanes in $\P(W)$.
Let $W_\kappa \subset W$ be the hyperplane corresponding to~$\kappa$.
Note that this hyperplane is distinct from the hyperplane $\bw2U_{5,-} \subset W$, 
since the latter corresponds to a singular hyperplane section of $X$.
Therefore, the intersection $Z_\kappa = \Gr(2,U_{5,-}) \cap \P(W_\kappa)$ is dimensionally transverse.

The preimage of $\P(W_\kappa)$ in $\Bl_{\Gr(2,U_{5,-})}(\P(W))$ is isomorphic to the blowup $\Bl_{Z_\kappa}(\P(W_\kappa))$ and,
at the same time, it is the strict transform of $X_\kappa$, hence is isomorphic to the blowup $\Bl_{\Pi^4_{U_{5,-}}}(X_\kappa)$.
This gives the required diagram.

It only remains to check that $Z_\kappa$ is smooth.
For this just note that $\Bl_{\Pi^4_{U_{5,-}}}(X_\kappa)$ is smooth and $Z_\kappa$ is a locally complete intersection.
Therefore, Lemma~\ref{lemma:smoothness-criterion} applies.
\end{proof}

Later we will show that $F_4(X_\kappa) \ne \varnothing$ (Corollary~\ref{corollary:f4-x-codim2}), so the above description is applicable.

\subsection{Blowup of a 6-quadric on $X$}

Next we present yet another description of the spinor tenfold $X$ by projecting from a maximal quadric 
and use it for an alternative description of its smooth hyperplane section.
Recall that for each point $v \in \rQ$ there is an exact sequence
\begin{equation}
\label{eq:cs-v-sequence}
0 \to \cS_{8,v} \to  \SS \to \cS_{8,v,-} \to 0
\end{equation} 
(this is the fiber at $v$ of the sequence of Lemma~\ref{lemma:spinor-sequence}).
The projective spaces $\P(\cS_{8,v})$ and $\P(\cS_{8,v,-})$ from this sequence contain smooth 6-dimensional quadrics
\begin{equation*}
\cQ_v \subset \P(\cS_{8,v})
\qquad\text{and}\qquad
\cQ_{v,-} \subset \P(\cS_{8,v,-}),
\end{equation*}
see~\eqref{eq:cq-in-ps}. 
Recall also that $\cQ_v = \OGr_+(4,v^\perp/v)$ and $\cQ_{v,-} = \OGr_-(4,v^\perp/v)$, see~\eqref{eq:qv-qvminus}. 
We denote by $\cU_4$ and $\cU_{4,-}$ the tautological bundles on $\cQ_v$ and $\cQ_{v,-}$, considered as isotropic Grassmannians.

\begin{proposition}
\label{proposition:blowup-x-ogr48}
There is an isomorphism
$\Bl_{\cQ_v} (X) \cong \P_{\cQ_{v,-}}(\cO(-1) \oplus \cU_{4,-}^\vee(-1))$ such that
the exceptional divisor $E_\cQ$ of the blowup is identified with $\P_{\cQ_{v,-}}(\cU_{4,-}^\vee(-1))$,
so that we have a diagram
\begin{equation}
\label{diagram:blowup-x-ogr48}
\vcenter{\xymatrix{
&&
E_\cQ \ar[ddll] \ar[d] \ar@{=}[r]^-\sim &
\P_{\cQ_{v,-}}(\cU_{4,-}^\vee(-1)) \ar[d] 
\\
&&
\Bl_{\cQ_v}(X) \ar@{=}[r]^-\sim \ar[dl]^{g_X} &
\P_{\cQ_{v,-}}(\cO(-1) \oplus \cU_{4,-}^\vee(-1)) \ar[d]_{g_\cQ} 
\\
\cQ_v \ar[r] &
X &&
\cQ_{v,-}
}}
\end{equation} 
The morphism $g_\cQ \colon \Bl_{\cQ_v} (X) \to \cQ_{v,-}$ is given by the linear system $|H_X - E_\cQ|$,
and the rational map $g_\cQ \circ g_X^{-1} \colon X \dashrightarrow \cQ_{v,-}$ is the linear projection 
$\P(\SS) \dashrightarrow \P(\cS_{8,v,-})$ induced by the second map from~\eqref{eq:cs-v-sequence}.
\end{proposition}

\begin{proof}
Recall that $\OGr(4,\rV) \cong \Spin(\rV)/\bP_{4,5}$, the homogeneous space of $\Spin(\rV)$ 
that corresponds to a non-maximal parabolic subgroup $\bP_{4,5}$ (associated wit the fundamental weights $\omega_4$ and $\omega_5$).
Consequently, it is a subvariety in the product $X \times X^\vee$, and is isomorphic to the projectivization 
of $\cU_5^\vee(-1)$ over $X$, resp.\ of $\cU_{5,-}^\vee(-1)$ over $X^\vee$.
On the other hand, we have a natural embedding $\cQ_{v,-} = \OGr_-(4,v^\perp/v) \hookrightarrow X^\vee$.
Consider the diagram
\begin{equation*}
\xymatrix{
\cQ_{v,-} \times_{X^\vee} \OGr(4,\rV) \ar[r] \ar[d] &
\OGr(4,\rV) \ar[r] \ar[d] &
X
\\
\cQ_{v,-} \ar[r] &
X^\vee
}
\end{equation*}
where the square is cartesian.
Clearly, 
\begin{equation*}
\cQ_{v,-} \times_{X^\vee} \OGr(4,\rV) \cong \P_{\cQ_{v,-}}\left(\cU_{5,-}^\vee(-1) \vert_{\cQ_{v,-}}\right).
\end{equation*}
On the other hand, since $\cQ_{v,-} = \OGr_-(4,v^\perp/v)$ parameterizes isotropic subspaces $U_{5,-} \subset \rV$ 
that contain the vector~$v$, hence we have an exact sequence
\begin{equation}
\label{eq:u5-qv}
0 \to \cO \xrightarrow{\ v\ } \cU_{5,-} \vert_{\cQ_{v,-}} \to \cU_{4,-} \to 0.
\end{equation}
Using Borel--Bott--Weil it is easy to check that $\Ext^1(\cU_{4,-},\cO) = 0$, hence the sequence splits, 
and after dualization and twist we get an isomorphism
\begin{equation*}
\cU_{5,-}^\vee(-1)\vert_{\cQ_{v,-}} \cong \cO(-1) \oplus \cU_{4,-}^\vee(-1).
\end{equation*}
Composing the arrows at the top row of the diagram, we thus obtain a map 
\begin{equation*}
g_X \colon \P_{\cQ_{v,-}}(\cO(-1) \oplus \cU_{4,-}^\vee(-1)) \cong \cQ_{v,-} \times_{X^\vee} \OGr(4,\rV) \to X.
\end{equation*}
By definition, its fiber over a point $[U_5]$ of $X$ is the intersection of the 4-space $\Gr(4,U_5) \subset X^\vee$ 
with the 6-quadric~$\OGr_-(4,v^\perp/v)$.
The argument of Lemma~\ref{lemma:linear-quadric-intersection} (applied to $X^\vee$ instead of $X$) shows that this intersection is a single point
(unless $v \in U_5$).
Hence the map $g_X$ is birational (and is an isomorphism over the complement of $\cQ_v = \OGr_+(4,v^\perp/v) \subset X$ 
which parameterizes subspaces $U_5$ that contain $v$).

Finally, define the scheme 
\begin{equation*}
E_\cQ := \OGr(3,v^\perp/v) 
\cong \P_{\cQ_{v,-}}(\cU_{4,-}^\vee(-1))
\cong \P_{\cQ_v}(\cU_{4,+}^\vee(-1)).
\end{equation*}
Clearly, it is a subscheme in $\OGr(4,\rV)$ and its projection to $X^\vee$ equals $\cQ_{v,-}$.
Hence it is contained in the fiber product and is a divisor in it.
On the other hand, its projection to $X$ equals $\cQ_v$, and thus we have the following commutative diagram
\begin{equation*}
\xymatrix@C=5em{
E_\cQ \ar@{=}[r] &
\P_{\cQ_v}(\cU_{4,+}^\vee(-1)) \ar[d] \ar@{^{(}->}[r] &
\P_{\cQ_{v,-}}(\cO(-1) \oplus \cU_{4,-}^\vee(-1)) \ar[d]^{g_X}
\\
&
\cQ_v \ar@{^{(}->}[r] &
X
}
\end{equation*}
It is clear that the relative Picard number for the map $g_X$ is equal to 1.
Since $E_\cQ$ is a divisor in~$\P_{\cQ_{v,-}}(\cO(-1) \oplus \cU_{4,-}^\vee(-1))$,
it follows from Lemma~\ref{lemma:blowup} that $g_X$ is the blowup of $\cQ_v \subset X$
and $E_\cQ$ is its exceptional divisor.

According to the above identification the divisor $E_\cQ$ is the zero locus of the natural map 
\begin{equation*}
\cO(-H_X) \to g_X^*(\cO(-1) \oplus \cU_{4,-}^\vee(-1)) \to g_X^*\cO(-1) = \cO(-H_\cQ),
\end{equation*}
where $H_X$ is the hyperplane class of $X$ and $H_\cQ$ is the hyperplane class of $\cQ_{v,-}$.
Therefore $H_\cQ = H_X - E_\cQ$.

Finally, since $H_\cQ = H_X - E_\cQ$, the map $g_\cQ \circ g_X^{-1}$ is given by the complete linear system $|H_X - E_\cQ|$
hence is a linear projection from the quadric $\cQ_v$, and is induced by the linear projection of $\P(\SS)$ 
from the linear span $\P(\cS_{8,v})$.
\end{proof}

\begin{remark}
\label{remark:degenerate-quadrics-through-x}
Consider the quadratic cone 
\begin{equation*}
\tcQ_{v,-} := \Cone_{\P(\cS_{8,v})}(\cQ_{v,-}) \subset \P(\SS)
\end{equation*}
over $\cQ_{v,-}$ with vertex $\P(\cS_{8,v}) \subset \P(\SS)$
(with respect to the linear projection from~\eqref{eq:cs-v-sequence}).
Since the projection of $X$ from $\P(\cS_{8,v})$ is contained in $\cQ_{v,-}$ by Proposition~\ref{proposition:blowup-x-ogr48},
the quadric $\tcQ_{v,-}$ contains $X$.
This is a geometric way to describe quadrics passing through $X$, see Corollary~\ref{corollary:resolution-x}.

Similarly, the quadratic cone 
\begin{equation*}
\tcQ_v := \Cone_{\P(\cS_{8,v,-})}(\cQ_v) \subset \P(\SS^\vee)
\end{equation*}
is a quadric containing $X^\vee$.
%
\end{remark}

\subsection{Blowups of 6-quadrics on smooth hyperplane sections of $X$}

Proposition~\ref{proposition:blowup-x-ogr48} can be applied for a description of a smooth hyperplane section of $X$.
But first, we check that any such hyperplane section contains a 4-space.
Recall the map $\gamma = q_- \circ p_- \colon \P(\SS^\vee) \setminus X^\vee \to \rQ$ defined by~\eqref{diagram:blowups-x-xvee}.
We denote the image of a point $\kappa \in \P(\SS^\vee) \setminus X^\vee$ under this map simply by $\gamma(\kappa) \in \rQ$.

\begin{lemma}
\label{lemma:g6-xkappa-smooth}
Let $\kappa \in \P(\SS^\vee) \setminus X^\vee$ and let $X_\kappa$ be the corresponding smooth hyperplane section of $X$.
Then $X_\kappa$ contains a unique $6$-dimensional quadric, i.e., 
\begin{equation*}
G_6(X_\kappa) \cong \Spec(\Bbbk),
\end{equation*}
and this quadric is nothing but $\cQ_v = \OGr_+(4,v^\perp/v)$, where $v = \gamma(\kappa) \in \rQ$.
\end{lemma}
\begin{proof}
From the description of six-dimensional quadrics in Corollary~\ref{corollary:quadrics-in-xk} it follows that 
the Hilbert scheme~$G_6(X_\kappa)$ equals the zero locus of the global section of the vector bundle 
\begin{equation*}
q_{-*}p_-^*(\cO_{\P(\SS^\vee)}(1)) \cong \cS_8^\vee
\end{equation*}
on $\rQ$ (see diagram~\eqref{diagram:blowups-x-xvee}) that corresponds to $\kappa \in \SS^\vee = H^0(\rQ,\cS_8^\vee)$.
But this zero locus is just the point~$\gamma(\kappa)$ --- this can be explained by an argument that  is completely analogous 
to the argument of Proposition~\ref{proposition:blowup-x}, one should only replace $X$ by $X^\vee$ and $\cS_{8,-}^\vee$ by $\cS_8^\vee$.
\end{proof}

For each $\kappa \in \P(\SS^\vee) \setminus X^\vee$ we set $v = \gamma(\kappa)$.
Note that $\kappa \in q_-^{-1}(v) = \P(\cS_{8,v,-})$.
Define
\begin{equation}
\label{def:cq-kappa-minus}
\cQ_{\kappa,-} = \cQ_{v,-} \cap \P(\kappa^\perp),
\end{equation}
where $\P(\kappa^\perp)$ is the orthogonal in the space $\P(\cS_{8,-,v})$ to the point $\kappa$
with respect to the natural quadratic form of this space.
This is a hyperplane section of the smooth quadric $\cQ_{v,-}$, and as we will see below it is itself smooth.
Recall that by triality the vector bundle $\cU_{4,-}$ on the 6-dimensional quadric $\cQ_{v,-}$ can be identified 
with one of its spinor bundles, and the bundle $\cU_{4,-}^\vee(-1)$ with the other spinor bundle.
The restriction of both bundles to the 5-dimensional quadric $\cQ_{\kappa,-}$ are then identified 
with the unique spinor bundle~$\cS_4$ on it.

Combining Proposition~\ref{proposition:blowup-x-ogr48} and Lemma~\ref{lemma:g6-xkappa-smooth}, we obtain the following result,
that was also mentioned in~\cite[Lemma~1.17]{pasquier2009}.

\begin{corollary}
\label{corollary:blowup-xkappa-smooth}
If $X_\kappa$ is a smooth hyperplane section of $X$ and $\cQ_v \subset X$ is the $6$-dimensional quadric contained in $X_\kappa$ 
then there is an isomorphism $\Bl_{\cQ_v} (X_\kappa) \cong \P_{\cQ_{\kappa,-}}(\cO(-1) \oplus \cS_4)$ and a diagram
\begin{equation}
\label{diagram:blowup-xkappa-ogr48}
\vcenter{\xymatrix{
&&
E'_\cQ \ar[ddll] \ar[d] \ar@{=}[r]^-\sim &
\P_{\cQ_{\kappa,-}}(\cS_4) \ar[d] \ar@{=}[r]^-\sim &
\OFl(1,3;7) \ar@/^5ex/[ddl]
\\
&&
\Bl_{\cQ_v}(X_\kappa) \ar@{=}[r]^-\sim \ar[dl]^{g_X} &
\P_{\cQ_{\kappa,-}}(\cO(-1) \oplus \cS_4) \ar[d]_{g_\cQ} 
\\
\cQ_v \ar[r] &
X_\kappa &&
\cQ_{\kappa,-}
}}
\end{equation} 
Moreover, the quadric $\cQ_{\kappa,-}$ is smooth.
\end{corollary}
\begin{proof}
The map $g_\cQ \circ g_X^{-1} \colon X \dashrightarrow \cQ_{v,-}$
constructed in Proposition~\ref{proposition:blowup-x-ogr48} is a linear projection,
hence the hyperplane $\P(\kappa^\perp) \subset \P(\SS)$ defined by $\kappa \in \P(\SS^\vee)$ is 
the preimage of a hyperplane in the ambient space~$\P(\cS_{8,-,v})$ of the quadric $\cQ_{v,-}$.
To understand which hyperplane it is, recall that, besides~\eqref{eq:cs-v-sequence},
we also have the following exact sequence (this is the second sequence of Lemma~\ref{lemma:spinor-sequence} for $\SS^\vee \cong \SS_-$):
\begin{equation}
\label{eq:s-dual-sequence}
0 \to \cS_{8,v,-} \to \SS^\vee \to \cS_{8,v} \to 0.
\end{equation}
The duality between $\SS$ and $\SS^\vee$ is compatible with these exact sequences, 
i.e., the above sequence is the dual of~\eqref{eq:cs-v-sequence},
and the induced pairing on $\cS_{8,v,-}$ coincides with the one given by the natural quadratic form on it
(whose associated quadric is $\cQ_{v,-}$).
This proves that the hyperplane in $\P(\cS_{8,v,-})$ corresponding to the hyperplane in $\P(\SS)$ defined by $\kappa$
is the orthogonal of $\kappa$ with respect to the natural quadratic form.

Taking the strict transform of the hyperplane section $X_\kappa$ on the left hand side of the diagram~\eqref{diagram:blowup-x-ogr48}
and the preimage of the hyperplane section $\cQ_{\kappa,-}$ of the quadric $\cQ_{v,-}$ on the right hand side, 
we deduce the isomorphism of the corollary and obtain the diagram~\eqref{diagram:blowup-xkappa-ogr48}.

Since $X_\kappa$ and $\cQ_v$ are smooth, the blowup $\Bl_{\cQ_v}(X_\kappa)$ is smooth, 
hence $\cQ_{\kappa,-}$ is smooth by Lemma~\ref{lemma:smoothness-criterion}.
\end{proof}

\begin{corollary}
\label{corollary:f4-x-codim2}
If $X_\kappa \subset X$ is a smooth hyperplane section of $X$ then
\begin{equation*}
F_4(X_\kappa) \cong \cQ_{\kappa,-},
\end{equation*}
and the universal family of $4$-spaces is given by $\P_{\cQ_{\kappa,-}}(\cO(-1) \oplus \cS_4)$.
\end{corollary}
\begin{proof}
Assume $\Pi = \Pi^4_{U_{5,-}}$ is a 4-space lying on $X_\kappa$.
By Lemma~\ref{lemma:linear-quadric-intersection} its intersection with $\cQ_v$ is either a point, or a 3-space.
If it is a point, then the image of $\Pi$ in the smooth five-dimensional quadric~$\cQ_{\kappa,-}$ should be a 3-space, which is of course impossible.
Therefore, the intersection is a 3-space, the image of~$\Pi$ in~$\cQ_{\kappa,-}$ is just a point,
and $\Pi$ is a fiber of the map $g_\cQ \colon \P_{\cQ_{\kappa,-}}(\cO(-1) \oplus \cS_4) \to \cQ_{\kappa,-}$.
\end{proof}

The isomorphism of the corollary gives an alternative proof of Theorem~\ref{theorem:motives-z-xk-06} 
for smooth hyperplane sections of $X$.

\section{Linear sections of codimension 2 and the spinor quadratic line complex}
\label{section:codimension-2}

The situation with linear sections of codimension 2 is more interesting than with hyperplane sections.
We show below that there are two isomorphism classes of smooth codimension 2 linear sections.
We use this result to define an important subvariety of $\Gr(2,\SS^\vee)$ which we call
the spinor quadratic complex.

\subsection{Special linear sections of codimension 2}

Let $K \subset \SS^\vee$ be a subspace of dimension $\dim K = 2$ such that $\P(K) \cap X^\vee = \varnothing$, 
so that $X_K$ is a smooth linear section of $X$ of codimension~2.
The easiest way to distinguish between different isomorphism classes of $X_K$ is by looking at their Hilbert scheme~$F_4(X_K)$.
Recall the map~$\gamma \colon \P(K) \to \rQ$, defined in Corollary~\ref{corollary:stratification-small-k}
(its image is a conic $\gamma(\P(K)) \subset \rQ$), and the 11-dimensional quotient space 
\begin{equation*}
W = \SS/\bw4U_{5,-}
\end{equation*}
associated with an isotropic subspace $U_{5,-} \subset \rV$, 
see~\eqref{eq:w-sequence-3} and the proof of Proposition~\ref{proposition:blowup-x-pi4}.
Recall also the line $L^-_{U_3} \subset X^\vee$ associated with an isotropic subspace $U_3 \subset \rV$, see~\eqref{def:l-u3-minus}.
The birational transformation in the next proposition is again an example of
a special birational transformations of type $(2,1)$ from~\cite{fu2015special}.

\begin{proposition}
\label{proposition:x2-with-p4}
Let $X_K$ be a smooth dimensionally transverse codimension $2$ linear section of $X$.
The following conditions are equivalent:
\begin{enumerate}
\item The Hilbert scheme $F_4(X_K)$ of linear $4$-spaces on $X_K$ is non-empty;
\item The linear span of the conic $\gamma(\P(K)) \subset \rQ$ is contained in $\rQ$;
\end{enumerate}
If both of these conditions hold true and $\Pi^4_{U_{5,-}}$ is a $4$-space on $X_K$, 
then there is an isomorphism $\Bl_{\Pi^4_{U_{5,-}}}(X_K) \cong \Bl_{Z_K}(\P(W_K))$,
and a diagram
\begin{equation}
\vcenter{\xymatrix{
&
E_{\Pi,K} \ar[r] \ar[dl] &
\Bl_{\Pi^4_{U_{5,-}}}(X_K) \ar@{=}[r]^-\sim \ar[dl] &
\Bl_{Z_K}(\P(W_K)) \ar[dr] &
E_{Z_K} \ar[l] \ar[dr]
\\
\Pi^4_{U_{5,-}} \ar[r] &
X_K &&&
\P(W_K) &
Z_K \ar[l]
}}
\end{equation}
where $W_K \subset W$ is a $9$-dimensional subspace corresponding to $K$, 
and $Z_K = \Gr(2,U_{5,-}) \cap \P(W_K)$ is a smooth linear section of the Grassmannian of codimension~$2$.
Moreover, in this case 
\begin{equation}
\label{eq:f4-xk-l-minus}
F_4(X_K) = L^-_{U_{3}}, 
\end{equation} 
where $\P(U_3) \subset \P(\rV)$ is the linear span of the conic $\gamma(\P(K))$.
\end{proposition}

\begin{proof}
For each point $\kappa \in \P(K)$ the hyperplane section~$X_\kappa$ of $X$ is smooth by Lemma~\ref{lemma:smoothness}.
By Corollary~\ref{corollary:f4-x-codim2} the Hilbert scheme $F_4(X_\kappa)$ of 4-spaces in $X_\kappa$ is equal to $\cQ_{\kappa,-}$.
Since 
\begin{equation}
\label{eq:f4x2}
F_4(X_K) = \bigcap_{\kappa \in \P(K)} F_4(X_\kappa) = \bigcap_{\kappa \in \P(K)} \cQ_{\kappa,-}
\end{equation}
and $\cQ_{\kappa,-} \subset \cQ_{v,-} = \OGr_-(4,v^\perp/v)$, where $v = \gamma(\kappa)$,
we see that the condition (1) implies the existence of a common point $[U_{5,-}]$ of all $\OGr_-(4,v^\perp/v)$ for~$v \in \gamma(\P(K))$.
Putting this in different terms, this means that the conic $\gamma(\P(K))$ is contained inside the 4-space $\P(U_{5,-}) \subset \rQ$.
Therefore, the linear span of the conic is contained in $\rQ$.

Conversely, assume that the linear span of the conic $\gamma(\P(K))$ is contained in $\rQ$.
Then it is equal to~$\P(U_3)$ for an isotropic space $U_3 \subset \rV$.
It follows that the intersection of $\OGr_-(4,v^\perp/v)$ for $v = \gamma(\kappa)$ and $\kappa$ running over $\P(K)$
is equal to the line~$L^-_{U_3}$.
This proves that 
\begin{equation*}
F_4(X_K) \subset L^-_{U_3}.
\end{equation*}
Below we show that this is an equality.

For this we consider also the line $L_{U_3} \subset X$.
Take any $[U_{5,-}] \in L^-_{U_3}$ and any $\kappa \in \P(K)$ and set $v = \gamma(\kappa)$. 
We have
\begin{equation*}
L_{U_3} \subset \Pi^4_{U_{5,-}} \cap \cQ_v.
\end{equation*}
Indeed, the inclusion $L_{U_3} \subset \Pi^4_{U_{5,-}}$ follows from $U_3 \subset U_{5,-}$, 
while $L_{U_3} \subset \cQ_v$ follows from~$v \in \P(U_3)$.
By Lemma~\ref{lemma:linear-quadric-intersection} we conclude that $\Pi^4_{U_{5,-}} \cap \cQ_v \cong \P^3$,
hence by Corollary~\ref{corollary:f4-x-codim2} we have $\Pi^4_{U_{5,-}} \subset X_\kappa$.
Since this holds for any $\kappa \in \P(K)$, we conclude that $\Pi^4_{U_{5,-}} \subset X_K$
and since this holds for any $[U_{5,-}] \in L^-_{U_3}$ we conclude that $L^-_{U_3} \subset F_4(X_K)$.
This proves that the conditions~(1) and (2) are equivalent, and that the equality~\eqref{eq:f4-xk-l-minus} holds.

Now assume that both conditions (1) and (2) hold, and let $U_{5,-}$ be a subspace corresponding to a point of~$F_4(X_K)$.
Consider the isomorphism of Proposition~\ref{proposition:blowup-x-pi4}.
Let 
\begin{equation*}
\tX_K \cong \Bl_{\Pi^4_{U_{5,-}}}(X_K)
\end{equation*}
be the strict transform of $X_K$ in the blowup of $X$ along $\Pi^4_{U_{5,-}}$.
Each hyperplane in $\SS$ corresponding to a point of $\P(K) \subset P(\SS^\vee)$ 
contains the 4-space~$\Pi^4_{U_{5,-}}$, hence by~\eqref{eq:picard-relations-blowup-gr-p10}
it corresponds to a hyperplane in the space $\P(W)$ that intersects the linear span of the Grassmannian $\Gr(2,U_{5,-})$ transversely and smoothly.
Therefore, $\tX_K$ is isomorphic to the blowup of a codimension two subspace $\P(W_K) \subset \P(W)$ 
along the corresponding linear section~$Z_K$ of the Grassmannian $\Gr(2,U_{5,-})$.
If the linear section $Z_K$ is not dimensionally transverse, then its preimage in $\tX_K$ is an irreducible component of the latter, which is absurd.
So, since $\tX_K$ is smooth, it follows that~$Z_K$ is smooth as well, see Lemma~\ref{lemma:smoothness-criterion}.
\end{proof}

\begin{definition}
\label{definition:special-section}
A smooth linear section $X_K \subset X$ of the spinor tenfold is called {\sf special} 
if both of the equivalent conditions of Proposition~\ref{proposition:x2-with-p4} hold for $X_K$.
\end{definition}

\begin{remark}
\label{remark:special-x2-unique}
The birational transformation of Proposition~\ref{proposition:x2-with-p4} shows that a special section $X_K$ of $X$
is unique up to an isomorphism (and hence up to a $\Spin(\rV)$-action, see Corollary~\ref{corollary:isomorphism-conjugation}).
Indeed, this follows from the classical fact that a smooth linear section of $\Gr(2,5)$ of codimension 2 is unique up to a projective isomorphism.
\end{remark}

Yet another characterization of special linear sections is the following.

\begin{lemma}
\label{lemma:normal-bundle-special-line}
A smooth linear section $X_K \subset X$ of codimension $2$ is special
if and only if there exists a line $L \subset X_K$ such that 
\begin{equation}
\label{eq:normal-special-line}
\cN_{L/X_K} \cong \cO_L(-2) \oplus \cO_L(1)^{\oplus 6}.
\end{equation} 
Moreover, such line is unique and is equal to the intersection of all $4$-spaces on $X_K$.
\end{lemma}
\begin{proof}
We will use the notation introduced in the proof of Proposition~\ref{proposition:x2-with-p4}.
In particular, let $\P(U_3) \subset \rQ$ be the linear span of the conic $\gamma(\P(K)) \subset \rQ$.

For one direction let us prove that the line $L = L_{U_3}$ is contained in $X_K$ and has the required normal bundle.
For this just note that there is a pencil of 4-spaces passing through~$L$ (which correspond to the pencil of $U_{5,-}$ containing $U_3$), 
and $L$ is the scheme-theoretic intersection of any two distinct 4-spaces $\Pi_1$ and $\Pi_2$ in the pencil.
Therefore, there is a natural embedding of vector bundles
\begin{equation*}
\cN_{L/\Pi_1} \oplus \cN_{L/\Pi_2} \hookrightarrow \cN_{L/X_K}.
\end{equation*}
It remains to note that the left hand side is isomorphic to $\cO_L(1)^{\oplus 6}$, 
while the right hand side is a bundle of rank $8 - 1 = 7$ and degree $6 - 2 = 4$.
Therefore, the cokernel of the above embedding is isomorphic to~$\cO_L(-2)$, hence the required formula for the normal bundle of $L$.

For the opposite direction, assume that $L = L_{U_3} \subset X_K$ is a line with the normal bundle as in~\eqref{eq:normal-special-line}.
Let $\Pi$ be any 4-space on $X$ containing $L$ (such 4-spaces correspond to $U_{5,-}$ containing $U_3$ and hence form a pencil).
If we show that $\Pi$ is contained in $X_K$, then it would follow that $F_4(X_K) \ne \varnothing$, hence $X_K$ is special.
It would also follow that $L$ is contained in the intersection of the pencil of 4-spaces on $X_K$, hence is the only such line on $X_K$.
So, consider the following diagram
\begin{equation*}
\xymatrix{
&&
\cN_{L/\Pi} \ar[d] \ar@{..>}[dl]
\\
0 \ar[r] &
\cN_{L/X_K} \ar[r] &
\cN_{L/X} \ar[r] &
\cN_{X_K/X}\vert_L \ar[r] &
0.
}
\end{equation*}
Its bottom line can be rewritten as
\begin{equation}
\label{eq:normal-bundles-sequence-special}
\xymatrix@1{
0 \ar[r] &
\cO_L(-2) \oplus \cO_L(1)^{\oplus 6} \ar[r] &
\cO_L^{\oplus 3} \oplus \cO_L(1)^{\oplus 6} \ar[r] &
\cO_L(1)^{\oplus 2} \ar[r] &
0
}
\end{equation}
(for the first term we use the assumption of the lemma, and for the second we use~\cite[Lemma~8.1]{Ranestad2000}).
It follows that the first map is an isomorphism on the summands $\cO_L(1)$.
Since moreover $\cN_{L/\Pi} \cong \cO_L(1)^{\oplus 3}$, it follows that the vertical arrow in the diagram factors through a dotted arrow.
Geometrically, this means that the tangent space to $\Pi$ at each point of $L$ is contained in the tangent space to $X_K$.
But since~$\Pi \subset X$ and $X_K$ is a linear section of $X$, it follows that $\Pi \subset X_K$.
\end{proof}

\begin{definition}
The line $L$ on a special linear section $X_K \subset X$ of codimension~$2$ such that~\eqref{eq:normal-special-line} holds
is called its {\sf special line}.
\end{definition}

The characterization of Lemma~\ref{lemma:normal-bundle-special-line} can be reformulated as follows.

\begin{corollary}
\label{corollary:f1-x1}
Let $X_K \subset X$ be a smooth linear section of codimension $2$.
The Hilbert scheme $F_1(X_K)$ of lines on $X_K$ is singular if and only if $X_K$ is special.
Moreover, if $X_K$ is special then the singular locus of $F_1(X_K)$ consists of a unique point and that point corresponds to the special line of $X_K$.
\end{corollary}
\begin{proof}
We have the standard exact sequence
\begin{equation}
\label{eq:normal-sequence-l-x2}
\xymatrix@1{
0 \ar[r] &
\cN_{L/X_K} \ar[r] &
\cO_L^{\oplus 3} \oplus \cO_L(1)^{\oplus 6} \ar[r] &
\cO_L(1)^{\oplus 2} \ar[r] &
0.
}
\end{equation}
Consider the restriction of the right map to the second summands $\cO_L(1)^{\oplus 6} \to \cO_L(1)^{\oplus 2}$.
Clearly, this is a map of constant rank.
If the rank equals 2, then $\cN_{L/X_K} \cong \cO_L^{\oplus 3} \oplus \cO_L(1)^{\oplus 4}$.
If the rank equals 1, then~$\cN_{L/X_K} \cong \cO_L(-1) \oplus \cO_L \oplus \cO_L(1)^{\oplus 5}$.
Finally, if the rank equals 0, then $\cN_{L/X_K} \cong \cO_L(-2) \oplus \cO_L(1)^{\oplus 6}$.
We see that $H^1(L,\cN_{L/X}) \ne 0$, i.e., the Hilbert scheme $F_1(X_K)$ is singular at point $[L]$ if and only if~\eqref{eq:normal-special-line} holds,
hence $X_K$ is special and $L$ is its special line.
\end{proof}

\subsection{Non-special linear sections of codimension 2}

In this subsection we show that there is a unique isomorphism class of smooth linear sections $X_K \subset X$ of codimension 2 
that do not contain a 4-space.

For each subspace $K \subset \SS$ and each $\kappa \in \P(K) \setminus X^\vee$ we denote 
\begin{equation}
\label{def:q-kappa-k}
\cQ_{\kappa,K} := \cQ_v \cap \P(K^\perp),
\end{equation} 
where $v = \gamma(\kappa)$ and $\cQ_v = \OGr_+(4,v^\perp/v)$.
Note that the quadric $\cQ_v$ is contained in the hyperplane~$\P(\kappa^\perp)$ (Lemma~\ref{lemma:g6-xkappa-smooth}), 
hence $\cQ_{\kappa,K}$ is a linear section of $\cQ_v$ of codimension at most $k - 1$, where $k = \dim K$.
Recall also the smooth five-dimensional quadric $\cQ_{\kappa,-}$ defined by~\eqref{def:cq-kappa-minus}.
Finally, recall that $\rc_4(\cS_4) = 0$ (see~\cite[Remark~2.9]{ottaviani}), 
hence 
a general morphism $\cS_4 \to \cO_{\cQ_{\kappa,-}}$ is surjective.
We denote by $\bcS_4$ the kernel of such morphism, so that we have an exact sequence
\begin{equation}
\label{eq:bcs4}
0 \to \bcS_4 \to \cS_4 \to \cO_{\cQ_{\kappa,-}} \to 0.
\end{equation}
This is a rank-3 vector bundle on $\cQ_{\kappa,-}$.
Since the group $\Spin(7)$ acts transitively on the open subset of $\P(\Hom(\cS_4,\cO_{\cQ_{\kappa,-}}))$
corresponding to surjective morphisms (see for more details the proof of Proposition~\ref{proposition:blowup-quadric-x2}), 
this bundle is defined uniquely up to an action of~$\Spin(7)$.

\begin{proposition}
\label{proposition:blowup-quadric-x2}
Let $X_K$ be a smooth dimensionally transverse linear section of $X$ of codimension~$2$.
The blowup $\Bl_{Q_{\kappa,K}}(X_K)$ is a relative hyperplane section 
in the $\P^4$-bundle $\P_{\cQ_{\kappa,-}}(\cO(-1) \oplus \cS_4)$ over $\cQ_{\kappa,-}$.
It is a flat $\P^3$-bundle if and only if $F_4(X_K) = \varnothing$, and in the latter case
\begin{equation*}
\Bl_{Q_{\kappa,K}}(X_K) \cong \P_{\cQ_{\kappa,-}}(\cO(-1) \oplus \bcS_4),
\end{equation*}
where $\bcS_4$ is the rank-$3$ bundle defined by~\eqref{eq:bcs4}.
In particular, such $X_K$ is unique up to an isomorphism.
\end{proposition}

\begin{proof}
Let $\kappa,\kappa' \in \P(K)$ be a basis and set $v = \gamma(\kappa)$, so that $Q_{\kappa,K} = \cQ_v \cap \P({\kappa'}^\perp)$.
Consider the isomorphism of Corollary~\ref{corollary:blowup-xkappa-smooth}.
The strict transform of $X_K = X_\kappa \cap \P({\kappa'}^\perp)$ is then isomorphic to the blowup of $X_K$ along~$Q_{\kappa,K}$.
On the other hand, it is a relative hyperplane section of the projective bundle~$\P_{\cQ_{\kappa,-}}(\cO(-1) \oplus \cS_4)$
corresponding to the composition of the maps 
\begin{equation}
\label{eq:morphism-defining-cf}
\cO(-1) \oplus \cS_4 \hookrightarrow \SS \otimes \cO \xrightarrow{\ \kappa'\ } \cO.
\end{equation}
So, we only have to check that the composition~\eqref{eq:morphism-defining-cf} is surjective if and only if $F_4(X_K) = \varnothing$.
Indeed, if the morphism is not surjective at some point, the fiber $\P^4$ of $\P_{\cQ_{\kappa,-}}(\cO(-1) \oplus \cS_4)$
over this point is contained in $\Bl_{\cQ_{\kappa,K}}(X_K)$, hence gives a 4-space in $X_K$.
Conversely, if $\Pi \subset X_K$ is a 4-space then $\Pi \subset X_\kappa$ and by Corollary~\ref{corollary:f4-x-codim2} we know that $\Pi$ is 
the image of a fiber of $\P_{\cQ_{\kappa,-}}(\cO(-1) \oplus \cS_4)$ over some point of $\cQ_{\kappa,-}$.
Furthermore, this fiber is equal to the strict transform of $\Pi$ in $\Bl_{\cQ_{\kappa,K}}(X_K)$,
hence the morphism~\eqref{eq:morphism-defining-cf} is zero at this point.

Now assume that the composition~\eqref{eq:morphism-defining-cf} is surjective. 
Its component $\cS_4 \to \cO$ is determined by a global section of the bundle $\cS_4^\vee$ on $\cQ_{\kappa,-}$.
The space of such global sections is the 8-dimensional spinor representation of $\Spin(7)$, 
that can be identified with the space $\cS_{8,v}$.
The $\Spin(7)$-action on its projectivization has two orbits, 
the smooth 6-dimensional quadric $\cQ_v \subset \P(\cS_{8,v})$ and its open complement.
It is easy to see that each global section of $\cS_4$ corresponding 
to a point of the closed orbit $\cQ_v$ vanishes on a certain plane $\P^2 \subset \cQ_{\kappa,-}$,
hence any its extension to a morphism~\eqref{eq:morphism-defining-cf} is not surjective.
Thus, if~\eqref{eq:morphism-defining-cf} is surjective then its kernel is an extension of $\cO(-1)$ by $\bcS_4$.
It is easy to check that $\Ext^1(\cO(-1),\bcS_4) = 0$, hence the kernel of~\eqref{eq:morphism-defining-cf} 
is isomorphic to $\cO(-1) \oplus \bcS_4$.

Consequently, $X_K$ is the image of $\P_{\cQ_{\kappa,-}}(\cO(-1) \oplus \bcS_4)$ under the map 
given by the linear system of relative hyperplane sections.
Its uniqueness up to an isomorphism follows from the uniqueness of $\bcS_4$ up to an action of $\Spin(7)$.
\end{proof}

In a combination with Remark~\ref{remark:special-x2-unique} this proves the following 

\begin{corollary}
There are exactly two isomorphism classes of smooth linear sections $X_K \subset X$ of codimension $2$,
described in Proposition~\textup{\ref{proposition:x2-with-p4}} and Proposition~\textup{\ref{proposition:blowup-quadric-x2}} respectively.
\end{corollary}

Using Corollary~\ref{corollary:isomorphism-conjugation}, we can rephrase this in the following form.

\begin{corollary}
The action of the group $\Spin(\rV)$ on the open subset of $\Gr(2,\SS)$ parameterizing lines that do not intersect the spinor tenfold $X$ has exactly two orbits,
one is open and the other is closed.
\end{corollary}

This corollary, together with Lemma~\ref{lemma:secondary-quadric} below provides a refining of~\cite[Proposition~32]{Sato1977}.

By using the description of singular hyperplane sections one can also describe the orbits of $\Spin(\rV)$
on the subset $\Gr(2,\SS)$ parameterizing lines intersecting $X$.
We leave this to the interested reader.

We also prove the following fact about the quadrics $\cQ_{\kappa,K}$ defined by~\eqref{def:q-kappa-k}; it will become useful later.

\begin{corollary}
\label{corollary:q-kappa-K}
If $X_K$ is not special then for any $\kappa \in K$ the quadric $Q_{\kappa,K}$ is smooth.
On a contrary, if $X_K$ is special then for any $\kappa \in K$ the quadric $Q_{\kappa,K}$ is singular.
\end{corollary}
\begin{proof}
The first follows immediately from Lemma~\ref{lemma:smoothness-criterion}, 
since the blowup of $X_K$ along $Q_{\kappa,K}$ is a $\P^3$-bundle over a smooth quadric $\cQ_{\kappa,-}$, hence smooth.

For the second note that the zero locus of the section $\kappa'$ of $\cO(1) \oplus \cS_4^\vee$ is equal to $F_4(X_K) \cong \P^1$,
hence the corresponding relative hyperplane section $\Bl_{\cQ_{\kappa,K}}(X_\kappa)$ inside $\P_{\cQ_{\kappa,-}}(\cO(-1) \oplus \cS_4)$ 
is not smooth, hence the blowup center $Q_{\kappa,K}$ is not smooth as well.
\end{proof}

\subsection{Spinor quadratic line complex}

We denote by $R_0 \subset \Gr(2,\SS^\vee)$ the closed $\Spin(\rV)$-orbit in the open subset of $\Gr(2,\SS^\vee)$,
parameterizing smooth special linear sections of $X$ of codimension~2 and by~$R$ its closure in~$\Gr(2,\SS^\vee)$.

\begin{lemma}
\label{lemma:secondary-quadric}
The subscheme $R \subset \Gr(2,\SS^\vee)$ is a divisor which is cut out by a smooth quadric in~$\P(\bw2\SS^\vee)$.
\end{lemma}

\begin{proof}
Let $\cK_2$ denote the tautological bundle of the Grassmannian~$\Gr(2,\SS^\vee)$.
Consider the composition
\begin{equation}
\label{eq:composition-v-k2}
\rV \otimes \cO \to \Sym^2\SS \otimes \cO \to \Sym^2\cK_2^\vee,
\end{equation}
where the first map is induced by the embedding of $\rV$, considered as the space of quadratic equations of~$X^\vee$ 
(see Corollary~\ref{corollary:resolution-x}), and the second map is tautological.
Then the dual map of~\eqref{eq:composition-v-k2} is the universal version of the map $\gamma \colon \P(K_2) \to \rQ \subset \P(\rV)$
discussed in Corollary~\ref{corollary:stratification-small-k}.

Now consider the composition
\begin{equation*}
\cO \xrightarrow{\ \bq_\rV\ } \Sym^2\rV \otimes \cO \to  \Sym^2(\Sym^2 \cK_2^\vee) \to \Sym^4\cK_2^\vee,
\end{equation*}
where the first map is given by the equation of the quadric $\rQ$, the second is the symmetric square of~\eqref{eq:composition-v-k2}, 
and the last is the multiplication map.
The composition is identically zero, because for general $[K_2] \in \Gr(2,\SS^\vee)$ we have $\P(K_2) \cap X^\vee = \varnothing$,
hence $\gamma(\P(K_2)) \subset \rQ$ by Corollary~\ref{corollary:stratification-small-k}.
Therefore the composition of the first two arrows factors through the kernel of the third, which is nothing but 
\begin{equation*}
\Ker\Big( \Sym^2(\Sym^2 \cK_2^\vee) \to \Sym^4\cK_2^\vee \Big) 
\cong
\Sym^2(\bw2\cK_2^\vee) 
\cong
\cO_{\Gr(2,\SS^\vee)}(2),
\end{equation*}
and thus gives a global section of $\cO_{\Gr(2,\SS^\vee)}(2)$ and determines a quadratic divisor in $\Gr(2,\SS^\vee)$.
Furthermore, for general $[K_2]$ this global section vanishes at $[K_2]$ if and only if 
the composition $\cO \to \Sym^4\cK_2^\vee$ vanishes at $[K_2]$,  
i.e., if and only if the linear span of the conic $\gamma(\P(K_2)) \subset \rQ$ is contained in $\rQ$.
By Proposition~\ref{proposition:x2-with-p4} this is equivalent to speciality of $X_{K_2}$.
Thus, the constructed global section of $\cO_{\Gr(2,\SS^\vee)}(2)$ defines the subscheme $R$.

Since $R \subset \Gr(2,\SS^\vee)$ is $\Spin(\rV)$-invariant, it follows that the quadric in $\P(\bw2\SS^\vee)$
corresponding to the constructed global section of $\cO_{\Gr(2,\SS^\vee)}(2)$ is also $\Spin(\rV)$-invariant.
On the other hand, $\bw2\SS^\vee$ is an irreducible representation of $\Spin(\rV)$ 
(actually, its highest weight is $\omega_3$ and thus it is isomorphic to $\bw3\rV$),
hence any $\Spin(\rV)$-invariant quadratic form on $\P(\bw2\SS^\vee)$ is non-degenerate
and the corresponding quadric is smooth.
\end{proof}

Recall that quadratic divisors in the Grassmannians of lines are traditionally called quadratic line complexes.
The one constructed in the above Lemma is very important for the geometry of linear sections of the spinor threefold.
Because of that we sugeest the following terminology.

\begin{definition}
\label{definition:r}
The quadratic divisor $R \subset \Gr(2,\SS^\vee)$ described in Lemma~\ref{lemma:secondary-quadric} is called the {\sf spinor quadratic line complex}.
\end{definition}

Before going further we discuss some properties of the spinor quadratic line complex $R$.
Recall that $R_0 \subset R$ denotes the open subset parameterizing points $[K_2] \in R$ 
such that $\P(K_2) \cap X^\vee = \varnothing$, i.e., those that correspond to smooth special linear sections of $X$ of codimension~2.

For each point $\kappa \in \P(\SS^\vee)$ there is a natural isomorphism between the projective space $\P(\SS^\vee/\kappa) \cong \P^{14}$
and the closed subvariety of $\Gr(2,\SS^\vee)$ parameterizing lines in $\P(\SS^\vee)$ passing through $\kappa$.
We will use implicitly this isomorphism by considering $\P(\SS^\vee/\kappa)$ as a subvariety of $\Gr(2,\SS^\vee)$.
We denote
\begin{equation}
\label{def:rkappa}
R_\kappa := R \cap \P(\SS^\vee/\kappa) \subset \P(\SS^\vee/\kappa).
\end{equation} 
The next lemma describes these subschemes of $\P(\SS^\vee/\kappa)$.
Recall the quadric $\tcQ_v \subset \P(\SS^\vee)$ described in Remark~\ref{remark:degenerate-quadrics-through-x}
and note that $\kappa$ is its singular point.

\begin{lemma}
Let $\kappa \in \P(\SS^\vee) \setminus X^\vee$.
The subscheme $R_\kappa \subset \P(\SS^\vee/\kappa)$ is the image of the quadric $\tcQ_v$ under the linear projection from $\kappa$,
so that $\tcQ_v = \Cone_\kappa(R_\kappa)$.
In particular, $R_\kappa = \Cone_{\P(\cS_{8,-,v}/\kappa)}(\cQ_v)$ and contains the projection of $X^\vee$ from $\kappa$.
\end{lemma}
\begin{proof}
Let $X_\kappa \subset X$ be the smooth hyperplane section of $X$ associated with $\kappa$ and set $v = \gamma(\kappa)$.
By definition of $R$, the subscheme $R_\kappa$ is the closure of the locus of hyperplanes in $\P(\kappa^\perp) \subset \P(\SS)$
such that the corresponding hyperplane sections of $X_\kappa$ is special.
By Corollary~\ref{corollary:q-kappa-K} this is equivalent to singularity 
of the hyperplane section $\cQ_v \cap \P({\kappa'}^\perp)$ of the smooth quadric $\cQ_v$, see~\eqref{def:q-kappa-k}. 
Thus, $R_\kappa$ is the cone over the projective dual quadric $\cQ_v^\vee$ with the vertex being the orthogonal to the linear span of $\cQ_v$.

Since the linear span of $\cQ_v$ is $\P(\cS_{8,v}) \subset \P(\SS)$, 
its orthogonal in $\P(\SS^\vee)$ is $\P(\cS_{8,v,-})$, and
its orthogonal in $\P(\SS^\vee/\kappa)$ is $\P(\cS_{8,v,-}/\kappa)$.
On the other hand, under the identification of $\P(\cS_{8,v}^\vee)$ with $\P(\cS_{8,v})$ via the natural quadratic form
the quadric $\cQ_v$ is self dual.
Thus, $R_\kappa = \Cone_{\P(\cS_{8,-,v}/\kappa)}(\cQ_v)$.
%
Finally, $R_\kappa$ contains the projection of $X^\vee$, since the cone over $R_\kappa$ with vertex in $\kappa$ is the quadric $\tcQ_v$
that contains $X^\vee$ by Remark~\ref{remark:degenerate-quadrics-through-x}.
\end{proof}

\begin{corollary}
\label{corollary:sing-r}
The spinor quadratic line complex $R \subset \Gr(2,\SS^\vee)$ contains the locus of lines intersecting~$X^\vee$ and its
singular locus is the variety of secant lines of $X^\vee$, 
that is, the image of $\Gr_{\rQ}(2,\cS_{8,-})$ under the natural map $\Gr_{\rQ}(2,\cS_{8,-}) \to \Gr(2,\SS^\vee)$.
In particular, $\codim_R(\Sing(R)) = 7$.
\end{corollary}
\begin{proof}
For the first it is enough to note that for $\kappa \not\in X^\vee$ the quadric $R_\kappa$ contains the projection of $X^\vee$.
For the second consider $\P_R(\cK_2)$ and the natural map $\P_R(\cK_2) \to \P(\SS^\vee)$.
This is a (non-flat) quadratic fibration, whose fibers over the points of the complement of $X^\vee$ 
are the quadrics $R_\kappa \subset \P^{14}$ described above,
and whose fibers over the points of $X^\vee$ are $\P^{14}$ (thus $X^\vee$ is the non-flat locus).
Note that the singular locus of $\P_R(\cK_2)$ is just~$\P_{\Sing(R)}(\cK_2)$.
On the other hand, it definitely contains the $\P^6$-fibration over $\Gr(2,\SS^\vee)$ 
corresponding to singular loci of the quadrics $R_\kappa$, 
whose projection to $R$ is the variety of secant lines, and it is easy to see that there is nothing else in it.
\end{proof}

We finish this section by constructing a nice resolution of singularities of $R$.
For each point $U_3$ of the isotropic Grassmannian $\OGr(3,\rV)$ 
consider the induced filtration of $\SS \otimes \cO$ (see Lemma~\ref{lemma:filtration-ss})
with factors $\cS_2$, $\cS_{2,-} \otimes \cU_3^\vee$, $\cS_2 \otimes \bw2\cU_3^\vee$, and $\cS_{2,-} \otimes \bw3\cU_3^\vee$.
We denote by $\cW \subset \SS \otimes \cO$ the subbundle generated by the first two factors,
and set $\cW_- := (\SS/\cW)^\vee \subset \SS^\vee \otimes \cO$, which is a similar subbundle in the dual space.
Thus we have a bunch of exact sequences
\begin{equation}
\label{eq:cw-cwv}
0 \to \cW \to \SS \otimes \cO \to \cW_-^\vee \to 0,
\qquad 
0 \to \cW_- \to \SS^\vee \otimes \cO \to \cW^\vee \to 0,
\end{equation}
and
\begin{equation}
\label{eq:cw-sequence}
0 \to \cS_2 \to \cW \to \cS_{2,-} \otimes \cU_3^\vee \to 0,
\qquad 
0 \to \cS_{2,-} \to \cW_- \to \cS_2 \otimes \cU_3^\vee \to 0.
\end{equation}

\begin{lemma}
\label{lemma:resolution-r}
The natural map 
\begin{equation*}
\tR := \Gr_{\OGr(3,\rV)}(2,\cW_-) \to \Gr(2,\SS^\vee)
\end{equation*}
is birational onto the hypersurface $R \subset \Gr(2,\SS^\vee)$ and is an isomorphism over the open subset $R_0 \subset R$.
\end{lemma}
\begin{proof}
Choose a point $[U_3] \in \OGr(3,\rV)$ and let $L = L_{U_3} \subset X$ be the corresponding line.
Note that~$L$ is the special line of a special linear section $X_K \subset X$ 
if and only if $X_K$ contains all 4-spaces in~$X$ containing~$L$.
The 4-spaces containing $L$ are parameterized by the line $L_{U_3}^- \subset X^\vee$,
and their linear span in $\SS$ can be identified with the fiber $\cW_{U_3}$ of the subbundle $\cW$ at $[U_3]$.
Indeed, for each point~$s \in \P(\cS_{2,-,U_3}) = L^-_{U_3}$, 
if~$U_{5,-}$ is the corresponding point of $L^-_{U_3}$ then the standard exact sequence
\begin{equation*}
0 \to (U_{5,-}/U_3) \otimes \det(U_3) \to \bw4U_{5,-} \to \det(U_{5,-}/U_3) \otimes \det(U_3) \otimes U_3^\vee  \to 0
\end{equation*}
can be interpreted as the sequence
\begin{equation*}
0 \to \cS_{2,U_3} \to \bw4U_{5,-} \to s \otimes U_3^\vee \to 0,
\end{equation*}
induced from the first sequence of~\eqref{eq:cw-sequence} at point $[U_3]$ 
via the embedding $s \otimes U_3^\vee \subset \cS_{2,-,U_3} \otimes U_3^\vee$.
Since the linear span of the pencil of planes $\P(s \otimes U_3^\vee)$ parameterized by $s \in \P(\cS_{2,-,U_3})$
is the 5-space~$\P(\cS_{2,-,U_3} \otimes U_3^\vee)$, it follows that the span of the pencil of the 4-spaces $\Pi^4_{U_{5,-}}$
is the fiber $\cW_{U_3}$ of the subbundle $\cW \subset \SS \otimes \cO$ at $[U_3]$.

It follows that $X_K$ contains all these 4-spaces if and only if $K$ is contained 
in the orthogonal of the subspace $\cW_{U_3} \subset \SS$, which is nothing but 
the fiber $\cW_{-,U_3}$ of the bundle $\cW_-$ at $[U_3]$.
Thus the relative Grassmannian $\tR$ parameterizes pairs $(L,K)$ such that~$L$ is a line on $X$
and~$K \subset \SS^\vee$ is a two-dimensional subspace such that $X_K$ contains the pencil of 4-spaces on $X$ containing~$L$.
In particular, the image of $\tR$ in $\Gr(2,\SS^\vee)$ contains the open subset~$R_0$ of $R$ and is an isomorphism over it.
Since $\tR$ is irreducible and
\begin{equation*}
\dim \tR =
\dim \OGr(3,10) + \dim \Gr(2,8) = 15 + 12 = 27 = \dim R,
\end{equation*}
it follows that the image is contained in $R$ and the map is birational.
Finally, since $\tR$ is smooth, the map~$\tR \to R$ is a resolution of singularities.
\end{proof}

\begin{remark}
It is easy to check that the intersection of the projectivized fiber $\P(\cW_{-,U_3})$ of the above bundle with $X^\vee$
is the five-dimensional cubic Segre cone 
\begin{equation*}
\Cone_{\P(\cS_{2,-,U_3})}(\P(\cS_{2,U_3}) \times \P(U_3^\vee)).
\end{equation*}
In particular, the preimage of the complement $R \setminus R_0$ is the subvariety of $\tR$
parameterizing 1-secant lines to that cone, hence is a divisor in $\tR$ of relative degree 3 over $\OGr(3,\rV)$.
In particular, the resolution is not small, and is not identical over the smooth locus of $R$.
\end{remark}

We can express the statement of the lemma as a commutative diagram
\begin{equation*}
\xymatrix{
R_0 \ar@{^{(}->}[rr] \ar@{^{(}->}[dr] &&
\tR \ar@{=}[r] \ar[dl] & 
\Gr_{\OGr(3,\rV)}(2,\cW_-) \ar[dr] 
\\
&
R \ar@{-->}[rrr] &&&
\OGr(3,\rV) 
}
\end{equation*}
It defines a rational map (the dashed map in the above diagram)
\begin{equation}
\label{def:lambda}
\lambda \colon R \dashrightarrow \OGr(3,\rV),
\end{equation}
which is regular on the open subset $R_0 \subset R$.

\begin{corollary}
We have an isomorphism $\lambda^*\cO_{\OGr(3,\rV)}(1) \cong \cO_{\Gr(2,\SS^\vee)}(3) \vert_{R_0}$
of line bundles on $R_0$.
\end{corollary}
\begin{proof}
Let us compute the canonical class of $\tR$.
If $H'$ is the hyperplane class of $\Gr(2,\SS^\vee)$ and~$H''$ is the hyperplane class of $\OGr(3,\rV)$, then
$K_{\OGr(3,\rV)} = -6H''$ and $\rc_1(\cW_-) = -2H''$, hence
\begin{equation*}
K_{\tR} = -6H'' + 4H'' - 8H' = -8H' -2H''.
\end{equation*}
On the other hand, $K_R = -14H'$ by adjunction, hence the discrepancy is $6H' - 2H'' = 2(3H' - H'')$.
This shows that the exceptional divisor of the resolution is linearly equivalent to $3H' - H''$
(the discrepancy multiplicity $2$ corresponds to $\codim_R(R \setminus R_0) = 3$).
Since the exceptional divisor does not intersect $R_0$, the restriction of $3H' - H''$ to $R_0$ is equivalent to zero,
whence the required relation.
\end{proof}

\begin{remark}
\label{remark:bundles-on-r}
The map $\lambda$ equips $R_0$ with a bunch of vector bundles.
Besides the tautological rank 2 bundle $\cK_2$ (restriction from $\Gr(2,\SS)$), these are
the pullback of the tautological rank 3 bundle $\cU_3$,
and the pullbacks the spinor bundles $\cS_{2,\pm}$ from $\OGr(3,\rV)$.
\end{remark}

\section{Linear sections of bigger codimension}
\label{section:other}

In this section we discuss some results concerning linear sections of $X$ of higher codimension.

\subsection{The quadratic invariant}

Let $K \subset \SS^\vee$ be a vector subspace of dimension $k \ge 2$ and let $X_K$ be the corresponding linear section of $X$
(in most cases we assume that $\P(K) \cap X^\vee = \varnothing$ so that $X_K$ is smooth, but this is not always necessary).
Define 
\begin{equation}
R_K := R \cap \Gr(2,K) \subset \Gr(2,K),
\end{equation} 
the intersection of the spinor quadratic line complex $R \subset \Gr(2,\SS^\vee)$, see Definition~\ref{definition:r}, 
with the Grassmannian $\Gr(2,K) \subset \Gr(2,\SS^\vee)$. 
The geometric sense of $R_K$ is quite straightforward --- it parameterizes those codimension~2 over-sections of $X_K$
(i.e., subvarieties $X_{K_2} \subset X$ such that $K_2 \subset K$), which are special or singular.

The next lemma shows that $R_K$ is an invariant of the isomorphism class of $X_K$.

\begin{lemma}
\label{lemma:rk-invariant}
If there is an isomorphism $X_{K_1} \cong X_{K_2}$ of dimensionally transverse linear sections of $X$
then $R_{K_1} \cong R_{K_2}$.
\end{lemma}
\begin{proof}
By Corollary~\ref{corollary:isomorphism-conjugation} an isomorphism $X_{K_1} \cong X_{K_2}$ can be realized 
by the action of an appropriate element $g \in \Spin(\rV)$ which takes $K_1$ to $K_2$.
Since $R \subset \Gr(2,\SS^\vee)$ is $\Spin(\rV)$-invariant, we conclude that~$g$ induces an isomorphism of $R_{K_1}$ and $R_{K_2}$.
\end{proof}

\begin{lemma}
\label{lemma:general-rk-smooth}
For a general subspace $K \subset \SS^\vee$ of codimension $2 \le k \le 5$ the subscheme $R_K \subset \Gr(2,K)$ is a smooth divisor.
\end{lemma}
\begin{proof}
Consider the universal family of $R_K$, i.e. the relative Grassmannian 
\begin{equation*}
\Gr_R(k-2,\SS^\vee/\cK_2) \cong R \times_{\Gr(2,\SS^\vee)} \Fl(2,k;\SS^\vee),
\end{equation*}
where $\cK_2$ is the restriction of the tautological vector bundle from $\Gr(2,\SS^\vee)$ to $R$.
Its dimension is equal to~$\dim R + \dim \Gr(k-2,14) = 27 + (k-2)(16 - k)$, and 
by Corollary~\ref{corollary:sing-r} its singular locus has dimension~$20 + (k-2)(16 - k)$.
The image of the singular locus in $\Gr(k,\SS^\vee)$ has codimension
\begin{equation*}
k(16 - k) - 20 - (k-2)(16 - k) = 12 - 2k,
\end{equation*}
which is strictly positive for $k \le 5$.
Therefore, the general fiber of the map $\Gr_R(k-2,\SS^\vee/\cK_2) \to \Gr(k,\SS^\vee)$ is smooth.
\end{proof}

\subsection{Birational constructions and rationality}

The two birational descriptions of the spinor tenfold~$X$ (Proposition~\ref{proposition:blowup-x-pi4} and~\ref{proposition:blowup-x-ogr48})
obtained by projections from a 4-space and a 6-dimensional quadric respectively, can be also applied to linear sections.
The first of them is quite effective for $X_K$ containing a 4-space (see Corollary~\ref{corollary:blowup-xkappa-smooth-pi4}
and Proposition~\ref{proposition:x2-with-p4} above an Proposition~\ref{proposition:x3-with-p4} below), but not as good otherwise.
On a contrary, the second description is quite useful for all linear sections.

Recall the quadric $\cQ_{\kappa,K} \subset \cQ_v$, where $v = \gamma(\kappa)$, defined in~\eqref{def:q-kappa-k}.
The next lemma describes it.

\begin{lemma}
If $X_K$ is a smooth linear section of the spinor tenfold $X$ and $1 \le k \le 5$ the quadric~$\cQ_{\kappa,K}$ has dimension $7 - k$.
It is smooth if and only if 
\begin{equation}
\label{def:r-kappa-k}
R_{\kappa,K} := \P(K/\kappa) \cap R_K \subset \Gr(2,K)
\end{equation} 
is a smooth quadric.
\end{lemma}
\begin{proof}
Set $v = \gamma(\kappa)$.
If the intersection~\eqref{def:q-kappa-k} is not dimensionally transverse, there exists $\kappa' \in K$ distinct from $\kappa$
such that~$\cQ_v \subset \P({\kappa'}^\perp)$.
By Lemma~\ref{lemma:g6-xkappa-smooth} this implies $\gamma(\kappa') = v$.
Therefore, the map $\gamma \colon \P(K) \to \rQ$ is not injective, which is impossible by Corollary~\ref{corollary:stratification-small-k}.
Thus $\dim \cQ_{\kappa,K} = 7 - k$.

When $K_2$ runs over the linear space $\P(K/\kappa)$ of all 2-dimensional subspaces $K_2$ such that \mbox{$\kappa \subset K_2 \subset K$}, 
the quadrics~$\cQ_{\kappa,K_2}$ run over the linear system of hyperplane sections of the smooth quadric $\cQ_v$ containing~$\cQ_{\kappa,K}$.
Furthermore, by Corollary~\ref{corollary:q-kappa-K} a quadric $\cQ_{\kappa,K_2}$ is singular if and only if $[K_2] \in R_{\kappa,K}$.
Thus, $R_{\kappa,K}$ is a linear section of the projective dual quadric $\cQ_v^\vee$ of $\cQ_v$ by the orthogonal subspace of the linear span of~$\cQ_{\kappa,K}$. 
In particular, by an analogue of Lemma~\ref{lemma:smoothness} it is smooth if and only if $\cQ_{\kappa,K}$ is.
\end{proof}

\begin{remark}
The same argument shows that the corank of the quadric $\cQ_{\kappa,K}$ equals the corank of the quadric $R_{\kappa,K}$.
This observation is especially useful when $\dim K = 5$ and $R_K$ is a smooth Gushel--Mukai fivefold, since in this case 
the corank stratification of the family of quadrics $R_{\kappa,K}$ is controlled by the corresponding EPW sextic,
see~\cite[Proposition~4.5]{debarre2015gushel}.
\end{remark}

The next proposition describes the blowup of $\cQ_{\kappa,K}$.
Recall the quadric $\cQ_{\kappa,-}$, see~\eqref{def:cq-kappa-minus}.

\begin{proposition}
\label{proposition:blowup-x345}
Let $X_K$ be a smooth dimensionally transverse linear section of the spinor tenfold~$X$ of codimension $k \le 5$.
There exists a piecewise locally trivial fibration
\begin{equation*}
\Bl_{\cQ_{\kappa,K}}(X_K) \to \cQ_{\kappa,-}
\end{equation*}
whose general fiber is $\P^{5 - k}$ and whose special fibers are projective spaces of bigger dimensions.
\end{proposition}

\begin{proof}
We repeat the argument of Proposition~\ref{proposition:blowup-quadric-x2}.
Consider the isomorphism of Corollary~\ref{corollary:blowup-xkappa-smooth}.
The preimage of the linear section $X_K = X_\kappa \cap \P(K^\perp)$ is isomorphic to the blowup of $X_K$ along the quadric~$\cQ_{\kappa,K}$.
By Corollary~\ref{corollary:blowup-xkappa-smooth} it can be also described as a relative linear section of codimension $k - 1$ 
in the projective bundle $\P_{\cQ_{\kappa,-}}(\cO(-1) \oplus \cS_4)$.
This linear section corresponds to the composition of the maps 
\begin{equation}
\label{eq:morphism}
\cO(-1) \oplus \cS_4 \hookrightarrow (\SS^\vee/\kappa)^\vee \otimes \cO \twoheadrightarrow (K/\kappa)^\vee \otimes \cO.
\end{equation}
and it remains to note that this composition is generically surjective, since otherwise the dimension of the general fiber of the map 
\begin{equation*}
\Bl_{\cQ_{\kappa,K}}(X_K) \hookrightarrow \P_{\cQ_{\kappa,-}}(\cO(-1) \oplus \cS_4) \to \cQ_{\kappa,-}
\end{equation*}
has dimension at least $6 - k$, and hence $\dim(\Bl_{\cQ_{\kappa,K}}(X_K)) \ge 5 + 6 - k = 11 - k$
which contradicts to the dimension of $X_K$ being $10 - k$.
\end{proof}

Using Proposition~\ref{proposition:blowup-x345}, we easily deduce rationality of $X_K$.

\begin{corollary}
\label{corollary:xk-rational}
For any $1 \le k \le 5$ a smooth linear section $X_K$ of codimension~$k$ of the spinor tenfold~$X$ is rational.
\end{corollary}

\subsection{Linear sections of higher codimension with 4-spaces}

In this section we discuss linear sections of $X$ of codimension higher than~2 that contain a 4-space.
We start by discussing the codimension~3 case.
The birational transformation in the next proposition is again an example of
a special birational transformations of type $(2,1)$ from~\cite[Proposition~2.12]{fu2015special}.

\begin{proposition}
\label{proposition:x3-with-p4}
Let $X_K$ be a smooth dimensionally transverse linear section of $X$ of codimension~$3$.
The following conditions are equivalent:
\begin{enumerate}
\item The Hilbert scheme $F_4(X_K)$ of linear $4$-spaces on $X_K$ is non-empty.
\item For any $K_2 \subset K$ the linear section $X_{K_2}$ is special, i.e., $R_K = \Gr(2,K)$.
\item The linear span of the Veronese surface $\gamma(\P(K))$ is a $\P^4$ and is contained in $\rQ$.
\end{enumerate}
If all of these conditions hold true and $\Pi^4_{U_{5,-}}$ is a $4$-space on $X_K$, 
then the span of $\gamma(\P(K))$ is equal to~$\P(U_{5,-}) \subset \rQ$,
there is an isomorphism $\Bl_{\Pi^4_{U_{5,-}}}(X_K) \cong \Bl_{Z_K}(\P(W_K))$,
and a diagram
\begin{equation}
\vcenter{\xymatrix{
&
E_{\Pi,K} \ar[r] \ar[dl] &
\Bl_{\Pi^4_{U_{5,-}}}(X_K) \ar@{=}[r]^-\sim \ar[dl] &
\Bl_{Z_K}(\P(W_K)) \ar[dr] &
E_{Z_K} \ar[l] \ar[dr]
\\
\Pi^4_{U_{5,-}} \ar[r] &
X_K &&&
\P(W_K) &
Z_K \ar[l]
}}
\end{equation}
where $W_K \subset W$ is a subspace of codimension~$3$ corresponding to $K$, 
and $Z_K = \Gr(2,U_{5,-}) \cap \P(W_K)$ is a smooth linear section of the Grassmannian of codimension~$3$.
Moreover, in this case $F_4(X_K) \cong \Spec(\Bbbk)$.
\end{proposition}

\begin{proof}
First, we prove equivalence of the conditions.

$(1) \Rightarrow (2)$
Since $X_K \subset X_{K_2}$, the condition $F_4(X_K) \ne \varnothing$ implies $F_4(X_{K_2}) \ne \varnothing$ for each $K_2 \subset K$, 
hence $X_{K_2}$ is special for each $K_2 \subset K$.

$(2) \Rightarrow (3)$
Assume the linear span of $\gamma(\P(K))$ is a $\P^5$.
Note that the union of the linear spans of the conics $\gamma(\P(K_2))$, when $K_2$ runs over the set of all hyperplanes in $K$,
is the secant variety of $\gamma(\P(K))$, i.e., the symmetric determinantal cubic in $\P^5$.
By assumption it is contained in the quadric $\rQ$.
But then~$\P^5 \subset \rQ$, which is impossible since $\rQ$ is smooth of dimension~8.

Therefore the linear span of $\gamma(\P(K))$ is a $\P^4$. 
In this case the union of the linear spans of the conics is equal to this $\P^4$, 
which then by assumption and Proposition~\ref{proposition:x2-with-p4} is contained in $\rQ$.

$(3) \Rightarrow (1)$
Assume that the linear span of the Veronese surface $\gamma(\P(K))$ is $\P(V_5) \subset \rQ$,
where $V_5 \subset V$ is an isotropic subspace
(so far it is not clear whether it corresponds to a point of $X$ or of $X^\vee$,
but later we will see that the second option holds).
For each subspace $U_4 \subset V_5$ the preimage $\gamma^{-1}(\P(U_4)) \subset \P(K)$ is a conic.
Since the divisor of reducible conics in the space $\P(\Sym^2(K^\vee))$ of all conics in $\P(K)$ is ample,
there exists a subspace $U_4$ such that the conic $\gamma^{-1}(\P(V_4))$ is reducible and reduced, that is
\begin{equation*}
\gamma^{-1}(\P(U_4)) = \P(K'_2) \cup \P(K''_2),
\end{equation*}
where $K'_2, K''_2 \subset K$ are distinct two-dimensional subspaces.
Then we have $\gamma(\P(K'_2)),\gamma(\P(K''_2)) \subset \P(U_4)$.
Let $U'_3,U''_3 \subset U_4$ be the linear spans of the conics $\gamma(K'_2)$ and $\gamma(K''_2)$ respectively
and let $U_4 \subset U_{5,-}$ be the unique extension of the isotropic subspace $U_4 \subset V_5 \subset \rV$ to a subspace 
corresponding to a point of~$X^\vee$.
Then 
\begin{equation*}
[U_{5,-}] \in L^-_{U'_3} \cap L^-_{U''_3} = F_4(X_{K'_2}) \cap F_4(X_{K''_2}) = F_4(X_K)
\end{equation*}
(in the last equality we use the fact that $X_K = X  \cap  \P(K^\perp) = X \cap \P({K'_2}^\perp) \cap \P({K''_2}^\perp) = X_{K'_2} \cap X_{K''_2}$,
hence a 4-space lies on $X_K$ if and only if it lies both on $X_{K'_2}$ and $X_{K''_2}$).
This proves that $F_4(X_K) \ne \varnothing$, completes the proof of the implication 
(3) $\Rightarrow$ (2), and hence of the equivalence of all three conditions.

Now assume that all three conditions of the proposition hold and let $U_{5,-}$ be a subspace corresponding to a point of $F_4(X_K)$.
Then for each subspace $K_2 \subset K$, if the linear span of $\gamma(K_2)$ equals $\P(U_3)$, 
then by Proposition~\ref{proposition:x2-with-p4} we have $[U_{5,-}] \in F_4(X_{K_2}) = L^-_{U_3}$,
hence $U_3 \subset U_{5,-}$.
Thus, the linear span of each conic~$\gamma(\P(K_2))$ is contained in $\P(U_{5,-})$, 
hence the linear span of $\gamma(\P(K))$ is contained in $\P(U_{5,-})$.
Since the Veronese surface cannot be isomorphically projected to a $\P^3$, 
it follows that the linear span of $\gamma(\P(K))$ is equal to $\P(U_{5,-})$
(in particular, the subspace $V_5$ that appeared in the proof of implication $(3) \Rightarrow (1)$ is equal to $U_{5,-}$).
Moreover, this also proves that $F_4(X_K) \cong \Spec(\Bbbk)$.

Finally, consider the isomorphism of Proposition~\ref{proposition:blowup-x-pi4} and let 
\begin{equation*}
\tX_K \cong \Bl_{\Pi^4_{U_{5,-}}}(X_K)
\end{equation*}
be the strict transform of $X_K$ in the blowup of $X$ along $\Pi^4_{U_{5,-}}$.
Each hyperplane in $\SS$ corresponding to a point of $\P(K) \subset P(\SS^\vee)$ 
contains the 4-space~$\Pi^4_{U_{5,-}}$, hence by~\eqref{eq:picard-relations-blowup-gr-p10}
it corresponds to a hyperplane in the space $\P(W)$ that intersects 
the linear span of the Grassmannian $\Gr(2,U_{5,-})$ transversely and smoothly.
Therefore, $\tX_K$ is isomorphic to the blowup of a codimension three subspace $\P(W_K) \subset \P(W)$ (the intersection of those hyperplanes)
along the linear section~$Z_K = \P(W_K) \cap \Gr(2,U_{5,-})$ of the Grassmannian.
If the linear section $Z_K$ is not dimensionally transverse, 
then its preimage in $\tX_K$ is an irreducible component of the latter, which is absurd.
So, since $\tX_K$ is smooth, it follows that~$Z_K$ is smooth as well by Lemma~\ref{lemma:smoothness-criterion}.
\end{proof}

\begin{definition}
\label{definition:very-special-section}
A smooth linear section $X_K \subset X$ of the spinor tenfold is called {\sf very special} 
if all of equivalent conditions Proposition~\ref{proposition:x3-with-p4} hold for $X_K$.
\end{definition}

\begin{remark}
\label{remark:very-special-unique}
Since a smooth linear section of $\Gr(2,5)$ of codimension~3 
is unique up to a projective transformation, it follows that a very special linear section of $X$ is also unique.
\end{remark}

We also check that in the codimension higher than~3 there are no smooth linear sections of $X$ containing a 4-space.

\begin{lemma}
\label{lemma:f4-x4}
If $X_K$ is a smooth linear section of $X$ of codimension~$4$ or higher then $F_4(X_K) = \varnothing$.
\end{lemma}
\begin{proof}
Let $\Pi = \Pi^4_{U_{5,-}}$ be a 4-space on $X$.
Note that $\cN_{\Pi/X} \cong \Omega_\Pi^2(2)$.
So, if $K$ is a 4-dimensional space of hyperplanes containing $\Pi$, there is a morphism on $X$
\begin{equation*}
K \otimes \cO \to I_\Pi(1) \to (I_\Pi/I^2_\Pi)(1) \cong \Omega_\Pi^2(3).
\end{equation*}
Since $\rc_3(\Omega_\Pi^2(3)) = 5$, any such morphism drops rank to 3 on some nonempty subscheme of $\Pi$, 
hence the corresponding linear section $X_K$ is singular along that subscheme.
\end{proof}

Note that $\rc_4(\Omega_\Pi^2(3)) = 0$ which explains the existence of a smooth codimension 3 linear section of $X$ containing $\Pi$.

We also check that in codimension higher than~3 the quadratic line complex $R_K$ is always a hypersurface in $\Gr(2,K)$.

\begin{lemma}
\label{lemma:rk-codim4}
If $X_K$ is a linear section of codimension~$4$, then $R_K \ne \Gr(2,K)$.
Moreover, if $X_K$ is general, there are no very special over-sections $X_{K_3} \subset X$ of codimension~$3$.
\end{lemma}
\begin{proof}
Assume $R_K = \Gr(2,K)$. 
Then for any $K_2 \subset K$ the linear span of the conic $\gamma(\P(K_2))$ is contained in~$\rQ$.
Therefore, the secant variety of $\gamma(\P(K))$ is contained in~$\rQ$.
Further, the secant variety of a 3-dimensional Veronese variety is not contained in a quadric, 
hence the span of $\gamma(\P(K))$ is contained in $\rQ$.
But the dimension of the span is at least~$7$, and $\rQ$ does not contain a linear space of dimension higher than~$4$.

Now assume that $X_K$ is general. 
Then $R_K \subset \Gr(2,K)$ is a smooth quadratic divisor (Lemma~\ref{lemma:general-rk-smooth}).
In particular, $R_K$ contains no planes by Lefschetz theorem.
But if $X_{K_3}$ is very special then the plane~$\Gr(2,K_3)$ is contained in~$R_K$.
\end{proof}


\end{document}